\documentclass[a4paper,english,10pt]{amsart}
\usepackage[applemac]{inputenc}
\usepackage{inputenc}
\usepackage{babel}
\usepackage{amsfonts}
\usepackage{amsmath}
\usepackage{amssymb}
\usepackage{amsthm}
\usepackage{calligra,mathrsfs}
\usepackage{graphicx}
\usepackage[all]{xy}
\usepackage{textcomp}
\usepackage[usenames,dvipsnames]{color}

\usepackage{tikz-cd}
\usepackage{enumitem} 
\usepackage{hyperref}
\hypersetup{
  colorlinks = true,
     urlcolor =   blue,
    linkcolor =  blue,
     citecolor = purple}
     
\newtheorem{thm}{Theorem}[section]  
\newtheorem{cor}[thm]{Corollary}
\newtheorem{lem}[thm]{Lemma}
\newtheorem{defi}[thm]{Definition}
\newtheorem{prop}[thm]{Proposition}
\newtheorem{es}[thm]{Example}
\newtheorem{rem}[thm]{Remark}

\newtheorem{sit}[thm]{Situation}

\DeclareMathOperator{\Spa}{Spa}
\DeclareMathOperator{\id}{id}
\DeclareMathOperator{\Spf}{Spf}

\DeclareMathOperator{\Spec}{Spec}

\DeclareMathOperator{\im}{im}

\DeclareMathOperator{\Homs}{\mathscr{H}\text{\kern -3pt {\calligra\large om}}\,}

\DeclareMathOperator{\sF}{\mathscr{F}}
\newcommand{\sG}{\mathscr{G}}

\DeclareMathOperator{\Inf}{Inf}
\DeclareMathOperator{\Spv}{Spv}

\title{Specialization morphisms}
\author{Ildar Gaisin, John Welliaveetil }

\begin{document}

\maketitle

\begin{abstract}
We define the notion of a specialization morphism from a locally noetherian analytic adic space to a scheme. This captures the (classical) specialization morphism associated to a formal scheme. There is a well behaved theory of compactifications and it turns out that the classical specialization morphism is \emph{proper} in this setup. As an application, we show that the nearby cycles functor commutes with lower shriek in great generality. 
\end{abstract}

 {\hypersetup {linkcolor = black} 
\tableofcontents
}
\pagebreak

\section{Introduction}

In classical algebraic geometry, given a scheme $S$, it is often convenient to pass to the associated reduced 
 scheme 
 $S_{\text{red}}$. The scheme $S_{\text{red}}$ has the same underlying topological space as $S$, and the structure sheaf is obtained by killing nilpotent functions. In this paper we work in the rigid setting and construct an analogous \emph{reduced} space $X_{\text{red}}$ associated to any locally noetherian analytic adic space $X$. The construction of $X_{\text{red}}$ (and the theory we develop) can be extended beyond the locally noetherian case (e.g. perfectoid spaces \cite{perfectoid}), but we make no attempt to do so. Roughly speaking, analogously to the scheme setting, $X_{\text{red}}$ is obtained from $X$ by killing the \emph{topologically} nilpotent functions (cf. Definition \ref{defi:vanisshideal}).

Let $k$ be a nonarchimedean field and $\mathfrak{X}$ be some admissible formal scheme over $k^{\circ}$. Then as constructed by say Huber (cf. \cite[\S 1.9]{hub96}), associated to $\mathfrak{X}$ is an adic generic fiber $\mathfrak{X}_{\eta}$ and a special fiber $\mathfrak{X}_s$ (which is a scheme roughly obtained by killing topologically nilpotent functions). Relating the two spaces $\mathfrak{X}_{\eta}$ and $\mathfrak{X}_s$ is a \emph{specialization} morphism 
\begin{equation} \label{eq:claspemors}
\lambda_{\mathfrak{X}} \colon \mathfrak{X}_{\eta} \to \mathfrak{X}_{s}
\end{equation} 
(cf. Proposition 1.9.1. in loc.cit.) of locally ringed spaces. One difficulty\footnote{For instance before writing this paper it was not even known whether nearby cycles (pushforward of $\lambda_{\mathfrak{X}}$) commutes with the lower shriek functor.} in studying $\lambda_{\mathfrak{X}}$ is that the source ($\mathfrak{X}_{\eta}$) and target ($\mathfrak{X}_s$) are of very different topological nature. Indeed, one is a rigid space and the other a scheme. It turns out that by considering \emph{any} morphism
\begin{equation} \label{eq:enspsmorpsh}
X_{\text{red}} \to S
\end{equation}
of locally ringed spaces, provides sufficient flexibility to formulate what a proper or even a smooth morphism between an adic space and a scheme means. In this paper we study only proper morphisms (cf. Definition \ref{def:propsepmor}). As an application we prove that nearby cycles commutes with the lower shriek functor. Our methods are flexible enough to drop the base field $k$ (cf. Theorem  \ref{vanishing cycles commutes with lower shriek ref}).

Let us give an outline of the paper. In \S \ref{apend:numner1} we define the reduced space $X_{\text{red}}$ and morphisms \eqref{eq:enspsmorpsh}, which we call \emph{specialization} morphisms. The main result in this section, which allows us to construct the pushforward between the  étale sites (of $X$ and $S$) along specialization morphisms is the existence of certain fibre products (cf. Proposition \ref{prop:exifibprodspecmor}). In \S\ref{sec:comfpadssdoce} we develop a theory of  \emph{proper} specialization morphisms and are able to \emph{compactify} specialization morphisms satisyfing some finiteness conditions (cf. Theorem \ref{thm:compspemor}). The proof of compactification follows similar ideas developed by Huber in \cite[Theorem 5.1.5]{hub96}. The short sections \S\ref{section : smooth base change}, \S\ref{sec:ciohomocompsup} record a (smooth) base change result and introduce the lower shriek functor for specialization morphisms, respectively. In \S\ref{sec:propbasersuil}, we establish a proper base change result for specialization morphisms (cf. Theorem \ref{proper base change}). This also recovers a version of \cite[Theorem 3.5.8]{hub96}. Finally in \S\ref{sec:anapplication} we apply the results of \S\ref{apend:numner1}-\S\ref{sec:propbasersuil} together and prove that nearby cycles commutes with lower shriek (cf. Theorem \ref{vanishing cycles commutes with lower shriek ref}).
\\ 

\noindent \textbf{Acknowledgements :} The authors would like to thank enormously the anonymous referee of \cite{ildarjohn} who suggested the idea of constructing specialization morphisms. Indeed the idea arose from attempting to generalize Theorem 4.10 in loc.cit.. We would also like to thank Naoki Imai and Teruhisa Koshikawa for several helpful conversations. We are also grateful to Kavli IPMU and the University of Tokyo for the excellent facilities without which this work would not have been possible.

\section{Specialization morphisms} \label{apend:numner1}

Unless otherwise stated, for what follows $X$ will denote any analytic adic space which is locally noetherian. 
By \cite[Proposition 1.6, Theorem 2.2]{hub93}, given such an $X$, we have a locally ringed space $(X, \mathcal{O}_X^+)$. In particular for every $x \in X$ the stalk $\mathcal{O}_{X,x}^+$ is a local ring with maximal ideal $\mathfrak{m}_x^{+}$. Similarly we denote the maximal ideal of $\mathcal{O}_{X,x}$ by $\mathfrak{m}_x$. 

\begin{defi} \label{defi:vanisshideal}
\emph{Let $\mathfrak{m}_{\mathcal{O}_X^+} \subset \mathcal{O}_X^+$ denote the} vanishing \emph{sheaf of ideals defined as follows. For an open $U \subseteq X$,
\begin{align*}
\mathfrak{m}_{\mathcal{O}_X^+}(U) &:= \left\{ f \in \mathcal{O}_X^+(U) \text{ } \lvert f \text{ vanishes at all points } x \in X \right\}\\
&:= \left\{ f \in \mathcal{O}_X^+(U) \text{ } \lvert \text{ } f_x \in \mathfrak{m}_x^{+} \subset \mathcal{O}_{X,x}^+ \right\},
\end{align*}
where $f_x$ is the image of $f$ in $\mathcal{O}_{X,x}^+$.
We denote by $X_{\text{red}}$ the locally ringed space $(X, \mathcal{O}_X^+ / \mathfrak{m}_{\mathcal{O}_X^+})$ and we call it the reduced adic space associated to $X$.}
\end{defi}

At this point $X_{\text{red}}$ is simply a locally ringed space, but we will shortly prove that it is an object in the category of \emph{adic} locally ringed spaces, $(V)$. We recall the definition.
\begin{defi}
\emph{The objects of the category $(V)$ are triples $(X, \mathcal{O}_X, (\lvert \cdot (x)\lvert)_{x \in X})$, consisting of a locally ringed topological space $(X, \mathcal{O}_X)$, and for each $x \in X$, $\lvert \cdot (x) \lvert$ is an equivalence class of continuous valuations on $\mathcal{O}_{X,x}$. Morphisms are given by morphisms of locally ringed topological spaces and compatible with valuations in the obvious sense.}
\end{defi}

We will also compare $\mathfrak{m}_{\mathcal{O}_X^+}$ with the presheaf of topologically nilpotent elements (cf. Proposition \ref{prop:sheafsameasmshe}).
\begin{defi}
\emph{Let} $\mathcal{O}_X^{\circ \circ} \subset \mathcal{O}_X^+$ \emph{denote the subpresheaf of ideals of $\mathcal{O}_X^+$, given by for an open $U \subseteq X$:
\[
\mathcal{O}_X^{\circ \circ}(U) := \left\{ f \in \mathcal{O}_X^+(U) \text{ } \lvert f \in \mathcal{O}_X^+(U)^{\circ \circ} \right\}
\]}
\end{defi}

To get things off the ground, we begin by proving some expected inclusions.

\begin{lem} \label{lem:prshinve}
The presheaf $\mathcal{O}_X^{\circ \circ}$ is a subpresheaf of $\mathfrak{m}_{\mathcal{O}_X^+}$.  Furthermore, for each $x \in X$, 
\[
\mathfrak{m}_x \subsetneq \mathcal{O}_{X,x}^{\circ \circ} \subseteq \mathfrak{m}_{\mathcal{O}_X^+,x} \subseteq \mathfrak{m}_x^{+} \subseteq \mathcal{O}_{X,x}^+.
\]
\end{lem}

\begin{proof}
Let $U \subseteq X$ be open. If $f \in \mathcal{O}_X^{\circ \circ}(U)$ then $\lvert f(x) \lvert^{n} \to 0$ for all $x \in U$ (by continuity of the valuation $x$). In particular this implies the image of $f$ in $\mathcal{O}_{X,x}$, must lie in $\mathfrak{m}_{x}^+$. Therefore $\mathcal{O}_X^{\circ \circ}$ is a subpresheaf of $\mathfrak{m}_{\mathcal{O}_X^+}$.

All inclusions are now obvious except the first. The inequality is obvious due to the existence of pseudouniformizers in $\mathcal{O}_{X,x}^{\circ \circ}$. Thus it remains to show $\mathfrak{m}_x \subseteq \mathcal{O}_{X,x}^{\circ \circ}$. We can assume $X = \Spa(A,A^+)$, where $A$ is a Tate ring. Let $t$ be any pseudouniformizer in $A$. Then observe that $\mathfrak{m}_x \subset \mathcal{O}_{X,x}$ is uniquely $t$-divisible (as it is an ideal in a ring where $t$ is a unit). Thus $\mathfrak{m}_x = t \cdot \mathfrak{m}_x \subset \mathcal{O}_{X,x}^{\circ \circ}$. 
\end{proof}

A morphism $f \colon X \to Y$ of analytic adic spaces, induces a morphism
\[
(X, \mathcal{O}_X^+ ) \to (Y, \mathcal{O}_Y^+)
\]
of locally ringed spaces\footnote{This is a priori not obvious from the definition of a morphism of adic spaces but comes from the condition on compatibility of valuation rings.} and the latter induces a morphism
\[
f_{\text{red}} \colon X_{\text{red}} \to Y_{\text{red}}
\]
of locally ringed spaces. In this way we obtain:

\begin{lem}
There is a functor
\begin{align*}
\left\{ \mbox{{locally noetherian analytic }} \text{adic spaces} \right\} &\to \left\{ \text{locally ringed spaces} \right\} \\
X &\mapsto X_{\text{red}} := (X, \mathcal{O}_X^+ / \mathfrak{m}_{\mathcal{O}_X^+}) \\
f &\mapsto f_{\text{red}}.
\end{align*}
\end{lem}

We now describe some properties of this functor.

\begin{lem} \label{lem:stalforexref}
Let $x \in X$, the stalks $\mathcal{O}_{X_{\text{red}},x}$ are naturally isomorphic to $k(x)^{+}/k(x)^{\circ \circ}$. In particular $\mathcal{O}_{X_{\text{red}},x}$ are valuation rings.
\end{lem}

\begin{proof}
Let $\mathfrak{m}_x \subset \mathcal{O}_{X,x}$ denote the maximal ideal of $\mathcal{O}_{X,x}$. Then $k(x)^{+} = \mathcal{O}_{X,x}^+ /\mathfrak{m}_x$ and $k(x)^{\circ \circ} = \mathfrak{m}_{\mathcal{O}_X^+,x} /\mathfrak{m}_x$. The first point is by definition. To see the second point, first note that $\mathfrak{m}_{\mathcal{O}_X^+,x} /\mathfrak{m}_x$ is radical and hence a prime ideal in $k(x)^{+}$ (every radical ideal in a valuation ring is prime). Now $k(x)^{\circ \circ}$ is the unique prime ideal of $k(x)^{+}$ of height 1 so it suffices to show 
\[
0 \neq \mathfrak{m}_{\mathcal{O}_X^+,x} /\mathfrak{m}_x \subset k(x)^{\circ \circ}.
\]
The inequality follows from Lemma \ref{lem:prshinve}, so we prove the inclusion. Let $y \in X$ be the unique rank 1 generalization of $x$. Let $\tilde{z} \in \mathfrak{m}_{\mathcal{O}_X^+,x} /\mathfrak{m}_x$ be any element. Then $\tilde{z}$ is the image of some element $z \in \mathfrak{m}_{\mathcal{O}_X^+}(U)$ for some $U \subset X$ open containing both $x$ and $y$. By definition the image of $z$ in $\mathfrak{m}_{\mathcal{O}_X^+,y} /\mathfrak{m}_y$ is topologically nilpotent in $k(y)$. Finally by \cite[Lemma 1.1.10(iii)]{hub96} the natural map $k(x) \to k(y)$ is a homeomorphism onto it's image and so in particular $k(x)^{\circ \circ} = k(y)^{\circ \circ} \cap k(x)$. Therefore $\tilde{z} \in k(x)^{\circ \circ}$.
 We compute
\begin{align*}
\mathcal{O}_{X_{\text{red}},x} &\overset{(i)}{\simeq} \mathcal{O}_{X,x}^+ / \mathfrak{m}_{\mathcal{O}_X^+,x} \\
&\overset{(ii)}{\simeq} (\mathcal{O}_{X,x}^+ /\mathfrak{m}_x ) / (\mathfrak{m}_{\mathcal{O}_X^+,x} /\mathfrak{m}_x) \\
&\overset{(iii)}{\simeq} k(x)^{+}/k(x)^{\circ \circ}
\end{align*}
where (i) is by definition, (ii) follows from the identity $R/M = (R/I)/(M/I)$ and (iii) follows from above.
\end{proof}

At the level of stalks we have the obvious inclusion $\mathfrak{m}_{\mathcal{O}_X^+,x} \subseteq \mathfrak{m}_x^{+}$. At rank 1 points this inclusion becomes an equality.

\begin{cor}
Suppose $x \in X$ corresponds to a rank 1 valuation (i.e. a maximal point). Then $\mathfrak{m}_{\mathcal{O}_X^+,x} = \mathfrak{m}_x^{+}$. 
\end{cor}

\begin{proof}
By Lemma \ref{lem:stalforexref}, we have $\mathcal{O}_{X_{\text{red}},x} = \mathcal{O}_{X,x}^+/\mathfrak{m}_{\mathcal{O}_X^+,x} = k(x)^{\circ}/k(x)^{\circ \circ}$ is a field. Therefore $\mathfrak{m}_{\mathcal{O}_X^+,x}$ must be the maximal ideal in $\mathcal{O}_{X,x}^+$.
\end{proof}

In some sense the sheaf $\mathfrak{m}_{\mathcal{O}_X^+}$ is a \emph{universal} object which contains $\mathfrak{m}_x$ and a pseudouniformizer at all stalks. The following lemma makes this precise.

\begin{lem} \label{lem:topnilelegv}
Let $\mathscr{F}$ be any subpresheaf of ideals of $\mathcal{O}_{X}^+$ which is contained in $\mathfrak{m}_{\mathcal{O}_X^+}$ such that $\mathfrak{m}_x \subsetneq\mathscr{F}_x$ at all points $x \in X$ and $\mathscr{F}_x/\mathfrak{m}_x$ is a prime ideal.
 Then the associated sheaf $\mathscr{F}^{\text{a}} = \mathfrak{m}_{\mathcal{O}_X^+}$.
\end{lem}

\begin{proof}
We have $0 \neq \mathscr{F}_x/\mathfrak{m}_x \subseteq \mathfrak{m}_{\mathcal{O}_X^+,x} /\mathfrak{m}_x \subset \mathcal{O}_{X,x}^+ /\mathfrak{m}_x = k(x)^+$ and by the proof of Lemma \ref{lem:stalforexref}, the quotient $\mathfrak{m}_{\mathcal{O}_X^+,x} /\mathfrak{m}_x$ coincides with the unique prime ideal of $k(x)^+$ of height 1. Therefore the first inclusion  $\mathscr{F}_x/\mathfrak{m}_x \subseteq \mathfrak{m}_{\mathcal{O}_X^+,x} /\mathfrak{m}_x$ is actually an equality. Therefore the inclusion $\mathscr{F} \subseteq \mathfrak{m}_{\mathcal{O}_X^+}$ induces an isomorphism on stalks and one obtains the conclusion as sheafification does not change stalks.
\end{proof}
As a consequence we obtain the following \emph{analytic} characterization of $\mathfrak{m}_{\mathcal{O}_X^+}$.

\begin{prop} \label{prop:sheafsameasmshe}
The sheafification of the presheaf of topologically nilpotent elements $\mathcal{O}_X^{\circ \circ}$ identifies with $\mathfrak{m}_{\mathcal{O}_X^+}$.
\end{prop}

\begin{proof}
Let $x \in X$. 
By Lemma  \ref{lem:prshinve}, $\mathcal{O}_X^{\circ \circ}$ 
is a subpresheaf of $\mathfrak{m}_{\mathcal{O}_X^+}$ and 
\[
\mathfrak{m}_x \subsetneq \mathcal{O}_{X,x}^{\circ \circ}
\]
Observe that for every rational subset $U$ which is a neighbourhood of $x$,
 $\mathcal{O}_{X}^{\circ \circ}(U)$ is a radical ideal.
 
  We deduce from this that 
 $\mathcal{O}_{X,x}^{\circ \circ}$ and $\mathcal{O}_{X,x}^{\circ \circ}/\mathfrak{m}_x$ 
 are radical as well. 
 Since $k(x)^+$ is a valuation ring, an ideal of $k(x)^+$ is radical if and only if it is prime. 
 It follows that $\mathcal{O}_{X,x}^{\circ \circ}/\mathfrak{m}_x$ is prime. 
We can now conclude the proof by using Lemma \ref{lem:topnilelegv}. 
\end{proof}

\begin{rem}
\emph{In general $\mathcal{O}_X^{\circ \circ}$ is not a sheaf and sheafification is necessary in Proposition \ref{prop:sheafsameasmshe}. This is due to the failure of $X$ being quasi-compact. Indeed consider $X = \bigcup_n \Spa(\mathbb{C}_p \left\langle T^n/p \right\rangle, \mathcal{O}_{\mathbb{C}_p}\left\langle T^n/p \right\rangle)$, the open unit $p$-adic disk over $\mathbb{C}_p$ with coordinate in $T$. Then on each affinoid piece $\Spa(\mathbb{C}_p \left\langle T^n/p \right\rangle, \mathcal{O}_{\mathbb{C}_p}\left\langle T^n/p \right\rangle)$, $T$ is topologically nilpotent, but it is not on $X$.}
\end{rem}

For analytic affinoid fields, $X_{\text{red}}$ has a simple description.

\begin{es} \label{es:casaffielspe}
If $X = \Spa(L,L^{+})$ is an analytic affinoid field, then $X_{\text{red}} \simeq \Spec(L^{+}/L^{\circ \circ})$ as locally ringed spaces. The fact that they are isomorphic as topological spaces follows from the fact that there is a (order-reversing) bijection of totally ordered sets
\begin{align*}
\left\{ \text{prime ideals of }L^{+}  \right\} &\simeq \left\{ \text{valuation rings of }L \text{ which contain } L^{+} \right\} \\
\mathfrak{p} &\mapsto L^{+}_{\mathfrak{p}}.
\end{align*}
Indeed there is an isomorphism of sets
\begin{align*}
\Spa(L,L^{+}) &\simeq \left\{ \text{valuation rings of }L \text{ which contain } L^{+} \right\} \backslash \{L\} \\
\Spec(L^{+}/L^{\circ \circ}) &\simeq \left\{ \text{prime ideals of }L^{+}  \right\} \backslash \{ (0)\},
\end{align*}
which can be extended to an isomorphism of topological spaces (after equipping the right hand sides with the obvious topology). Finally $X_{\text{red}}$ and $\Spec(L^{+}/L^{\circ \circ})$ are isomorphic as locally ringed spaces follows from Lemma \ref{lem:stalforexref}.
\end{es}

Inspired by Example \ref{es:casaffielspe}, the following lemma shows that the structure sheaf of $\mathcal{O}_{\widetilde{X}}$ behaves more like the structure sheaf of a scheme than an adic space (cf. \cite[Lemma 1.1.10(iii)]{hub96} for the analogous statement for adic spaces). In this way it is helpful to think of $X_{\text{red}}$ as the \emph{scheme part} of $X$, where roughly one is killing valuation information by quotienting out the topologically nilpotent elements.

\begin{lem}
Let $X$ be an analytic adic space and $x \in X$. There is a one to one correspondence between the generalizations $y$ of $x$ in $X$ and the prime ideals $\mathfrak{p}$ of $\mathcal{O}_{X_{\text{red}},x}$. For corresponding $y$ and $\mathfrak{p}$, the natural ring homomorphism $\mathcal{O}_{X_{\text{red}},x} \to \mathcal{O}_{X_{\text{red}},y}$ is a localization of $\mathcal{O}_{\widetilde{X},x}$ at $\mathfrak{p}$.
\end{lem}

\begin{proof}
Denote $\Spa \kappa(x):= \Spa(k(x),k(x)^{+})$. The natural morphism $\Spa \kappa(x) \to X$ gives a homeomorphism from $\Spa \kappa(x)$ to the set of all generalizations of $x$ in $X$ (cf. \cite[(1.1.9)]{hub96}). By Lemma \ref{lem:stalforexref} and Example \ref{es:casaffielspe}, $(\Spa \kappa(x))_{\text{red}} \simeq \Spec (\mathcal{O}_{X_{\text{red}},x})$ as locally ringed spaces. This proves the first part. 

For the second part, note that the morphism $\mathcal{O}_{X_{\text{red}},x} \to \mathcal{O}_{X_{\text{red}},y}$ is induced by the morphism $k(x)^{+} \to k(y)^{+}$, which by the (order-reversing) bijection of Example \ref{es:casaffielspe} is just localization at the corresponding prime ideal.  
\end{proof}




In general $\mathcal{O}_{X,x}^{+}$ may not be henselian along its maximal ideal. However the henselian property is satisfied at stalks of the vanishing sheaf of ideals.
\begin{lem} \label{lem:henseproofinach}
For each $x \in X$, the pair $(\mathcal{O}_{X,x}^{+}, \mathfrak{m}_{\mathcal{O}_X^+,x})$ is henselian.
\end{lem}

\begin{proof}
Since henselian pairs are preserved under filtered colimits of such pairs (cf. \cite[Tag 0FWT]{stacks-project}), by Proposition \ref{prop:sheafsameasmshe}, it suffices to show that the pair $(A^+, A^{\circ \circ})$ is henselian for a complete Tate ring $A$. This is the content of \cite[Lemma 7.2.3(5)]{bhatt_notes}.
\end{proof}

One deduces, somewhat surprisingly, that in particular $k(x)^{\circ}$ is henselian along it's maximal ideal, at rank 1 points $x \in X$.

\begin{cor} \label{residue plus is henselian}
For each $x \in X$, the pair $(k(x)^+, k(x)^{\circ \circ})$ is henselian.
\end{cor}

\begin{proof}
By passing to the quotient, Lemma \ref{lem:henseproofinach} shows that the pair 
\[
(\mathcal{O}_{X,x}^{+}/\mathfrak{m}_x, \mathfrak{m}_{\mathcal{O}_X^+,x}/\mathfrak{m}_x) = (k(x)^+, k(x)^{\circ \circ}).
\]
is henselian (we refer the reader to the proof of Lemma \ref{lem:stalforexref} for this identification).
\end{proof}

The next lemma shows that we do not lose cohomological information of $\mathcal{O}_X^{+}$ by passing to $\mathcal{O}_{X_{\text{red}}}$.

\begin{lem} \label{lem:gentophensspa}
Let $X$ be a topological space and $\sF$ a sheaf of rings on $X$. Suppose $\mathscr{I} \subseteq \mathscr{F}$ is a subpresheaf of ideals of $\mathscr{F}$, such that for each $x \in X$ and for each open subset $U \subseteq X$ containing $x$, there exists an open subset $U' \subseteq U$ containing $x$ such that the pair $(\mathscr{F}(U'), \mathscr{I}(U'))$ is henselian. Then for every open $V \subseteq X$, the pair $(\mathscr{F}(V), \mathscr{I}^{a}(V))$ is henselian.
\end{lem}

\begin{proof}
We can assume without loss of generality that $V = X$. We begin by showing that $\mathscr{I}^{a}(X)$ is contained in the Jacobson radical of $\mathscr{F}(X)$. By \cite[Tag 0AME]{stacks-project} it suffices to show that $1+i$ is a unit in $\mathscr{F}(X)$ for every $i \in \mathscr{I}^{a}(X)$. Take a covering $\{X_i \}$ of $X$ such that $i \in \mathscr{I}(X_i) \subseteq \mathscr{I}^{a}(X_i)$ and the pair $(\mathscr{F}(X_i), \mathscr{I}(X_i))$ is henselian. Then $1+i$ is a unit in $\sF(X_i)$. Therefore $1+i$ is a unit in $\sF(X)$.

To complete the proof we verify the criterion in \cite[Tag 09XI]{stacks-project}. Let $f(T) \in \mathscr{F}(X)[T]$ be a monic polynomial of the form
\[
f(T) = T^{n}(T-1)+ a_nT^{n}+ \cdots + a_1T +a_0
\]
with $a_n, \ldots, a_0 \in \mathscr{I}^{a}(X)$ with $n \geq 1$. Take a covering $\{X_i \}$ of $X$ such that $a_j \in \mathscr{I}(X_i) \subseteq \mathscr{I}^{a}(X_i)$ for all $0 \leq j \leq n$ and the pair $(\mathscr{F}(X_i), \mathscr{I}(X_i))$ is henselian. Then there exists a unique $a_{X_i} \in \mathscr{F}(X_i)$ with $f(a_{X_i}) = 0$ and $a_{X_i} \in 1+ \mathscr{I}(X_i)$. The $a_{X_i}$ now glue to give a unique $a_X \in 1+ \mathscr{I}^{a}(X)$ with $f(a_X) = 0$.
\end{proof}

\begin{cor} \label{cor:henspairlif}
The pair $(\mathcal{O}_X^{+}(X), \mathfrak{m}_{\mathcal{O}_X^+}(X))$ is henselian.
\end{cor}

\begin{proof}
This is a consequence of Proposition \ref{prop:sheafsameasmshe}, Lemma \ref{lem:gentophensspa} and \cite[Lemma 7.2.3(5)]{bhatt_notes}.
\end{proof}

Let
 $\mathfrak{X}$ be a type (S) formal scheme
 (cf. \cite[\S 1.9]{hub96}).
 Recall that we have a scheme 
 $\mathfrak{X}_{s}$ associated to $\mathfrak{X}$ which is defined as follows. 
 Let $I \subseteq \mathcal{O}_X$ 
 be the sheaf of ideals with $I(U) := \{ f \in \mathcal{O}_X(U) \text{ } | \text{ } f(x) = 0 \mbox{ for every } x \in U\}$ for every open subset $U \subseteq X$. 
 Then 
 $\mathfrak{X}_{s}:= (\mathfrak{X},\mathcal{O}_\mathfrak{X}/I)$ is a reduced scheme.
 Furthermore, as in \cite[Proposition 1.9.1]{hub96}, 
 we have an analytic adic space $\mathfrak{X}_\eta$ (denoted by $d(\mathfrak{X})$ in loc.cit.) and 
 a morphism of locally 
 ringed spaces (nearby cycles)
 \begin{equation} \label{eq:thepsisneacu}
 \lambda_\mathfrak{X} \colon (\mathfrak{X}_{\eta}, \mathcal{O}^{+}_{\mathfrak{X}_{\eta}}) \to (\mathfrak{X}, \mathcal{O}_{\mathfrak{X}}).
 \end{equation}
 After killing functions which vanish at all points, $\lambda_\mathfrak{X}$ induces a morphism of locally ringed spaces 
 \begin{equation} \label{eq:nearbycuspe}
 \mathfrak{X}_{\eta, \text{red}} \to \mathfrak{X}_{s},
 \end{equation}
 which we continue to denote by $\lambda_{\mathfrak{X}}$.
 In this way an analytic adic space \emph{specializes} to a scheme.
In an attempt to enlarge the class of
specialization morphisms, this observation motivates the following definition.

\begin{defi} [Specialization morphism] \label{defi:spemorrepov}
\emph{Given an analytic adic space $X$ and a scheme $S$, we call a morphism of locally ringed spaces
\[
X_{\text{red}} \to S
\]
a}
 specialization morphism \emph{and denote it by $X \to S$.}
\end{defi}


\begin{rem} \label{rem:compmoretad}
\emph{Specialization morphisms from analytic adic spaces to schemes can be composed
 with morphisms of adic spaces and morphisms of schemes. More precisely
  if $\alpha \colon X \to S$ is a specialization morphism, $f \colon Y \to X$ a morphism of adic spaces and $g \colon S \to T$ a morphism of schemes, then we get specialization morphisms $\alpha \circ f := \alpha \circ f_{\text{red}} \colon Y \to S$ and $g \circ \alpha \colon X \to T$.}
\end{rem}

We now show the existence of fibre products in certain situations.

\begin{prop} \label{prop:exifibprodspecmor}
Let $S$ and $T$ be schemes and $X$ be an analytic adic space.
 Let $\alpha \colon X \to S$ be a specialization morphism and 
 $g \colon T \to S$ an étale morphism of schemes. Then there exists an analytic adic space $Y$, a 
 morphism $f \colon Y \to X$ of analytic adic spaces and a specialization morphism $\beta \colon Y \to T$ such that
$$
\begin{tikzcd} [row sep = large, column sep = large] 
Y \arrow[r, "\beta"] \arrow[d, "f"] &
T \arrow[d, "g"] \\
X \arrow[r, "\alpha"] &
S  
\end{tikzcd}
$$
commutes in the sense of Remark \ref{rem:compmoretad}. Moreover the following universal property is satisfied. For every analytic 
adic space $Y'$, every morphism $f' \colon Y' \to X$ of analytic adic 
spaces and every specialization morphism $\beta' \colon Y' \to T$ such
 that $\alpha \circ f' = g \circ \beta'$, there exists a unique morphism $h \colon Y' \to Y$ of analytic adic
  spaces with $f' = f \circ h$ and $\beta' = \beta \circ h$. We denote $X \times_S T := Y$. Moreover
   the projection $f \colon X \times_S T \to X$ is an étale morphism of analytic adic spaces.
\end{prop}
\begin{proof}

We may assume that $S$ is an affine scheme and that $X$ is an affinoid adic space 
 of the form $\mathrm{Spa}(A,A^+)$ where $A$ is complete. 
 Since $X$ is analytic and as it suffices to prove the proposition locally for $X$, we 
 can suppose without loss of generality that $A$ is Tate. 
 We fix a topologically nilpotent unit $s \in A$ which belongs to $A^+$. 
 Furthermore, we assume that $T$ is affine.
 It follows that $T$ must be of the form $\mathrm{Spec}(E)$ where 
  $E$ is an \'etale algebra over $D := \mathcal{O}_S(S)$. 
   By \cite[Tag 00U9]{stacks-project}, there exists a presentation
$E = D[x_1, \ldots, x_n]/(f_1, \ldots f_n)$
such that the image of $\det(\partial f_j/\partial x_i)$ is invertible in $E$.
  Since $\alpha \colon X_{\text{red}} \to S$ is a morphism 
 of locally ringed spaces, we get a map 
 $\mathcal{O}_S(S) \to  (\mathcal{O}_X^+ / \mathfrak{m}_{\mathcal{O}_X^+})(X)$.
 It follows that if $C := (\mathcal{O}_X^+ / \mathfrak{m}_{\mathcal{O}_X^+})(X)$ then by base change,
 the morphism $$C \to C' := C \otimes_D E$$ is an étale morphism of 
 algebras. We have that 
$C' = C[x_1, \ldots, x_n]/(f_1, \ldots f_n)$
such that the image of $\det(\partial f_j/\partial x_i)$ is invertible in $C'$ 
where we abuse notation and use $f_i$ to denote the \emph{image} in 
$C[x_1,\ldots,x_n]$ 
of the polynomial $f_i \in D[x_1,\ldots,x_n]$ via 
the map $D[x_1,\ldots,x_n] \to C[x_1,\ldots,x_n]$ 
induced by the map $D \to C$. 

 The quotient map 
    $\mathcal{O}_X^+ \to \mathcal{O}_X^+/\mathfrak{m}_{\mathcal{O}_X^+}$
    induces a map $A^+ \to C$. This map is not necessarily surjective. However, since
      $\mathcal{O}_X^+ \to \mathcal{O}_X^+/\mathfrak{m}_{\mathcal{O}_X^+}$ is surjective as a morphism of sheaves 
      and because it suffices to construct the adic space $Y$ from the statement of the proposition locally over $X$, we 
      can suppose that the image of $A^+$ in $C$ contains the coefficients of the 
      polynomials $f_i$. For every $1 \leq i \leq n$, let $g_i$ be a lift of $f_i$ along $A^+ \to C$. 
      
      Let 
\begin{equation} \label{eq:refpoviequs}
Y := \Spa((A,A^{+})\langle X_1,\ldots,X_n\rangle / (g_1,\ldots,g_n)) =: \Spa(B,B^{+})
\end{equation}
where we refer the reader to \cite[\S 3]{hub93} and \cite[(1.4.1)]{hub96} for the definition of the quotient appearing in \eqref{eq:refpoviequs}. In particular
\begin{equation} \label{eq:formforB}
B = A\langle X_1,\ldots,X_n\rangle / (g_1,\ldots,g_n)
\end{equation}
and 
 \begin{equation} \label{eq:formforB+}
B^{+} = (A^{+}\langle X_1,\ldots,X_n\rangle^{c} / ((g_1,\ldots,g_n)A\langle X_1,\ldots,X_n\rangle \cap A^{+}\langle X_1,\ldots,X_n\rangle^{c})^{c} 
\end{equation}
where $A^{+}\langle X_1,\ldots,X_n\rangle^{c}$ is the integral closure of $A^{+}\langle X_1,\ldots,X_n\rangle$ in $A\langle X_1,\ldots,X_n\rangle$ and $B^{+}$ is the integral closure of 
\[
A^{+}\langle X_1,\ldots,X_n\rangle^{c} / ((g_1,\ldots,g_n)A\langle X_1,\ldots,X_n\rangle \cap A^{+}\langle X_1,\ldots,X_n\rangle^{c})
\]
in $B$.

      Let $B_0 := A ^+\langle X_1,\ldots,X_n\rangle / (g_1,\ldots,g_n)$ and 
      observe that we have a well defined map $B_0 \to B^+$. 
       %
The construction thus far fits into the following 
commutative\footnote{Since for any $h \in A^+ \langle X_1,\ldots,X_n \rangle$, almost all the coefficients of $h$ are topologically nilpotent in $A$, it follows that there is a well defined map $A^+ \langle X_1,\ldots,X_n \rangle/(g_i)_i \to \mathcal{O}_{X_{\text{red}}}(X) \otimes_D E$.} diagram: 
           
 \begin{equation} \label{d2}
\begin{tikzcd} [row sep = large, column sep = small] 
A \langle X_1,\ldots,X_n \rangle/ (g_i)_i & A^+ \langle X_1,\ldots,X_n \rangle/(g_i)_i \arrow[l] \arrow[r] & \mathcal{O}_{X_{\text{red}}}(X) \otimes_D E & E \arrow[l] \\
A \arrow[u] & A^+ \arrow[l] \arrow[r] \arrow[u] & \mathcal{O}_{X_{\text{red}}}(X) \arrow[u] & D. \arrow[u] \arrow[l] 
\end{tikzcd}
\end{equation}
  
           Let $f \colon Y\to X$ be the map induced by the adic morphism
      $A \to B$.  
      We now construct the map $\beta \colon Y_{\text{red}} \to T$. By \cite[Tag 01I1]{stacks-project}, it suffices 
   to construct a ring homomorphism $E \to \mathcal{O}_{Y_{\text{red}}}(Y)$ and to show that the following diagram is 
   commutative:
  
    \begin{equation} \label{d3}
\begin{tikzcd} [row sep = large, column sep = large] 
E \arrow[r]  & \mathcal{O}_{Y_{\text{red}}}(Y)  \\
D \arrow[u] \arrow[r] & \mathcal{O}_{X_{\text{red}}}(X). \arrow[u] 
\end{tikzcd}
\end{equation}

   The following commutative diagram
    \begin{equation} \label{d1}
\begin{tikzcd} [row sep = large, column sep = large]
|[alias = Z]|  B_0 \arrow[r] & B^+ \arrow[r]  & \mathcal{O}_{Y_{\text{red}}}(Y) \arrow[r] & \mathcal{O}_{Y_{\text{red}},y} \\
& A^+ \arrow[to = Z] \arrow[u] \arrow[r] & \mathcal{O}_{X_{\text{red}}}(X) \arrow[r] \arrow[u] & \mathcal{O}_{X_{\text{red}},x}. \arrow[u]
\end{tikzcd}
\end{equation}   
   implies that the map $\mathcal{O}_{X_{\text{red}}}(X) \to \mathcal{O}_{Y_{\text{red}}}(Y)$ factors 
   through a map 
   \begin{equation} \label{simple but imp}
   B_0 \otimes_{A^+} \mathcal{O}_{X_{\text{red}}}(X) \to  \mathcal{O}_{Y_{\text{red}}}(Y).
   \end{equation} 
   By Lemma \ref{equality lemma}, this implies that we have a map 
   $\mathcal{O}_{X_{\text{red}}}(X) \otimes_D E \to \mathcal{O}_{Y_{\text{red}}}(Y)$. 
   Let $\beta$ denote the composition $E \to \mathcal{O}_{X_{\text{red}}}(X) \otimes_D E \to \mathcal{O}_{Y_{\text{red}}}(Y)$.
   Observe that with this choice of $\beta$ the diagram (\ref{d3}) commutes. 

\begin{lem} \label{equality lemma}
 We have that
 $$B_0 \otimes_{A^+} \mathcal{O}_{X_{\text{red}}}(X) = \mathcal{O}_{X_{\text{red}}}(X) [X_1,\ldots,X_n]/(f_i)_i = \mathcal{O}_{X_{\text{red}}}(X) \otimes_D E.$$
\end{lem} 
\begin{proof} 
Recall by construction, $\mathcal{O}_{X_{\text{red}}}(X) [X_1,\ldots,X_n]/(f_i)_i = \mathcal{O}_{X_{\text{red}}}(X) \otimes_D E$.
Hence, it suffices to show
$$B_0 \otimes_{A^+} \mathcal{O}_{X_{\text{red}}}(X) = \mathcal{O}_{X_{\text{red}}}(X) [X_1,\ldots,X_n]/(f_i)_i.$$
Recall by definition, $B_0 = A^+ \langle X_1,\ldots,X_n \rangle/(g_i)$.
 We see that the image of $s$ in 
 $\mathcal{O}_{X_{\text{red}}}(X)$ is zero. It follows that 
 $$B_0 \otimes_{A^+} \mathcal{O}_{X_{\text{red}}}(X) = (B_0/sB_0) \otimes_{(A^+/sA^+)} \mathcal{O}_{X_{\text{red}}}(X).$$
  
  Clearly, $B_0/sB_0 = (A^+/sA^+)[X_1,\ldots,X_n]/ (\widetilde{g}_i)$ where 
   $\widetilde{g_i}$ is obtained from $g_i \in A^+[X_1,\ldots,X_n]$ by 
  reducing each of its coefficients modulo $sA^+$.   
  This follows by first observing that 
  $(A^+ \langle X_1,\ldots,X_n \rangle)/(sA^+ \langle X_1,\ldots,X_n \rangle) = (A^+/sA^+)[X_1,\ldots,X_n]$
  which is a consequence of the fact that
   any element of $A^+ \langle X_1,\ldots,X_n \rangle$ has all but finitely many coefficients
  not belonging to $sA^+$. 
  We hence get the equality $B_0/sB_0 = (A^+/sA^+)[X_1,\ldots,X_n]/ (\widetilde{g}_i)$ by observing 
  that 
  if $R = A^+\langle X_1,\ldots,X_n \rangle$ then
   $B_0/sB_0 = (R/sR)/((sR + IR)/sR)$ where $I$ is the ideal generated by the $g_i$. 
  Now observe that the ideal $sR+IR$ in $R/sR$ is generated by the elements $\widetilde{g_i}$. 
  
  Hence 
  $$B_0 \otimes_{A^+} \mathcal{O}_{X_{\text{red}}}(X) = \mathcal{O}_{X_{\text{red}}}(X)[X_1,\ldots,X_n]/(\widetilde{g_i}).$$ 
  By construction, the image of $g_i$, and hence $\widetilde{g_i}$, in $\mathcal{O}_{X_{\text{red}}}(X)[X_1,\ldots,X_n]$ coincides
  with $f_i$. This completes the proof. 
  \end{proof} 
      
       We claim that $Y$ is étale over $X$ or 
      equivalently that the determinant of the Jacobian
      $u := \det(\partial g_j/\partial X_i)$ is invertible in $B$ (cf. \cite[Proposition 1.7.1]{hub96} ). 
       Observe that $u \in B _0$. 
       By construction, 
    the image of $u$ in 
 $C' = \mathcal{O}_{X_{\text{red}}}(X) [X_1,\ldots,X_n] /(f_i)_i$ is invertible. 
 By (\ref{simple but imp}), Lemma \ref{equality lemma} and diagram (\ref{d1}) 
 the map $B_0 \to \mathcal{O}_{Y_{\text{red}}}(Y)$ factors through a 
 map $C' \to \mathcal{O}_{Y_{\text{red}}}(Y)$. It follows that the image of 
 $u$ in $\mathcal{O}_{Y_{\text{red}}}(Y)$ is invertible and we deduce from this that 
 $u$ is invertible in $\mathcal{O}_{Y_{\text{red}},y}$ for every $y \in Y$. 
 For any $y \in Y$, 
 since the map $\mathcal{O}_{Y,y}^+ \to \mathcal{O}_{Y_{\text{red}},y}$ is a local morphism of rings, 
       we must have that $u$ is a unit in $\mathcal{O}_{Y,y}^+$ as well.
       Hence we conclude that $u$ is a unit in $\mathcal{O}^+_{Y}(Y) = B^+$.

      We now prove the universal property. This is essentially a variant of Hensel's Lemma. Let $Y'$ be an analytic adic space as in the statement of the proposition. We are provided a morphism of adic spaces
      $f' \colon Y' \to X$ and a specialization morphism 
      $\beta' \colon Y' \to T$ such that 
      $\alpha \circ f' = g \circ \beta'$ or equivalently that the 
      following diagram commutes:
 $$      
\begin{tikzcd} [row sep = large, column sep = large] 
{Y'} \arrow[r,"\beta'"] \arrow[d,"f'"] & T \arrow[d,"g"] \\
{X} \arrow[r,"\alpha"] & S.
\end{tikzcd}
$$
     We deduce that the following diagram commutes: 
         $$
\begin{tikzcd} [row sep = large, column sep = large] 
\mathcal{O}_{Y'_{\text{red}}}(Y')  & E \arrow[l] \\
\mathcal{O}_{X_{\text{red}}}(X) \arrow[u] & D. \arrow[u] \arrow[l] 
\end{tikzcd}
$$
      Hence we have a unique map 
      $\mathcal{O}_{X_{\text{red}}}(X) \otimes_D E \to \mathcal{O}_{Y'_{\text{red}}}(Y')$. 

  By diagram (\ref{d1}), we have the following commutative diagram:
      
       \begin{equation} \label{d5}
\begin{tikzcd} [row sep = large, column sep = large] 
                                                               &|[alias = A0]| \mathcal{O}^+_{Y'}(Y') \arrow[r] &  |[alias = A1]| \mathcal{O}_{Y'_{\text{red}}}(Y') \\ 
 A^+ \langle X_1,\ldots,X_n \rangle/(g_i)_i \arrow[r]  & \mathcal{O}_{X_{\text{red}}}(X) \otimes_D E \arrow[to = A1]& E \arrow[l] \\
 A^+ \arrow[to = A0]  \arrow[r] \arrow[u] & \mathcal{O}_{X_{\text{red}}}(X) \arrow[u] & D. \arrow[u] \arrow[l] 
\end{tikzcd}
\end{equation}

    To prove the universal property, it suffices to construct a unique map   
    $$A^+ \langle X_1,\ldots,X_n \rangle/(g_i)_i \to\mathcal{O}_{Y'}^+(Y')$$ such that 
    with the addition of this morphism, the diagram above remains commutative. 
    At this point, we simplify the situation and assume that $n = 2$, $g_1 \in A^+[X_1]$
    and $g_2$ is of the form $X_2\cdot h - 1$ where $h \in A^+[X_1]$. 
    Note that such a simplification is possible by assuming that the 
    morphism $D \to E$ is a standard étale extension.
    We can make this assumption on the morphism $D \to E$ 
    since by \cite[Tag 02GT]{stacks-project}, every étale morphism is locally standard étale
    and it suffices to prove our lemma on a cover of $S$. 
          
      \begin{lem} \label{unique morphism lemma}
        There exists a unique morphism 
        $$A^+ \langle X_1,X_2 \rangle/(g_1,g_2) \to\mathcal{O}_{Y'}^+(Y')$$ 
        such that the following diagram is commutative: 
$$
\begin{tikzcd} [row sep = large, column sep = large] 
                                                               &|[alias = A0]|\mathcal{O}_{Y'}^+(Y') \arrow[r] &  |[alias = A1]| \mathcal{O}_{Y'_{\text{red}}}(Y') \\ 
 A^+ \langle X_1,X_2 \rangle/(g_1,g_2) \arrow[r]  \arrow[to = A0]& \mathcal{O}_{X_{\text{red}}}(X) \otimes_D E \arrow[to = A1]& E \arrow[l] \\
 A^+  \arrow[r] \arrow[to = A0] \arrow[u] & \mathcal{O}_{X_{\text{red}}}(X) \arrow[u] & D. \arrow[u] \arrow[l] 
\end{tikzcd}
$$
            \end{lem} 
    \begin{proof} 
       
        Observe from diagram \eqref{d5} above 
   that we have a well defined
    map $$A^+\langle X_1,X_2 \rangle/(g_1,g_2) \to \mathcal{O}_{Y'_{\text{red}}}(Y').$$
        For every $i$, let $\widetilde{x}_i$ be the image in $\mathcal{O}_{Y'_{\text{red}}}(Y')$ 
        of $X_i$ for this map. 
        Our goal is to show that there exists a unique set of lifts $x_1,x_2 \in \mathcal{O}^+_{Y'}(Y')$ of $\widetilde{x}_1,\widetilde{x}_2$ such 
        that if $\mathbf{x} := (x_1,x_2)$ then 
        $g_j(\mathbf{x}) = 0$ for every $j$. 
        
{Since $\mathcal{O}_{Y'_{\text{red}}}$ is the sheafification of the quotient presheaf $U \mapsto \mathcal{O}_{Y'}^+(U) / \mathfrak{m}_{\mathcal{O}_{Y'}^+}(U)$, we can take a covering $\{U_i\}$ of $Y'$ such that the $(\widetilde{x}_1)_{\lvert U_i}$ are in the image of
\[
\mathcal{O}_{Y'}^+(U_i) / \mathfrak{m}_{\mathcal{O}_{Y'}^+}(U_i) \to \mathcal{O}_{Y'_{\text{red}}}(U_i)
\]
and $g_1(y_{1i}) = 0$ for some $y_{1i} \in \mathcal{O}_{Y'}^+(U_i) / \mathfrak{m}_{\mathcal{O}_{Y'}^+}(U_i)$ mapping to $(\widetilde{x}_1)_{\lvert U_i}$. Since $\widetilde{x}_1$ is a simple root of $g_1$, it follows that the choice of the $y_{1i}$ are uniquely determined. In particular they coincide on intersections. Now $\mathcal{O}_{Y'}^+(U_i)$ is henselian along the ideal $\mathfrak{m}_{\mathcal{O}_{Y'}^+}(U_i)$ (cf. Corollary \ref{cor:henspairlif}) and so there exists a unique lift $x_{1i} \in \mathcal{O}_{Y'}^+(U_i)$ of $y_{1i}$ such that $g_1(x_{1i}) = 0$. The $x_{1i}$ glue to give $x_1 \in \mathcal{O}_{Y'}^+(Y')$. Clearly $x_1$ does not depend on the choice of the covering $\{ U_i \}$ (for another such covering $\{ U_{i'} \}$, one considers the union $\{U_i\} \cup \{U_{i'} \}$). 

It remains to show that $h(x_1)$ is invertible (the value of $x_2$ is thereafter uniquely determined). Along the map
\[
\mathcal{O}_{Y'}^+(Y') \to \mathcal{O}_{Y'_{\text{red}}}(Y'),
\] 
$h(x_1)$ is mapped to $h(\widetilde{x}_1)$. Now $h(\widetilde{x}_1)$ is invertible in $\mathcal{O}_{Y'_{\text{red}}}(Y')$ so one can take a covering $\{ V_j \}$ of $Y'$ such that the inverse of $h(\widetilde{x}_1)$ in $\mathcal{O}_{Y'_{\text{red}}}(V_j)$ comes from the image of $\mathcal{O}_{Y'}^+(V_j)$. The kernel of the map
\[
\mathcal{O}_{Y'}^+(V_j) \to \mathcal{O}_{Y'_{\text{red}}}(V_j)
\]
is $\mathfrak{m}_{\mathcal{O}_{Y'}^+}(V_j)$, which in particular, is contained in the Jacobson radical of $\mathcal{O}_{Y'}^+(V_j)$ (cf. Corollary \ref{cor:henspairlif}). Therefore $h(x_1)$ is invertible in each $\mathcal{O}_{Y'}^+(V_j)$. One can now glue the inverses to obtain an inverse for $h(x_1)$ in $\mathcal{O}_{Y'}^+(Y')$.
}

      \end{proof}  
      
      To conclude the proof,
      we must obtain a morphism 
      $Y' \to Y$. Such a morphism 
      corresponds to the morphism 
      $A\langle X_1,X_2 \rangle/ (g_1,g_2) \to \mathcal{O}_{Y'}(Y')$ which is the unique 
      extension of the map 
      $A^+\langle X_1,X_2 \rangle/(g_1,g_2) \to\mathcal{O}_{Y'}^+(Y')$ from 
      Lemma \ref{unique morphism lemma}.
\end{proof}

\begin{rem} \label{rem:somkinadho}
\emph{In the proof of Proposition \ref{prop:exifibprodspecmor}, we only showed that \eqref{eq:refpoviequs} is the fibre product of the morphisms $\alpha \colon X \to S$ and $g \colon T \to S$ when the following assumptions are satisfied (using the notation from the proof itself):
\begin{enumerate}[label=(\roman*)]
\item $X$ is affinoid, and $S$ and $T$ are affine,
\item $g$ is standard  étale, 
\item the image of $A^+$ in $C$ contains the coefficients of the polynomials $f_i$, and
\item $A$ is Tate.
\end{enumerate}
{Recall that we used $Y$ to denote $X \times_S T$.}
In general we do not know whether (i) implies $Y$ is affinoid. One possible obstruction seems to be that there is no good notion of ``affinoid morphism" between analytic adic spaces. A related question is given a specialization morphism $Z \to \Spec(D)$, does it factorize via the canonical morphism $Z \to \Spa(\mathcal{O}_{Z}(Z),\mathcal{O}_{Z}^{+}(Z))$?}
\end{rem}



\begin{lem} \label{lem:fineetafibpr}
In the setting of Proposition \ref{prop:exifibprodspecmor}, suppose $T$ and $S$ are affine, and $X$ and $Y$ are affinoid. If in addition $g \colon T \to S$ is finite, then so is the projection morphism $f \colon Y \to X$.
\end{lem}

\begin{proof}
 By \cite[Proposition 1.4.6]{hub96}, it suffices to show that $f$ is proper. The morphism $f$ is in addition quasi-separated and so by Lemma 1.3.10 in loc.cit., it suffices to verify the valuative criterion for properness. This follows from the valuative criterion for properness (for schemes) and the universal property of fibre products.
\end{proof}


We look at some fibre products (of Proposition \ref{prop:exifibprodspecmor}) appearing in nature.

\begin{es}
The diagram 
\[
 \xymatrix{
 \Spa (\mathbb{Q}_{p^n}, \mathbb{Z}_{p^n})  \ar@{->}[d]\ar@{->}[r] & \Spec (\mathbb{F}_{p^n}) 
 \ar@{->}[d] \\ 
 \Spa (\mathbb{Q}_{p}, \mathbb{Z}_{p})  \ar@{->}[r] & 
 \Spec (\mathbb{F}_{p})
 } 
\]
is a fibre product, where $\mathbb{Q}_{p^n}$ is the unramified extension of degree $n$ of $\mathbb{Q}_{p}$. In this way one can also recover unramified covers of the Fargues-Fontaine curve (cf. \cite{farguesgeo}). The diagram
\[
 \xymatrix{
 Y/\varphi^{n\mathbb{Z}} \ar@{->}[d]\ar@{->}[r] & \Spec (\mathbb{F}_{p^n}) 
 \ar@{->}[d] \\ 
 Y/\varphi^{\mathbb{Z}}  \ar@{->}[r] & 
 \Spec (\mathbb{F}_{p})
 } 
\]
is a fibre product, where $Y/\varphi^{\mathbb{Z}}$ is the Fargues-Fontaine curve ($Y = \Spa(A_{\Inf}) \backslash V(p[p^{\flat}])$).
\end{es}

We now show that fibre products preserve surjective morphisms. We will need a \emph{thining out} result for an étale morphism of affine schemes:

\begin{lem} \label{lem:fulfomash}
Let $g \colon T \to S$ be an étale morphism of affine schemes and $t \in T$ be a point which lies over $s \in S$. Then there exists a commutative diagram of affine schemes
\[
 \xymatrix{
 T' \ar@{->}[d]^{g'} \ar@{->}[r] & T 
 \ar@{->}[d]^{g} \\ 
 S'  \ar@{->}[r] & 
 S
 } 
\]
such that
\begin{enumerate}
\item the morphism $g' \colon T'\to S'$ is finite étale and $S' \to S$ is étale,
\item there exists a point $s'\in S'$ lying over $s$ with $\kappa(s') = \kappa(s)$,
\item $T'$ has exactly one point $t'$ lying over $s'$. Furthermore $t'$ lies over $t$.
\end{enumerate}
\end{lem}

\begin{proof}
This is a consequence of \cite[Tag 00UJ]{stacks-project}.
\end{proof}


\begin{prop} \label{lem:surjrnk1po}
Let $S$ and $T$ be schemes and $X$ be an analytic adic space.
 Let $\alpha \colon X \to S$ be a specialization morphism and 
 $g \colon T \to S$ an étale morphism of schemes. Given a point $x \in X$ and any point $t \in T$, there is a point of $X \times_S T$ mapping to $x$ and $t$ under the projections if and only if $x$ and $t$ lie above the same point of $S$.
\end{prop}

\begin{proof}
The condition is obviously necessary. We prove the converse. The proof in \cite[Lemma 3.5.1(i)]{hub96} relies on a comparison between the \emph{surjectivity} of $\Spec (f)$ and $\Spv (f)$ for a morphism of rings $f \colon R_1 \to R_2$, which we avoid (cf. \cite[Proposition 2.1.3]{hubkne}). In contrast, we make use of the universal property of fibre products. 

Suppose $x$ and $t$ lie above the same point $s \in S$ and assume without loss of generality that $S$ and $T$ are affine and $g$ is standard   étale. The point $x$ determines a morphism $\Spa (k(x), k(x)^{+}) \to X$, whose image of the closed point in $\Spa (k(x), k(x)^{+})$ is $x$. Let $k(x)^{+h}$ denote the henselization of the local ring $k(x)^{+}$ and set $k(x)^{h}$ to be the fraction field of $k(x)^{+h}$. Now $k(x)^{+h}$ is again a valuation ring with the same value group as $k(x)^+$ (cf. \cite[Tag 0ASK]{stacks-project}). Equipping $k(x)^{h}$ with the valuation topology of $k(x)^{+h}$, makes the pair $(k(x)^{h}, k(x)^{+h})$ an analytic affinoid field. The extension $k(x)^{+} \to k(x)^{+h}$ is local and so in particular the image of the closed point via $\Spa (k(x)^{h}, k(x)^{+h}) \to \Spa ((k(x), k(x)^{+}))$ is again closed. 

There exists a commutative diagram
\[
 \xymatrix{
 X \times_S T \ar@/^2pc/[drrr] \ar@/_2pc/[rrddd] &&& 
  \\ 
 & P_x' \ar@{.>}[ul] \ar@{->}[dd] \ar@{->}[r] & T' \ar@{->}[d]^{g'} \ar@{->}[r] & T \ar@{->}[dd]^{g} \\
 && S' \ar@{->}[rd] & \\
 & \Spa (k(x)^h, k(x)^{+h}) \ar@{->}[r] \ar@{.>}[ru] & X  \ar@{->}[r]^{\alpha} & S.
 } 
\]
We take a moment to explain its features:
\begin{enumerate}
\item the square 
\[
 \xymatrix{
 T' \ar@{->}[d]^{g'} \ar@{->}[r] & T 
 \ar@{->}[d]^{g} \\ 
 S'  \ar@{->}[r] & 
 S
 } 
\]
satisfies the conditions of Lemma \ref{lem:fulfomash} and hereafter we use the notation setup there,
\item since the  étale morphism $S' \to S$ induces an isomorphism on the residue fields $\kappa(s') = \kappa(s)$, it follows that the specialization morphism 
$$\Spa (k(x)^h, k(x)^{+h}) \to S$$
factors uniquely through $S'$ such that the image of the closed point in $\Spa (k(x)^h, k(x)^{+h})$ is $s' \in S'$ (cf. \cite[Tag 08HQ]{stacks-project}),
\item by (2), we can define the fibre product $P_x' := \Spa (k(x)^{h}, k(x)^{+h}) \times_{S'} T'$. Note that $P'_x$ must be affinoid because there exists an adic space $Z$ (defined as $\Spa (k(x)^{h}, k(x)^{+h}) \times_{S'} T''$ where $T''$ in the notation of \cite[Tag 00UJ]{stacks-project} is $\Spec(B)$) such that the disjoint union of $Z$ and $P'_x$ is affinoid (the disjoint union is just $\Spa (k(x)^{h}, k(x)^{+h}) \times_{S} T$ which is affinoid by Remark \ref{rem:somkinadho}). Therefore by Lemma \ref{lem:fineetafibpr} the projection morphism $P_x' \to \Spa (k(x)^{h}, k(x)^{+h})$ is finite  étale,
\item by the universal property, there is a unique morphism $P_x' \to X \times_S T$ rendering the diagram commutative.
\end{enumerate}

Now since $P_x' \to \Spa (k(x)^{h}, k(x)^{+h})$ is finite (in particular it is specializing), there is a point $y \in P_x'$ which lies over the closed point of $\Spa (k(x)^{h}, k(x)^{+h})$. Since there exists a unique point $t' \in T'$ over $s'$, it follows that $y$ must lie over $t' \in T'$. Therefore the image of $y$ in $X \times_S T$ satisfies the conditions of the proposition.

\end{proof}

%

\begin{cor}
Let $S$ be a scheme, $X$ be an analytic adic space and $\alpha \colon X \to S$ a specialization morphism. The functor
\begin{align*}
\alpha_{\text{ét}} \colon S_{\text{ét}} &\to X_{\text{ét}}\\
T &\mapsto X \times_S T
\end{align*}
induces a morphism of sites $\alpha_{\text{ét}} \colon X_{\text{ét}} \to S_{\text{ét}}$.
\end{cor}

\begin{proof}
   By \cite[Tag 00X1]{stacks-project}, we must show the following. 
  \begin{enumerate}
  \item The functor $\alpha_{\text{ét}}$ preserves fibre products. 
  \item The functor $\alpha_{\text{ét}}$ preserves coverings i.e. if $T \in S_{\text{ét}}$ and $\{T_i \to T\}_i$ is an étale cover of $T$ then 
           $\{X \times_S T_i \to X \times_S T\}_i$ is an étale cover of $X \times_S T$. 
   \item  The pullback functor $\alpha_{\text{ét}}^*$ is exact. 
  \end{enumerate}  
Property (1) can be deduced from the universal property of fibre products in Proposition \ref{prop:exifibprodspecmor}. 
Property (2) is a consequence of Proposition \ref{lem:surjrnk1po}.
 Property (3) is a consequence of  \cite[Tag 00X6]{stacks-project} and (1) above. 
\end{proof}

\begin{prop} \label{prop:compmorrepovi}
Let $\alpha \colon X \to S$ be a specialization morphism, $f \colon Y \to X$ a morphism of analytic adic spaces 
and $g \colon S \to T$ a morphism of schemes. Then $(\alpha \circ f)_{\text{ét}} = \alpha_{\text{ét}} \circ f_{\text{ét}}$ and $(g \circ \alpha)_{\text{ét}} = g_{\text{ét}} \circ \alpha_{\text{ét}}$.
\end{prop}

\begin{proof}
This is a consequence of the universal property of fibre products. We first show $(\alpha \circ f)_{\text{ét}} = \alpha_{\text{ét}} \circ f_{\text{ét}}$. Let $h \colon R \to S$ be an étale morphism of schemes and consider a commutative diagram
$$
\begin{tikzcd} [row sep = large, column sep = large] 
Z \arrow[rrrd] \arrow[rdd] &&& \\
& Y \times_X (X \times_S R) \arrow[r] \arrow[d] &
 X \times_S R \arrow[r] \arrow[d] &
R \arrow[d, "h"] \\
& Y \arrow[r, "f"] &
X \arrow[r, "\alpha"] &
S 
\end{tikzcd}
$$
where $Z$ is an analytic adic space, $Z \to Y$ a morphism of analytic adic spaces and $Z \to R$ a specialization morphism. Then by the universal property of $X \times_S R$, there exists a unique morphism of analytic adic spaces $Z \to X \times_S R$, which results in a commutative diagram. Finally one applies the universal property of $Y \times_X (X \times_S R)$ to conclude.

In a similar fashion we show $(g \circ \alpha)_{\text{ét}} = g_{\text{ét}} \circ \alpha_{\text{ét}}$. Let $h \colon R \to T$ be an étale morphism of schemes
 and consider a commutative diagram

$$
\begin{tikzcd} [row sep = large, column sep = large] 
Z \arrow[rrrd] \arrow[rdd] &&& \\
& X \times_S (S \times_T R) \arrow[r] \arrow[d] &
 S \times_T R \arrow[r] \arrow[d] &
R \arrow[d, "h"] \\
& X \arrow[r, "\alpha"] &
S \arrow[r, "g"] &
T 
\end{tikzcd}
$$
where $Z$ is a analytic adic space, $Z \to X$ a morphism of analytic adic spaces and $Z \to R$ a specialization morphism. By \cite[Tag 01JN]{stacks-project}, the fibre product $S \times_T R$ is actually a fibre product in the category of locally ringed spaces. Thus by the universal property of $S \times_T R$ (in the category of locally ringed spaces), there exists a unique specialization morphism $Z \to S \times_T R$ which results in a commutative diagram. Finally one applies the universal property of $X \times_S (S \times_T R)$ to conclude. 
\end{proof}

We close this section by relating the nearby cycles functor (the pushforward of \eqref{eq:thepsisneacu}) and the pushforward of specialization morphisms (cf. \eqref{eq:nearbycuspe}).

\begin{lem} \label{lem:comsquspemoareneacu}
Let $\mathfrak{f} \colon \mathfrak{Y} \to \mathfrak{X}$ be an étale morphism of type (S) formal schemes. Then the diagram
$$
\begin{tikzcd} [row sep = large, column sep = large] 
\mathfrak{Y}_{\eta} \arrow[r, "\lambda_{\mathfrak{Y}}"] \arrow[d, "\mathfrak{f}_{\eta}"] &
 \mathfrak{Y}_s \arrow[d, "\mathfrak{f}_s"] \\
\mathfrak{X}_{\eta} \arrow[r, "\lambda_{\mathfrak{X}}"] &
\mathfrak{X}_s  
\end{tikzcd}
$$
is cartesian in the sense of Proposition \ref{prop:exifibprodspecmor}.
\end{lem}

\begin{proof}
We can assume $\mathfrak{X} = \Spf(A)$ and $\mathfrak{Y} = \Spf (B)$ and $\mathfrak{f}_s \colon \mathfrak{Y}_s \to \mathfrak{X}_s$ is standard \'etale. Then 
\[
\mathfrak{Y}_s =  \Spec(\mathcal{O}_{\mathfrak{X}_s}(\mathfrak{X}_s)[x_1,x_2]/(f_1,f_2)).
\]
 Let $g_i$ be any lifts of $f_i$ along $A \langle X_1, X_2 \rangle \to \mathcal{O}_{\mathfrak{X}_s}(\mathfrak{X}_s)[x_1, x_2]$. Then 
\[
\Spf (B) \simeq \Spf (A \langle X_1, X_2 \rangle/(g_1,g_2)).
\]
In the case $\mathfrak{X}$ is of type (S)(b), the lemma is now clear from the construction of the fibre product in Proposition \ref{prop:exifibprodspecmor} (in this case $\mathfrak{X}_{\eta}$ and $\mathfrak{Y}_{\eta}$ are both affinoid and the ring of global functions coincides with the one constructed in Proposition \ref{prop:exifibprodspecmor}). Suppose now $\mathfrak{X}$ is of type (S)(a) (in this case we do not know if $\mathfrak{X}_{\eta} = \Spa(A,A)_{a}$ is affinoid). This means $\mathfrak{Y}$ is also of type (S)(a). In this case we check universality directly. Let $Z$ be an analytic adic space making the following diagram commutative
$$
\begin{tikzcd} [row sep = large, column sep = large] 
Z \arrow[r] \arrow[d] &
 \mathfrak{Y}_s \arrow[d, "\mathfrak{f}_s"] \\
\mathfrak{X}_{\eta} \arrow[r, "\lambda_{\mathfrak{X}}"] &
\mathfrak{X}_s.
\end{tikzcd}
$$
It suffices to show that there exists a unique morphism $Z \to \Spa(B,B)$ making the following diagram commutative: 
$$
\begin{tikzcd} [row sep = large, column sep = large] 
Z \arrow[rrrd] \arrow[rrd] \arrow[rdd] &&& \\
& \mathfrak{Y}_\eta \arrow[r] \arrow[d] &
 \Spa(B,B) \arrow[r] \arrow[d] &
\mathfrak{Y}_s \arrow[d, "\mathfrak{f}_s"] \\
& \mathfrak{X}_{\eta} \arrow[r] &
\Spa(A,A) \arrow[r] &
\mathfrak{X}_s. 
\end{tikzcd}
$$
This is because such a morphism $Z \to \Spa(B,B)$ must be adic (since the composition $Z \to \Spa(B,B) \to \Spa(A,A)$ is adic) and therefore factors through the open subspace $\mathfrak{Y}_{\eta} \subseteq \Spa(B,B)$. Let us remark that the morphisms $\Spa(A,A) \to \mathfrak{X}_s$ and $\Spa(B,B) \to \mathfrak{Y}_s$ are defined in the same way as in Definition \ref{defi:spemorrepov} (the only difference in the current situation being that $\Spa(A,A)$ and $\Spa(B,B)$ are not necessarily analytic). To construct a morphism $Z \to \Spa(B,B)$, it suffices to construct a continuous morphism $B \to \mathcal{O}^{+}_{Z}(Z)$. For this one proceeds as in Lemma \ref{unique morphism lemma}.
\end{proof}

\section{Compactifications of specialization morphisms}\label{sec:comfpadssdoce}

In this section we construct compactifications of specialization morphisms which satisfy some finiteness conditions. Therefore we need to develop notions of specialization morphisms, which are of \emph{finite type}, \emph{separated} and \emph{partially proper}. The definition of a valuation ring of a scheme/adic space and of a center of a valuation ring of a scheme/adic space extend immediately to arbitrary locally ringed spaces. This will be the starting point for us.

\begin{defi} [valuation ring of a locally ringed space] \label{defi:valrinlrsp}

\emph{Let $(X, \mathcal{O}_X)$ be a locally ringed space. A} valuation ring \emph{of $(X, \mathcal{O}_X)$ is a pair $(x,A)$, where $x$ is a point of $X$ and $A$ is a valuation ring of the residue class field $k(x)$ of $x$.}
\end{defi}

\begin{defi} [center of a valuation ring] \label{defi:valrincenlrs}
\emph{Let $(X, \mathcal{O}_X)$ be a locally ringed space with a valuation ring $(x,A)$. A point $y \in X$ is called a} center \emph{of $(x,A)$ if there exists a morphism $\Spec(A) \to (X, \mathcal{O}_X)$ of locally ringed spaces, such that the image of the generic point (resp. closed point) of $\Spec(A)$ in $X$ is $x$ (resp. $y$). Moreover we demand that the composition $\Spec k(x) \to \Spec(A) \to (X, \mathcal{O}_X)$ induces $\id \colon \mathcal{O}_{X,x}/\mathfrak{m}_x \to k(x)$, where $\mathfrak{m}_x \subset \mathcal{O}_{X,x}$ is the maximal ideal.}
\end{defi}


\begin{lem} \label{lem:unicenqfvqr}
Let $(s,A)$ be a valuation ring of an affine scheme $S$. Then $(s,A)$ has at most one center on $S$. 
\end{lem}

\begin{proof}
Let $S = \Spec(R)$ and let $\mathfrak{p} \subseteq R$ correspond to $s \in S$. Then morphisms $\Spec(A) \to \Spec(R)$ such that $(0)$ is sent to $\mathfrak{p}$ correspond to injections $R/\mathfrak{p} \hookrightarrow A$. There is at most one such morphism over $k(s)$.
\end{proof}

\begin{rem} \label{rem:substoprefpoin}
\emph{For a locally ringed space $(X, \mathcal{O}_X)$, let $X_v$ denote the set of all valuation rings of $X$. Every morphism $\alpha \colon (X, \mathcal{O}_X) \to (Y, \mathcal{O}_Y)$ induces a mapping $\alpha_v \colon X_v \to Y_v$ given by $(x,A) \mapsto (\alpha(x), A \cap k(\alpha(x)))$. We equip $X_v$ with the topology which is generated by the sets
\[
D(U,f) := \left\{ (x,A) \in X_v \text{ }\lvert \text{ } x \in U \text{ and } f(x) \in A \right\}
\]
where $U \subset X$ is open and $f \in \mathcal{O}_X(U)$.}
\end{rem}

In case of overlap, the following lemma says that the notion of a valuation ring/center of a valuation ring introduced in Definitions \ref{defi:valrinlrsp}/\ref{defi:valrincenlrs} coincide with that of Huber's definition in \cite{hub96}.

\begin{lem} \label{lem:valutopsamasadsp}
Let $X$ be an analytic adic space. Let $X_v^h$ \emph{be the} $X_v$ defined in \cite[page 52]{hub96}. We equip $X_v^h$ with the topology defined in (1.3.11) loc.cit. Then there is an isomorphism as topological spaces $(X_{\text{red}})_v \simeq X_v^h$. 
\end{lem}

\begin{proof}
It is enough to note that the canonical morphism $X_{\text{red}} \to (X,\mathcal{O}_X^{+})$ induces an isomorphism on residue fields (for $x \in X$, the residue field corresponds to the residue field of $k(x)^+$) and that there is a 1-1 correspondence between the set of all valuation rings of the residue field of $k(x)^+$ and valuation rings of $k(x)$ which are contained in $k(x)^+$.
\end{proof}

{
\begin{defi} [locally of finite type] \label{def:lftspmorref}
\emph{Let $\alpha \colon X \to S$ be a specialization morphism. Then $\alpha$ is called} locally of finite type \emph{if for every $x \in X$, there exists an affinoid open neighbourhood $U \ni x$ in $X$ and an affine open subscheme $W \subseteq S$
satisfying the following conditions. 
\begin{enumerate}
\item  We have that $\alpha(U) \subseteq W$ and if
\begin{align*}
&a \colon \mathcal{O}_S(W) \to \mathcal{O}_{X_{\text{red}}}(U) \\
&b \colon \mathcal{O}_X^+(U) \to \mathcal{O}_{X_{\text{red}}}(U)
\end{align*}
denote the natural homomorphisms then $\im(a) \subseteq \im(b)$.
\item There is a finite subset $E \subseteq \im(b)$ with 
$\im(b)$ integral over $\im(a)[E]$.
\end{enumerate}}
\emph{We will sometimes call such a pair $(U,W)$ (and the induced specialization morphism $\alpha_{\lvert U} \colon U \to W$)} good.
\end{defi}}


The reader will note (despite the choice of terminolgy) that Definition \ref{def:lftspmorref} resembles the notion of locally of $^{+}$weakly finite type morphisms between adic spaces \emph{modulo topologically nilpotent elements}.

\begin{rem} \label{rem:spmorlftrefde}
\emph{It is easy to check that if the pair $(U,W)$ satisfies the conditions in Definition \ref{def:lftspmorref} then so does $(V,W)$ for any affinoid open neighbourhood $V \ni x$ in $X$ with $V\subseteq U$. Indeed the modified morphisms $a'$ and $b'$ sit in commutative diagrams
\[
a' \colon \mathcal{O}_S(W) \xrightarrow{a} \mathcal{O}_{X_{\text{red}}}(U) \xrightarrow{\widetilde{\rho}_{UV}} \mathcal{O}_{X_{\text{red}}}(V)
\]
and}
$$
\begin{tikzcd} [row sep = large, column sep = large] 
\mathcal{O}_X^+(U) \arrow[r, "b"] \arrow[d, "\rho_{UV}"] &
\mathcal{O}_{X_{\text{red}}}(U) \arrow[d, "\widetilde{\rho}_{UV}"] \\
\mathcal{O}_X^+(V) \arrow[r, "b'"] &
\mathcal{O}_{X_{\text{red}}}(V).  
\end{tikzcd}
$$
\emph{Thus $\im(a') = \im(\widetilde{\rho}_{UV} \circ a) \subseteq \im(\widetilde{\rho}_{UV} \circ b) \subseteq \im(b')$. Since $V\to U$ is of finite type, there exists a finite subset $E_V$ of $\mathcal{O}_X(V)$ such that $\mathcal{O}_X^+(V)$ is the integral closure of $\mathcal{O}_X^+(U)[E_V \cup \mathcal{O}_X(V)^{\circ \circ}]$. 
Thus $\im(b')$ is integral over $\im(a')[b'(E_V) \cup \widetilde{\rho}_{UV}(E)]$.}
%

\emph{Similarly, it is easy to check that if the pair $(U,W)$ satisfies the above conditions then so does $(U,Z)$ for any affine open subscheme $W \subseteq Z \subseteq S$.}
\end{rem}

We now show that the notion of good pair is stable under certain base change. 

{
\begin{lem} \label{lem:base change for goodness}
Let $\alpha \colon X \to S$ be a {good} specialization morphism 
  such that $\mathcal{O}_X(X)$ is 
 a Tate ring. Let $T$ be an affine scheme and $g \colon T \to S$ be a standard \'etale morphism. 
 Let $Y := X \times_S T$ be the fibre product as constructed in Proposition \ref{prop:exifibprodspecmor} and 
 $\beta \colon Y \to T$ be the projection. {We then have that $Y$ is affinoid and
 $\beta$ is a good specialization morphism. } 
\end{lem} 
}

\begin{proof} 
{We make use of the notation introduced in Proposition \ref{prop:exifibprodspecmor}. 
Hence, $S = \mathrm{Spec}(D)$ and $T = \mathrm{Spec}(E)$ 
where $E = D[x_1,x_2]/ (f_1,f_2)$ where $f_1 \in D[x_1]$ and 
$f_2$ is of the form $x_2h - 1$ for some $h \in D[x_1]$. 
This is due to the fact that we assumed $E$ is standard \'etale over $D$. 
Furthermore, let $(A,A^+) = (\mathcal{O}_X(X),\mathcal{O}^+_X(X))$. 

Since the specialization $\alpha$ is good,
 we deduce that we have a morphism 
$D \to A^+/\mathfrak{m}_{\mathcal{O}_X^+}(X)$.
Indeed, we have an injection 
$A^+/\mathfrak{m}_{\mathcal{O}_X^+}(X) \hookrightarrow \mathcal{O}_{X_{\text{red}}}(X)$.
Furthermore, the image of
$D$ for the morphism $D \to \mathcal{O}_{X_{\text{red}}}(X)$ is contained in 
$A^+/\mathfrak{m}_{\mathcal{O}_X^+}(X)$. 
Hence, we have a morphism $D \to A^+/\mathfrak{m}_{\mathcal{O}_X^+}(X)$.
It follows that we have a morphism
$D[x_1,x_2] \to A^+/\mathfrak{m}_{\mathcal{O}_X^+}(X)[x_1,x_2]$. 
Let $g_1,g_2 \in A^+[X_1,X_2]$ be some lifts of the images of $f_1,f_2$ respectively.

Remark \ref{rem:somkinadho} and our construction in Proposition \ref{prop:exifibprodspecmor}
 shows that 
$Y$ is of the form $\mathrm{Spa}(B,B^+)$ where $B = A \langle X_1,X_2 \rangle/(g_1,g_2)$.
Let $$B'_0 := A^+ \langle X_1,X_2 \rangle/((g_1,g_2)A \langle X_1,X_2 \rangle \cap A^+ \langle X_1,X_2 \rangle).$$ 
The preceding discussion gives the following diagram where we add labels to the natural morphisms
referred to in Definition \ref{def:lftspmorref}.}
$$
\begin{tikzcd} [row sep = large, column sep = large] 
B^+ \arrow[dr,"b'"] & B'_0 \arrow[l,hookrightarrow] \arrow[d] & A^+[X_1,X_2] \arrow[l] \arrow[d,twoheadrightarrow] &  A^+ \arrow[l]  \arrow[d, twoheadrightarrow] \arrow[dr,"b"] & \\ 
&\mathcal{O}_{Y_{\text{red}}}(Y) & A^+/\mathfrak{m}_{\mathcal{O}_X^+}(X)[x_1,x_2] \arrow[l] & 
 A^+/\mathfrak{m}_{\mathcal{O}_X^+}(X) \arrow[l] \arrow[r,hookrightarrow] & \mathcal{O}_{X_{\text{red}}}(X)   \\
& E \arrow[u,"a'"] & D[x_1,x_2] \arrow[l, twoheadrightarrow] \arrow[u]   & D \arrow[u,dashed] \arrow[l] \arrow[ur,"a"] &\\  
\end{tikzcd}
$$
{
In the diagram above, the morphism 
$A^+[X_1,X_2] \to A^+/\mathfrak{m}_{\mathcal{O}_X^+}(X)[x_1,x_2]$ sends $X_1,X_2$ to $x_1,x_2$ respectively. 
The map $A^+/\mathfrak{m}_{\mathcal{O}_X^+}(X)[x_1,x_2] \to \mathcal{O}_{Y_{\text{red}}}(Y)$ is induced by the universal property of the tensor product. 
To check the commutativity of the diagram, we must verify that the top left square is commutative since
commutativity is clear everywhere else. This reduces to checking that the images of $X_1,X_2$ for the composition
$A^+[X_1,X_2] \to B'_0 \to \mathcal{O}_{Y_{\text{red}}}(Y)$ coincides with the images of 
$x_1,x_2$ for the map $A^+/\mathfrak{m}_{\mathcal{O}_X^+}(X)[x_1,x_2] \to \mathcal{O}_{Y_{\text{red}}}(Y)$ respectively. Equivalently, 
we must show that the images of $X_1,X_2$ for the composition
$A^+[X_1,X_2] \to B'_0 \to \mathcal{O}_{Y_{\text{red}}}(Y)$ coincides with the images of 
$x_1,x_2$ for the map $a'$ respectively.
This is a consequence of the construction of the morphism 
$a'$ in Proposition \ref{prop:exifibprodspecmor} with the key point being Lemma \ref{equality lemma}.}

{We claim that 
$a'(E) \subseteq b'(B'_0)$. 
This can be verified by a simple diagram chase. This verifies condition (1) in 
Definition \ref{def:lftspmorref}.}

{We now verify condition (2) of Definition \ref{def:lftspmorref}. 
Observe that it suffices to show that
there exists a finite set $\mathfrak{E} \subset b'(B'_0)$ such that 
$b'(B'_0)$ is integral 
over $a'(E)[\mathfrak{E}]$. 
This is a consequence of the fact that $B^+$ is integral over $B'_0$.
By definition, we have that 
$$B_1 := A^+[X_1,X_2]/((g_1,g_2)A\langle X_1,X_2\rangle \cap A^+[X_1,X_2]) \subseteq B'_0.$$ 
We claim
$b'(B_1) = b'(B'_0)$. 
Indeed, 
let $A_0$ be a ring of definition such that $A_0 \subseteq A^+$
  and $I \subset A_0$ be an ideal of definition.   
 Recall that $I \subseteq \mathcal{O}^{\circ \circ}_{X}(X) \subseteq \mathfrak{m}_{\mathcal{O}^+_{X}}(X)$
 where the first inclusion is by definition and the second is due to Lemma \ref{lem:prshinve}. 
 Observe that every element $f \in A^+\langle X_1,X_2 \rangle$ can be realized 
 as $f_1 + f_2$ where $f_1 \in A^+[X_1,X_2]$ and $f_2 \in I A^+\langle X_1,X_2 \rangle$. 
 The claim is now easily deduced from this.

 The goodness of the morphism $\alpha$ implies
 that there exists a finite set $\mathfrak{D} \subset b(A^+)$ such that 
 $b(A^+)$ is integral over
 $a(D)[\mathfrak{D}]$.
  Let $\mathfrak{E}$ denote the image of $\mathfrak{D}$
   in  $\mathcal{O}_{Y_{\text{red}}}(Y)$ for the map 
   $\mathcal{O}_{X_{\text{red}}}(X) \to \mathcal{O}_{Y_{\text{red}}}(Y)$.
   Observe that $\mathfrak{E} \subset b'(B'_0)$. 
 A diagram chase then
  shows that the subring generated by the image of $A^+$ in $\mathcal{O}_{Y_{\text{red}}}(Y)$
  and $a'(E)$ 
  is integral over the ring $a'(E)[\mathfrak{E}]$.

 Furthermore, since the diagram above commutes, the images of $X_1,X_2$
 for the map $b'_{|B_1}$ are contained in
 $a'(E)$. 
  It follows that $b'(B_1)$ 
 is in fact integral over $a'(E)[\mathfrak{E}]$. 
 This completes the proof.}
\end{proof}

We now use Lemma \ref{lem:base change for goodness} to deduce that locally finite type
specializations are stable under base change. 
{
\begin{prop} \label{prop : base change for lft}
 Let $\alpha \colon X \to S$ be a specialization morphism which is locally of finite type.
  Let $g \colon T \to S$ be an \'etale morphism. 
 Let $Y := X \times_S T$ be the fibre product as constructed in Proposition \ref{prop:exifibprodspecmor} and 
 $\beta \colon Y \to T$ be the projection. {Then
 $\beta$ is a specialization morphism which is locally of finite type.}
\end{prop} 
\begin{proof} 
 Let $p \in Y$ and let $h \colon Y \to X$ denote the projection. 
 Since $\alpha$ is locally of finite type, there exists an affine open neighbourhood
 $S' \subseteq S$ of $\alpha(h(p))$ 
 and an affinoid open neighbourhood $X' \subseteq X$ of $h(p)$ such that 
 the restriction $\alpha_{\lvert X'} \colon X' \to S'$
  is good. 
By Remark \ref{rem:spmorlftrefde}, we can shrink $X'$ to be such that 
$\mathcal{O}_{X'}(X')$ is Tate. 
Let $T' \subset T$ be an affine open neighbourhood of $\beta(p)$ such that the 
restriction $g_{\lvert T'} \colon T' \to S'$ is standard \'etale. 
By Lemma \ref{lem:base change for goodness}, 
the map $X' \times_{S'} T' \to T'$ is good. 
Now observe that by the universal property of base change, 
$X' \times_{S'} T'$ is isomorphic as an adic space to 
$\beta^{-1}(T') \cap h^{-1}(X')$ which is open in $X \times_S T$. 
\end{proof}
}

{In the following lemma, given a site $C$, we abuse notation and use $C$ itself to 
denote the associated topos $C^{\sim}$.}

{
\begin{lem} \label{push forward commutes with filtered colimits}
Let $X$ be an analytic adic space and $S$ be a scheme. 
\begin{enumerate}
\item Then the topoi $X_{\mathrm{\acute{e}t}}$ and $S_{\mathrm{\acute{e}t}}$ are algebraic (cf. \cite[Expos\'e VI, D\'efinition 2.3]{SGA4}).
\item If $f \colon X \to S$ is a good specialization morphism  such that $\mathcal{O}_X(X)$ is Tate
 then the induced morphism of topoi
$f_{\mathrm{\acute{e}t}}  \colon X_{\mathrm{\acute{e}t}} \to S_{\mathrm{\acute{e}t}}$ is coherent (cf. \cite[Expos\'e VI, D\'efinition 3.1]{SGA4}). 
\end{enumerate} 
\end{lem}
\begin{proof} 
 We only show that $X_{\mathrm{\acute{e}t}}$ is an algebraic topos.
 The proof for $S_{\mathrm{\acute{e}t}}$ is similar.
 Let $\mathcal{F}$ be the full subcategory of objects 
 $Y \in X_{\mathrm{\acute{e}t}}$ such that $Y$ is affinoid and in addition the morphism $Y \to X$ factors through an open embedding
 $Z \to X$ where $Z$ is affinoid. Note that in this case,
 $Y \times_X Y = Y \times_Z Y$ and $Y \times_Z Y$ is affinoid. 
 Observe that $\mathcal{F}$ is a generating family.
 By \cite[Expos\'e VI, Proposition 2.1]{SGA4}, 
every object in $\mathcal{F}$
  is coherent. 
 By the criterion given in \cite[Expos\'e VI, Proposition 2.2(i ter)]{SGA4},
 we see that 
 $X_{\mathrm{\acute{e}t}}$ is an algebraic topos.
 
 
 Lastly, we check that the induced morphism of topoi 
$X_{\mathrm{\acute{e}t}} \to S_{\mathrm{\acute{e}t}}$
is coherent. This is a direct consequence of \cite[Expos\'e VI, Proposition 3.2]{SGA4}
applied to the generating family of affines which are standard \'etale over $S$
and Lemma \ref{lem:base change for goodness}.
\end{proof}}

\begin{defi} [finite type]
\emph{Let $\alpha \colon X \to S$ be a specialization morphism. Then $\alpha$ is called} finite type \emph{if $\alpha$ is locally of finite type and quasi-compact.}
\end{defi}

\begin{defi} [quasi-separated]
\begin{enumerate}
\item \emph{Let $(X, \mathcal{O}_X)$ be a locally ringed space. Then $(X, \mathcal{O}_X)$ is called} quasi-separated \emph{if the intersection of two quasi-compact open subsets of $X$ is quasi-compact.}
\item \emph{A morphism $\alpha \colon (X, \mathcal{O}_X) \to (Y, \mathcal{O}_Y)$ between locally ringed spaces is called} quasi-separated \emph{if for every quasi-separated open subspace $U$ of $Y$, the inverse image $\alpha^{-1}(U)$ is quasi-separated.}
\end{enumerate}
\end{defi}

{We next show (cf. Lemma \ref{lem : base change for qs}) that quasi-separated is stable under a standard \'etale base change.

\begin{lem} \label{lem : finite intersections in finite intersections}
 Let $X$ be a locally noetherian adic space. Let $Y_1,Y_2$ be affinoid open subsets of $X$ such that 
 $Y_1 \cap Y_2$ is quasi-compact. 
 Let $U_1 \subseteq Y_1$ and $U_2 \subseteq Y_2$ be affinoid open subsets. 
 Then $U_1 \cap U_2$ is quasi-compact. 
\end{lem} 
\begin{proof}
 Let $\mathcal{V}$ be a finite family of affinoid open subsets of $Y_1 \cap Y_2$ 
 such that $Y_1 \cap Y_2 = \bigcup_{V \in \mathcal{V}} V$.
 It suffices to verify that $U_1 \cap U_2 \cap V$ is affinoid for every $V \in \mathcal{V}$. 
 Given $V \in \mathcal{V}$, observe that since $U_1$ and $V$ are both affinoid open subsets of the affinoid adic space $Y_1$, 
$U_1 \cap V$ is an affinoid open subset of $V$. 
Indeed, one checks using the universal property of the fibre product that 
$U_1 \cap V = U_1 \times_{Y_1} V$. 
Similarly, $U_2 \cap V$ is an affinoid open subset of $V$.
Applying the same principle again, we get that 
$(U_1 \cap V) \cap (U_2 \cap V) = U_1 \cap  U_2 \cap V$ is an affinoid open subset of $V$. 
\end{proof} 

\begin{lem} \label{lem : condition for qs}
  Let $X$ be a locally noetherian adic space.   
  The following statements are equivalent. 
    \begin{enumerate}
  \item The space $X$ is quasi-separated. 
  \item There exists a family $\mathcal{V}$ of affinoid open subsets such that 
  $$X = \bigcup_{V \in \mathcal{V}} V$$ and 
  for every $V_1,V_2 \in \mathcal{V}$, $V_1 \cap V_2$ is quasi-compact. 
  \end{enumerate}
\end{lem}
\begin{proof}
We prove (2) implies (1) as the other implication is clear. 
 Let $U_1, U_2$ be quasi-compact open subsets of $X$. 
 For $i \in \{1,2\}$, let $\mathcal{W}_i$ be a finite
 family of affinoid open subsets which
  cover $U_i$ such that for every $W \in W_i$, 
  there exists $V \in \mathcal{V}$ and $W \subseteq V$. 
  Such a cover exists by observing that for every point 
  $p \in U_i$, there exists $V \in \mathcal{V}$ containing $p$ and
  we can take an affinoid open neighbourhood of 
  $p$ contained in $U_i \cap V$. 
  Note that 
  \[U_1 \cap U_2 = \bigcup_{W_1 \in \mathcal{W}_1, W_2 \in \mathcal{W}_2} W_1 \cap W_2.\]
  By Lemma \ref{lem : finite intersections in finite intersections} and our assumption on the 
  family $\mathcal{V}$, 
  for every $W_1 \in \mathcal{W}_1$ and $W_2 \in \mathcal{W}_2$, 
  $W_1 \cap W_2$ is quasi-compact. 
  This concludes the proof. 
\end{proof} 
%
%
\begin{lem} \label{lem : base change for qs}
 Let $f \colon X \to S$ be a specialization morphism where $X$ is an analytic adic space
  and $S$ is an affine scheme. 
 Let $g \colon T \to S$ be a standard \'etale morphism of schemes. 
 If $X$ is quasi-separated then $X \times_S T$ is quasi-separated. 
\end{lem} 
\begin{proof} 
There exists a 
family $\mathcal{V}$ of affinoid open subsets covering $X$ and
satisfying the following properties. 
\begin{enumerate}
\item For every $V_1,V_2 \in \mathcal{V}$, $V_1 \cap V_2$ is quasi-compact. 
\item If $V \in \mathcal{V}$ then for any affinoid open subset $V' \subseteq V$, 
 $V' \times_S T$ is affinoid. 
\end{enumerate}
Property (1) can be obtained using that
$X$ is quasi-separated. 
Observe that by Remark \ref{rem:somkinadho}, for every point $x \in X$, there exists an affinoid open neighbourhood $U$ of 
$x$ such that for every affinoid open set $U' \subseteq U$, $U' \times_S T$ is affinoid. 
Hence we have property (2).  


It follows that 
there exists a family $\mathcal{U}$ of affinoid open subsets that covers
$V_1 \cap V_2$ and for every $U \in \mathcal{U}$, 
  $U \times_S T$ is an affinoid open in $X \times_S T$. 
  Since $V_1 \cap V_2$ is quasi-compact, 
we can take $\mathcal{U}$ to be a finite family. 
  Note that
  for every $V_1,V_2 \in \mathcal{V}$, 
$$(V_1 \times_S T) \cap (V_2 \times_S T) = \bigcup_{U \in \mathcal{U}} U \times_S T.$$ 
 This implies that 
  $(V_1 \times_S T) \cap (V_2 \times_S T)$ is quasi-compact. 
 The lemma now follows from 
 Lemma \ref{lem : condition for qs}
 applied to $X \times_S T = \bigcup_{V \in \mathcal{V}} V \times_S T$. 
 \end{proof} }

\begin{defi} [separated] \label{defi:sepspemor}
\emph{Let $\alpha \colon X \to S$ be a specialization morphism. Then $\alpha$ is called} separated \emph{if $\alpha$ is quasi-separated and for every valuation ring $(x,A)$ of $X_{\text{red}}$ and for every center $s \in S$ of $\alpha_v((x,A))$, there exists at most one center $z \in X$ of $(x,A)$ with $\alpha(z) = s$.}
\end{defi}


\begin{defi} [partially proper]  \label{defi:parprospecmorre}
\emph{Let $\alpha \colon X \to S$ be a specialization morphism. Then $\alpha$ is called} partially proper \emph{if $\alpha$ is locally of finite type, quasi-separated and for every valuation ring $(x,A)$ of $X_{\text{red}}$ and for every center $s \in S$ of $\alpha_v((x,A))$, there exists a unique center $z \in X$ of $(x,A)$ with $\alpha(z) = s$.}
\end{defi}

\begin{defi} [proper] \label{def:propsepmor}
\emph{Let $\alpha \colon X \to S$ be a specialization morphism. Then $\alpha$ is called} proper \emph{if $\alpha$ is partially proper and quasi-compact.}
\end{defi}

As expected the definitions of a separated (resp. partially proper) specialization morphism admits the following reformulation.

\begin{lem} \label{lem:valucritsperef}
Let $\alpha \colon X \to S$ be a specialization morphism. The condition on centers of valuation rings in Definition \ref{defi:sepspemor} (resp. \ref{defi:parprospecmorre}) is equivalent to the following condition: For every commutative diagram
$$
\begin{tikzcd} [row sep = large, column sep = large] 
U \arrow[r, "i"] \arrow[d, "j"] &
X \arrow[d, "\alpha"] \\
V \arrow[r, "\beta"] &
S  
\end{tikzcd}
$$
with $U = \Spa(K,K^{+})$ and $V = \Spa(L,L^{+})$ spectra of analytic affinoid fields, $j$ a morphism of analytic
 adic spaces with $\mathcal{O}_V(V) \to \mathcal{O}_U(U)$ an isomorphism, $i$ a morphism of analytic
  adic spaces and $\beta$ a specialization morphism, then there is at most one (resp. unique) morphism of analytic adic spaces $f \colon V \to X$ such that
  the following diagram
   commutes
$$
\begin{tikzcd} [row sep = large, column sep = large] 
U \arrow[r, "i"] \arrow[d, "j"] &
X \arrow[d, "\alpha"] \\
V \arrow[r, "\beta"] \arrow[ru, "f"] &
S.
\end{tikzcd}
$$
\end{lem}

\begin{proof}
This is an immediate consequence of Example \ref{es:casaffielspe}.
\end{proof}

\begin{rem} \label{rem : certain base change for proper}
\emph{If $f \colon X \to S$
 is a proper specialization morphism 
 and 
 $U \subseteq S$ is open then the induced morphism $f_{\lvert f^{-1}(U)} \colon f^{-1}(U) \to U$ is proper. 
 This is an easy consequence of Proposition \ref{prop : base change for lft}.}
\end{rem}

{
\begin{lem} \label{lem : base change for proper}
 Let $f \colon X \to S$ be a specialization morphism where $X$ is an analytic adic space
  and $S$ is an affine scheme. 
 Let $g \colon T \to S$ be a standard \'etale morphism of schemes. 
 If $f$ is (partially) proper then $f_T$ is (partially) proper where $f_T \colon X \times_S T \to T$ is the base change morphism. 
\end{lem} 
\begin{proof}
Suppose $f$ is partially proper.
By Lemma \ref{prop : base change for lft}, the morphism $f_T$ is locally of finite type.
Since $S$ is quasi-separated and $f$ is quasi-separated, $X$ is quasi-separated.
By Lemma \ref{lem : base change for qs}, 
$X_T$ is quasi-separated. In particular, the morphism $f_T$ is quasi-separated. 
The condition on centers of valuation rings in Definition \ref{defi:parprospecmorre}
is easily deduced from
Lemma \ref{lem:valucritsperef}. 
Hence $f_T$ is partially proper. 

Suppose in addition that $f$ is quasi-compact. 
We claim that $f_T$ is quasi-compact. 
We show first that $X_T$ is quasi-compact. Indeed, $X$ is quasi-compact and hence 
there exists a finite family $\mathcal{U}$ of affinoid open subsets $U$ of $X$ such that
$\mathcal{O}_U(U)$ is Tate, the restriction $f_{|U}$ is good and 
$X = \bigcup_{U \in \mathcal{U}} U$.
By Lemma \ref{lem:base change for goodness},
for every $U \in \mathcal{U}$,
 $U_T := U \times_S T$ is 
affinoid as well. Since $X_T = \bigcup_{U \in \mathcal{U}} U_T$,
 we deduce that $X_T$ is quasi-compact. 
By using the same argument, we deduce that if $W \subseteq T$ is standard open then 
$f_T^{-1}(W) = X_T \times_T W$ is quasi-compact.
This verifies the claim and concludes the proof. 
\end{proof}}

For the notion of taut, we mean \cite[Definition 5.1.2]{hub96}. 

\begin{prop} \label{prop:commorinvaprofrrep}
The properties of locally of ($^+$weakly) finite type, finite type, quasi-separated, separated, partially proper, proper and taut are stable under compositions of specialization morphisms with  morphisms of analytic adic spaces or morphisms of schemes, when they make sense.
\end{prop}

\begin{proof}
Let us fix $\alpha \colon X \to S$ a specialization morphism, $f \colon Y \to X$ a morphism of analytic 
adic spaces and $g \colon S \to T$ a morphism of schemes.
\begin{enumerate}
\item (locally of finite type) Suppose $\alpha$ is locally of finite type and $f$ is locally of $^{+}$weakly finite type. We show $\alpha \circ f$ is locally of finite type. Let $y \in Y$. Since $\alpha$ is locally of finite type, there exists an affinoid open neighbourhood $U \ni f(y)$ in $X$ and an affine open subscheme $W \subseteq S$ such that $\alpha(U) \subseteq W$ and for the natural homomorphisms 
\begin{align*}
&a \colon \mathcal{O}_S(W) \to \mathcal{O}_{X_{\text{red}}}(U) \\
&b \colon \mathcal{O}_X^+(U) \to \mathcal{O}_{X_{\text{red}}}(U)
\end{align*}
we have that $\im(a) \subseteq \im(b)$ and there is a finite subset $E \subseteq \im(b)$ with $\im(b)$ integral over $\im(a)[E]$. 
Shrinking $U$ if necessary (cf. Remark \ref{rem:spmorlftrefde}), since $f$ is locally of $^{+}$weakly finite type we can assume
 that there exists an open affinoid $V \ni y$ of $Y$ such that $f(V) \subseteq U$ and $\mathcal{O}^+_{Y}(V)$ is integral over $\mathcal{O}^+_{X}(U)[E' \cup \mathcal{O}_Y(V)^{\circ \circ}]$ for some finite subset $E' \subseteq \mathcal{O}^+_{Y}(V)$. 

We claim that the composition $V \to U \to W$ is good. 
Indeed, we have the following commutative diagram where we add appropriate labels. 
 $$
\begin{tikzcd} [row sep = large, column sep = large] 
\mathcal{O}_X^+(U) \arrow[d,"b"] \arrow[r] &   \mathcal{O}_Y^+(V)  \arrow[d,"b'"]  \\ 
\mathcal{O}_{X_{\text{red}}}(U) \arrow[r]& \mathcal{O}_{Y_{\text{red}}}(V)   \\
 \mathcal{O}_S(W) \arrow[u,"a"] \arrow[ur,"a'"] \\  
\end{tikzcd}
$$
A diagram chase shows that $\mathrm{im}(a') \subseteq \mathrm{im}(b')$. 
We deduce from Lemma \ref{lem:prshinve} that $\mathrm{im}(b')$ is integral over the image of $\mathcal{O}^+_{X}(U)[E']$ 
in $\mathcal{O}_{Y_{\text{red}}}(V)$ . Let $\widetilde{E'}$ denote the image in $\mathcal{O}_{Y_{\text{red}}}(V)$
of $E'$. Similarly, let ${E_V}$ denote the image of $E$ in $\mathcal{O}_{Y_{\text{red}}}(V)$.
We see that the image of $\mathcal{O}^+_{X}(U)$ in $\mathcal{O}_{Y_{\text{red}}}(V)$
is integral over $\mathrm{im}(a')[{E}_V]$. Hence, $\mathrm{im}(b')$ is integral 
over $\mathrm{im}(a')[{E}_V \cup \widetilde{E'}]$. 
This verifies the claim. 

%
%
%
\item (finite type) We just need to check the following case: $\alpha$ and $g$ are locally of finite type
 then so is $g \circ \alpha$. But this
  is clear from the definitions. 
   The remaining cases follow from (1), since the condition of quasi-compactness is a topological one.
\item (quasi-separated) This is obvious.
\item (separated) This follows immediately from Lemma \ref{lem:valucritsperef} and the analogous statement for schemes/adic spaces.
\item (partially proper) This follows immediately from Lemma \ref{lem:valucritsperef} and the analogous statement for adic spaces.
\item (proper) This follows from (5).
\item (taut) This is obvious. 
\end{enumerate}
\end{proof}

{\begin{lem} \label{lem : affinoid to affine sep and taut}
Let $f \colon X \to S$ be a specialization morphism where $X$ is affinoid and $S$ is affine.
Then $f$ is separated.
If in addition $f$ is good and $\mathcal{O}_X(X)$ is Tate 
then $f$ is taut.
\end{lem} 
\begin{proof}  
 By \cite[Theorem 3.5(i)]{HuCv}, the space $X$ is spectral.
 It follows that $X$ is quasi-separated.  We deduce from this that the morphism $f$ is quasi-separated as well. 
  The fact that the morphism $f$ is separated can be deduced from
  \cite[Lemma 1.3.6(ii)]{hub96} and
  Lemma \ref{lem:valutopsamasadsp}.  
  
  We now prove the second assertion of the lemma assuming in addition that 
  $f$ is good and $\mathcal{O}_X(X)$ is Tate. 
 We claim the morphism $f$ is quasi-compact. Indeed, let $W \subseteq S$ be
 a quasi-compact open subset. We 
 must verify that $f^{-1}(W)$ is quasi-compact.
 Let $R$ be a ring such that $S = \mathrm{Spec}(R)$.
 Since every quasi-compact open subset of $S$ can be covered by finitely many 
 affine open subsets of the form $\mathrm{Spec}(R_f)$ for $f \in R$, we can assume 
 without loss
 of generality that
 $W$ is of the form $\mathrm{Spec}(R_f)$ where $f \in R$. 
  By Remark \ref{rem:somkinadho}, 
  $X \times_S W$ is affinoid. One checks easily that 
  $f^{-1}(W)$ is isomorphic as an adic space to $X \times_S W$. 
  In particular, it must be quasi-compact. 
  Observe that as $f$ is quasi-compact and $X$ is quasi-separated,
   $f$ must be spectral. 
  By \cite[Lemma 5.1.3(i),(iii)]{hub96}, we see that $f$ is taut. 
 
  \end{proof} 
}

The next proposition gives a crucial example of a proper specialization morphism.

\begin{prop} \label{prop:specmorarproprepov}
Let $\mathfrak{X}$ be a type (S) formal scheme. The specialization morphism (cf. \eqref{eq:nearbycuspe})
\[
\lambda_{\mathfrak{X}} \colon \mathfrak{X}_{\eta} \to \mathfrak{X}_s
\]
is proper in the sense of Definition \ref{def:propsepmor}.
\end{prop} 

\begin{proof}
Topologically the morphism $\lambda_{\mathfrak{X}}$ is constructed in \cite[Proposition 1.9.1]{hub96}. By (1.9.3) in loc.cit.,  $\lambda_{\mathfrak{X}}$ is quasi-compact and quasi-separated (for quasi-separated one also uses Lemma \ref{lem : condition for qs}). Thus it remains to show that $\lambda_{\mathfrak{X}}$ is locally of finite type and satisfies the valuative criterion for a partially proper morphism (cf. Definition \ref{defi:parprospecmorre}).

We begin by proving $\lambda_{\mathfrak{X}}$ is locally of finite type. We can assume $\mathfrak{X} = \Spf (A)$. For this we treat the type (S)(a) and (S)(b) cases separately. Suppose first $\mathfrak{X}$ is of type (S)(a) (in particular this means $\mathfrak{X}_{\eta} = \Spa(A,A)_{a}$ is an open subspace of $\Spa(A,A)$). In this case it suffices to prove the specialization morphism $\Spa(A,A) \to \mathfrak{X}_s$ is locally of finite type\footnote{Strictly speaking, as $\Spa(A,A)$ is not necessarily analytic, we did not define what it means for such a morphism to be locally of finite type, but the same definition as in \ref{def:lftspmorref} works and as in Proposition \ref{prop:commorinvaprofrrep}, compositions (whenever they make sense) of locally of finite type morphisms are again locally of finite type. So once $\Spa(A,A) \to \mathfrak{X}_s$ is proven to be locally of finite type, then so is the composition $\mathfrak{X}_{\eta} \to \Spa(A,A) \to \mathfrak{X}_s$, as desired.}. We have a commutative diagram
$$
\begin{tikzcd} [row sep = large, column sep = large] 
\mathcal{O}_{\mathfrak{X}}(\mathfrak{X}) \arrow[r, "i"] \arrow[d, "q"] &
\mathcal{O}^{+}_{\Spa(A,A)}(\Spa(A,A)) \arrow[d, "b"] \\
\mathcal{O}_{\mathfrak{X}_s}(\mathfrak{X}_s) \arrow[r, "a"] &
\mathcal{O}^{+}_{\Spa(A,A)_{\text{red}}}(\Spa(A,A)) 
\end{tikzcd}
$$
where $i$ is the identity map and $q$ is a surjection by vanishing of higher (coherent) cohomology on affine spaces. Thus $\im(a) = \im(b)$ and the map $\Spa(A,A) \to \mathfrak{X}_s$ is clearly locally of finite type. Suppose now $\mathfrak{X}$ is of type (S)(b). In this case $\mathfrak{X}_{\eta} = \Spa(A[\tfrac{1}{s}],B)$, where $s \in A$ is such that $sA$ is an ideal of definition of $A$ and $B$ is the integral closure of $A$ in $A[\tfrac{1}{s}]$. We have a commutative diagram 
$$
\begin{tikzcd} [row sep = large, column sep = large] 
\mathcal{O}_{\mathfrak{X}}(\mathfrak{X}) \arrow[r, "i"] \arrow[d, "q"] &
\mathcal{O}^{+}_{\mathfrak{X}_{\eta}}(\mathfrak{X}_{\eta}) \arrow[d, "b"] \\
\mathcal{O}_{\mathfrak{X}_s}(\mathfrak{X}_s) \arrow[r, "a"] &
\widetilde{\mathcal{O}_{\mathfrak{X}_{\eta}}^+}(\mathfrak{X}_{\eta}) 
\end{tikzcd}
$$
where again by vanishing of higher cohomology $q$ is a surjection. Thus $\im(a) = \im(a \circ q) \subseteq \im(b)$. Moreover the map $i \colon A \to B$ is integral (as remarked above). Thus $\im(b)$ is integral over $\im(a)$ and this shows $\lambda_{\mathfrak{X}}$ is locally of finite type.

Finally we show $\lambda_{\mathfrak{X}}$ satisfies the valuative criterion for partially proper morphisms. In fact this is an immediate consequence of the universal property satisfied by $\text{sp}$. Indeed given a commutative diagram (as in Lemma \ref{lem:valucritsperef})
$$
\begin{tikzcd} [row sep = large, column sep = large] 
U \arrow[r, "i"] \arrow[d, "j"] &
\mathfrak{X}_{\eta} \arrow[d, "\lambda_{\mathfrak{X}}"] \\
V \arrow[r, "\beta"] &
\mathfrak{X}_s  
\end{tikzcd}
$$
by \cite[Proposition 1.9.1(c)]{hub96}, there exists a unique morphism $f \colon V \to \mathfrak{X}_{\eta}$, which makes the relevant diagram commute. 
\end{proof}

For later use we remark the following results.

\begin{lem} \label{lem:hub1313}
Let $\alpha \colon X \to S$ be a partially proper specialization morphism. Then for every quasi-compact subset $T$ of $X$, the closure $\overline{T}$ of $T$ in $X$ is quasi-compact.
\end{lem}

\begin{proof}
It is enough to prove the lemma for $S$ affine. We may assume that $T$ is open in $X$. Put 
\[
Y := \left\{ (x,A) \in (X_{\text{red}})_{v} \text{ } \lvert \text{ } x \in T \text{ and } \alpha_v((x,A)) \text{ has a center on } S \right\}.
\]
By Lemma \ref{lem:unicenqfvqr} and Definition \ref{defi:parprospecmorre}, every $(x,A) \in Y$ has a unique center on $X$. Let $c \colon Y \to X$ be the map which assigns to each $(x,A) \in Y$ the center of $(x,A)$ on $X_{\text{red}}$. We equip $Y$ with the subspace topology of $(X_{\text{red}})_{v}$ (cf. Remark \ref{rem:substoprefpoin}). By Lemma \ref{lem:valutopsamasadsp}, just as in the proof of \cite[Lemma 1.3.13]{hub96}, we obtain that $Y$ is quasi-compact.

We now show that $c$ is continuous. Let $U$ be an affinoid open subset of $X$. We have to show $c^{-1}(U)$ is open. Since $\alpha$ is in particular locally of finite type, by Remark \ref{rem:spmorlftrefde} the natural homomorphisms 
\begin{align*}
&a \colon \mathcal{O}_S(S) \to \mathcal{O}_{X_{\text{red}}}(U) \\
&b \colon \mathcal{O}_X^+(U) \to \mathcal{O}_{X_{\text{red}}}(U)
\end{align*}
are such that $\im(a) \subseteq \im(b)$ and there is a finite subset $F \subseteq \im(b)$ with $\im(b)$ integral over $\im(a)[F]$. Let $E \subseteq \mathcal{O}_X^{+}(U)$ be any finite subset such that $b(E) = F$. Then by By Lemma \ref{lem:valutopsamasadsp} and \cite[Lemma 1.3.6(ii)]{hub96}
\[
c^{-1}(U) = \left\{ (x,A) \in Y \text{ } \lvert \text{ }x\in U \text{ and } e(x) \in A \text{ for every }e \in E \right\}.
\]
Hence $c^{-1}(U)$ is open. The rest of the proof is the same as \cite[Lemma 1.3.13]{hub96}.
\end{proof}

\begin{lem} \label{lem:extmoralrepov}
Let $S$ be a scheme, $X'$ and $Y$ be analytic adic spaces equipped with specialization morphisms $\alpha \colon X' \to S$ and 
$\beta \colon Y \to S$, $X$ an open subspace of 
$X'$ and $f \colon X \to Y$ a morphism of analytic adic spaces sitting in a commutative diagram
$$
\begin{tikzcd} [row sep = large, column sep = large] 
X \arrow[r, "f"] \arrow[d, "\alpha_{\lvert X}"] &
Y \arrow[dl, "\beta"] \\
S. &  
\end{tikzcd}
$$
\begin{enumerate}
\item If every point of $X'$ is a specialization point of $X$ and $\beta$ is separated, then there is at most one morphism $g \colon X' \to Y$ (over $S$) such that $g_{\lvert X} = f$.
\item Suppose $\beta$ is locally of finite type, the inclusion $X \to X'$ is quasi-compact and for every open subset $V$ of $X'$, the restriction mapping $\mathcal{O}_{X'}(V) \to \mathcal{O}_X(V \cap X)$ is an isomorphism of topological rings. Let $x' \in X'$, $y \in Y$ such that there exists a valuation ring $(x,A)$ of $(X, \mathcal{O}_X^{+})$ such that $x'$ is a center of $(x,A)$ on $(X', \mathcal{O}_{X'}^+)$, $y$ is a center of $f_v((x,A))$ on $(Y, \mathcal{O}_Y^{+})$, $x'$ and $y$ lie over the same point on $S$, and $x$ is the closed point of the set of generalizations of $x'$ in $X$. Then there exists an open neighbourhood $U$ of $x'$ in $X'$ and a morphism of analytic adic spaces $g \colon U \to Y$ (over $S$) with $g_{\lvert U \cap X} = f_{ \lvert U \cap X}$ and $y = g(x')$.
\item Let $g \colon X' \to Y$ be a morphism of analytic adic spaces (over $S$) with $g_{\lvert X} = f$. If every point of $X'$ is a specialization point of $X$, $\alpha$ is separated, $X'$ is quasi-compact, $Y$ is quasi-separated, $f$ is injective and for every $x \in X$, the residue class field $k(f(x))$ is dense in the residue class field $k(x)$, then $g$ is a homeomorphism onto its image. 
\end{enumerate}
\end{lem}

\begin{proof}
\begin{enumerate}
\item By Lemma \ref{lem:valutopsamasadsp}, this is the same proof as the proof of \cite[Lemma 1.3.14(i)]{hub96}.
\item Since $\beta$ is locally of finite type, there exists an affinoid open neighbourhood $L \ni y$ in $Y$ and an affine open subscheme $M \subseteq S$ such that $\beta(L) \subseteq M$ and for the natural homomorphisms 
\begin{align*}
&a \colon \mathcal{O}_S(M) \to \mathcal{O}_{Y_{\text{red}}}(L) \\
&b \colon \mathcal{O}_Y^+(L) \to \mathcal{O}_{Y_{\text{red}}}(L)
\end{align*}
we have that $\im(a) \subseteq \im(b)$ and there is a finite subset $F \subseteq \im(b)$ with $\im(b)$ integral over $\im(a)[F]$. Let $E' \subseteq \mathcal{O}_Y^+(L)$ be any finite set such that $b(E') = F$. Then the rest of the proof is the same as the proof of \cite[Lemma 1.3.14(iii)]{hub96} with $E'$ playing the role of $E$ in loc.cit. Indeed employing the notation of loc.cit. the only thing we need to check is $\psi (\mathcal{O}_Y^{+}(L)) \subseteq \mathcal{O}_{X'}^{+}(U)$.

Indeed take $y \in \mathcal{O}_Y^{+}(L)$. Since $\im(b)$ is integral over $\im(a)[F]$, we see that there exists $m \in \mathfrak{m}_{\mathcal{O}_Y^+}(L)$ and $a_0, a_1, \ldots, a_{n-1} \in b^{-1}(\im(a)[F])$ such that
\[
m = y^n+a_{n-1}y^{n-1}+ \cdots + a_0.
\]
It is easy to conclude that if $\lvert \psi(a_i)(x) \lvert \leq 1$ for all $0 \leq i \leq n$ and $x \in U$, then this implies $\lvert \psi(y)(x) \lvert \leq 1$.
\item This is the same proof as \cite[Lemma 1.3.14(ii)]{hub96}. The key point is that since $\alpha$ is separated, then $g$ satisfies the valuative criterion for separatedness\footnote{One can prove that $g$ is quasi-separated, however this is not needed.} by Lemma \ref{lem:valucritsperef}. 
\end{enumerate}
\end{proof}

\begin{rem} \label{rem:difrembueasen}
\emph{Let us remark that in the statement of Lemma \ref{lem:extmoralrepov}(2) if $f$ is in addition locally of $^{+}$weakly finite type, then so is $g$. This follows from Lemma \cite[Lemma 3.3(ii), (iv)]{hub93}.}
\end{rem}

Now we construct compactifications of specialization morphisms. Just as in the case of adic spaces (or more generally morphisms of $v$-stacks, cf. \cite{SchEtdia}), there will exist \emph{canonical} compactifications.

\begin{defi} [compactification] \label{def:compmorrefpoiofvie}
\emph{Let $\alpha \colon X \to S$ be a specialization morphism.}
\begin{enumerate}
 \item \emph{A} compactification \emph{of $\alpha$ is a commutative triangle} 
$$
\begin{tikzcd} [row sep = large, column sep = large] 
X \arrow[r, "j", hookrightarrow] \arrow[d, "\alpha"] &
Y \arrow[dl, "\beta"] \\
S &  
\end{tikzcd}
$$ 
\emph{where $Y$ is an analytic adic space, $\beta$ a partially proper specialization morphism and $j$ a quasi-compact open embedding of analytic adic spaces.}
\item \emph{A} universal compactification \emph{of $\alpha$ is a compactification $(Y, \beta, j)$ of $\alpha$ such that if}
$$
\begin{tikzcd} [row sep = large, column sep = large] 
X \arrow[r, "h"] \arrow[d, "\alpha"] &
Z \arrow[dl, "\gamma"] \\
S &  
\end{tikzcd}
$$
\emph{is a commutative triangle where $Z$ is an analytic adic space, $\gamma$ a partially proper specialization morphism and $h$ a morphism of analytic 
adic spaces, then there exists a unique morphism of analytic adic spaces $i \colon Y \to Z$ such that the diagram commutes}
$$
\begin{tikzcd} [row sep = large, column sep = large] 
X \arrow[r, "j", hookrightarrow] \arrow[rr, "h", bend left] \arrow[d, "\alpha"] &
Y \arrow[dl, "\beta"] \arrow[r, "i"] &
Z \arrow[dll, "\gamma"] \\
S. &&
\end{tikzcd}
$$
\end{enumerate}
\end{defi}

\begin{lem} \label{lem:firredrefspemor}
Let 
$$
\begin{tikzcd} [row sep = large, column sep = large] 
X \arrow[r, "j", hookrightarrow] \arrow[d, "\alpha"] &
Y \arrow[dl, "\beta"] \\
S. &  
\end{tikzcd}
$$
be a commutative diagram where $\alpha$ 
is a specialization morphism, $\beta$ is a partially proper specialization,
and $j$ is a quasi-compact open embedding such that every point of $Y$ is a specialization of a point of $j(X)$ and $\mathcal{O}_Y \to j_*\mathcal{O}_X$ is an isomorphism of sheaves of topological rings. Then for every open subset $U \subseteq S$, the induced diagram
$$
\begin{tikzcd} [row sep = large, column sep = large] 
\alpha^{-1}(U) \arrow[r, "j_{\lvert \alpha^{-1}(U)}", hookrightarrow] \arrow[d, "\alpha_{\lvert \alpha^{-1}(U)}"] &
\beta^{-1}(U) \arrow[dl, "\beta_{\lvert \beta^{-1}(U)}"] \\
U &  
\end{tikzcd}
$$
is a universal compactification of $\alpha_{\lvert \alpha^{-1}(U)}$.
\end{lem}

\begin{proof}
Let $\gamma \colon Z \to U$ be a partially proper specialization morphism and $h \colon \alpha^{-1}(U) \to Z$ a morphism of analytic adic spaces (over $U$). We have to show that there exists a unique morphism $i \colon \beta^{-1}(U) \to Z$ (over $U$), which extends $h$. Uniqueness is a consequence of Lemma \ref{lem:extmoralrepov}(1). By Lemma \ref{lem:valutopsamasadsp}, the rest of the proof is the same as the proof of \cite[Lemma 5.1.7]{hub96} with Lemma \ref{lem:extmoralrepov}(2) playing the role of Lemma 1.3.14(iii) in loc.cit.
\end{proof}

\begin{thm} \label{thm:compspemor}
Let $\alpha \colon X \to S$ be a separated, taut and locally of finite type specialization morphism. Then $\alpha$ has a universal compactification
$$
\begin{tikzcd} [row sep = large, column sep = large] 
X \arrow[r, "j", hookrightarrow] \arrow[d, "\alpha"] &
Y \arrow[dl, "\beta"] \\
S. &  
\end{tikzcd}
$$ 
Every point of $Y$ is a specialization of a point of $j(X)$ and $\mathcal{O}_Y \to j_*\mathcal{O}_X$ is an isomorphism of sheaves of topological rings.
\end{thm}

\begin{proof}
We follow closely the proof of \cite[Theorem 5.1.5]{hub96}. By Lemma \ref{lem:firredrefspemor}, we can assume $S$ is affine. Let $\sF$ be the set of all open affinoid subsets $U$ of $X$ such that the natural homomorphisms  
\begin{align*}
&a_U \colon \mathcal{O}_S(S) \to \mathcal{O}_{X_{\text{red}}}(U) \\
&b_U \colon \mathcal{O}_X^+(U) \to \mathcal{O}_{X_{\text{red}}}(U)
\end{align*}
satisfy $\im(a_U) \subseteq \im(b_U)$ and there is a finite subset $E_U \subseteq \im(b_U)$ with  $\im(b_U)$ integral over $\im(a_U)[E_U]$. For every $U \in \sF$, put $U_c := \Spa(O_X(U), I(U))$, where $I(U)$ is the smallest ring of integral elements containing $b_U^{-1}(\im(a_U))$ (i.e. $I(U)$ is the integral closure\footnote{Note that $\ker(b_U) = \mathfrak{m}_{\mathcal{O}_X^+}(U)$ which contains $\mathcal{O}_X(U)^{\circ \circ}$ and so $I(U)$ is indeed open in $\mathcal{O}_X(U)$.} of $b_U^{-1}(\im(a_U))$ in $\mathcal{O}_X(U)$). We have a diagram (the squares being commutative)

\begin{equation} \label{main fig}
\begin{tikzcd} [row sep = large, column sep = large] 
0 \arrow[r] & 
\mathfrak{m}_{\mathcal{O}_{U_c}^+}(U_c) \arrow[r] \arrow[d, equal] &
\mathcal{O}^{+}_{U_c}(U_c) \arrow[r, "b_U'"] \arrow[d, hookrightarrow] &
\mathcal{O}_{U_{c, \text{red}}}(U_c) \arrow[d] & \\
0 \arrow[r] & 
\mathfrak{m}_{\mathcal{O}_X^+}(U) \arrow[r] &
\mathcal{O}_X^+(U) \arrow[r, "b_U"] &
\mathcal{O}_{X_{\text{red}}}(U) &
\mathcal{O}_S(S). \arrow[l, "a_U"] \arrow[lu, "a_U'", dashed]  
\end{tikzcd}
\end{equation}

{We explain the equality 
$$\mathfrak{m}_{\mathcal{O}_{X}^+}(U) = \mathfrak{m}_{\mathcal{O}_{U_c}^+}(U_c).$$ 
Note that the identity $O_X(U) \to O_X(U)$
 induces an injective morphism $\varphi_U \colon U \to U_c$.
  We claim that every point of $U_c$
   is a specialization point of $\varphi_U(U)$.
    Indeed this follows by the same proof as in \cite[Proposition 4.4.3(5)(ii)]{hub96}.
   This implies the equality 
$\mathfrak{m}_{\mathcal{O}_{X}^+}(U) =\mathfrak{m}_{\mathcal{O}_{U_c}^+}(U_c)$.}

We need to show that there is a morphism $a_U' \colon \mathcal{O}_S(S) \to \mathcal{O}_{U_{c, \text{red}}}(U_c)$ (indicated by the dashed line in the above diagram), which makes the relevant triangle commutative. This is a simple diagram chase. Let $s \in \mathcal{O}_S(S)$. Since $\im(a_U) \subseteq \im(b_U)$, there exists $y \in \mathcal{O}_X^+(U)$ such that $b_U(y) = a_U(s)$. By construction $y \in \mathcal{O}^{+}_{U_c}(U_c)$ and is unique up to a choice of element $x \in \mathfrak{m}_{\mathcal{O}_{U_c}^+}(U_c)$. Thus $b_U'(y)$ depends only on $x$ and we can set $a_U'(s) := b_U'(y)$. By \cite[Tag 01I1]{stacks-project} the morphism $a_U'$ corresponds to a specialization morphism $\gamma_U \colon U_c \to S$.
{In summary we have a commutative triangle}
$$
\begin{tikzcd} [row sep = large, column sep = large] 
U \arrow[r, "\varphi_U"] \arrow[d] &
U_c \arrow[dl, "\gamma_U"] \\
S. &  
\end{tikzcd}
$$ 
We claim that

\begin{lem} \label{properties of universal compactification}
\begin{enumerate}
\item $\varphi_U$ gives an isomorphism from $U$ onto a rational subspace of $U_c$.
\item $\mathcal{O}_{U_c} \to \varphi_{U*}\mathcal{O}_U$ is an isomorphism of sheaves of topological rings.
\item {The specialization morphism $\gamma_U$ is good (in fact it is tight, cf. Definition \ref{def : tight morphism}).}
\end{enumerate}
\end{lem}

\begin{proof}
Let us prove the first point. Let $E_U^{+} \subset \mathcal{O}_X^{+}(U)$ be any finite subset such that $b_U(E_U^{+}) = E_U$. Then we claim that $\varphi_U$ is an isomorphism from $U$ onto the rational subspace 
$$\left\{ x \in U_c \text{ } \lvert \text{ } \lvert e^{+}(x) \lvert \leq 1 \text{ } \forall e^{+} \in E_U^{+} \right\}.$$
Indeed take $y \in \mathcal{O}_X^{+}(U)$. Since $\im(b_U)$ is integral over $\im(a_U)[E_U]$, we see that there exists $m \in \mathfrak{m}_{\mathcal{O}_X^+}(U)$ and $a_0, a_1, \ldots, a_{n-1} \in b_U^{-1}(\im(a_U)[E_U])$ such that
\[
m = y^n+a_{n-1}y^{n-1}+ \cdots + a_0.
\]
It is easy to conclude that if $\lvert a_i(x) \lvert \leq 1$ for all $0 \leq i \leq n$ and $x \in U_c$, then this implies $\lvert y(x) \lvert \leq 1 $ (the point is that $\mathfrak{m}_{\mathcal{O}_{U_c}^+}(U_c) = \mathfrak{m}_{\mathcal{O}_X^+}(U)$). 
{The second point follows the same proof as \cite[Proposition 4.4.3(5)(iii)]{hub96}.}

 {We now proceed to verify assertion (3) of the lemma.  
 Recall the natural homomorphisms 
\begin{align*}
&a'_U \colon \mathcal{O}_S(S) \to \mathcal{O}_{U_{c,\text{red}}}(U_c) \\
&b'_U \colon \mathcal{O}_{U_c}^+(U_c) \to \mathcal{O}_{U_{c,\text{red}}}(U_c)
\end{align*}
associated to $\gamma_U$.
By construction of the morphism $a'_U$, we see that
$\im(a'_U) \subseteq \im(b'_U)$. 
Equivalently, the image of 
$\mathcal{O}_{U_c}^+(U_c)/\mathfrak{m}_{\mathcal{O}_{U_c}^+}(U_c)$ in 
$\mathcal{O}_{U_{c,\text{red}}}(U_c)$ contains $a'_U(\mathcal{O}_S(S))$.
By definition, the image of $\mathcal{O}_{U_c}^+(U_c)$ in $\mathcal{O}_{X_{\text{red}}}(U)$ is 
integral over $a_U(\mathcal{O}_S(S))$ where 
$a_U$ is the morphism $\mathcal{O}_S(S) \to \mathcal{O}_{X_{\text{red}}}(U)$ associated to the specialization $U \to S$. 
Observe from diagram (\ref{main fig}) that 
we have an injection 
$$\mathcal{O}_{U_c}^+(U_c)/\mathfrak{m}_{\mathcal{O}_{U_c}^+}(U_c) \hookrightarrow \mathcal{O}_{X_{\text{red}}}(U).$$
We deduce from this and the commutativity of diagram (\ref{main fig}) that
$\mathcal{O}_{U_c}^+(U_c)/\mathfrak{m}_{\mathcal{O}_{U_c}^+}(U_c)$ is integral over 
$a'_U(\mathcal{O}_S(S))$.}
\end{proof}

Via the open embedding $\varphi_U$, we consider $U$ as an open subspace of $U_c$. For every $U \in \mathscr{F}$, we define subsets of the space $(X_{\text{red}})_v$ as follows:
\begin{enumerate}
\item $R(U) := \left\{ (x,A) \in (U_{\text{red}})_v \text{ } \lvert \text{ } \varphi_{U,v}((x,A)) \text{ has a center on } U_{c, \text{red}} \right\}$
\item $S(U) := \left\{ (x,A) \in U_v \text{ } \lvert \text{ } (x,A) \text{ has a center on } X \right\}$
\item $T(U) \subseteq R(U)$ denotes the set $(x,A) \in R(U)$ such that for every valuation ring $A \subseteq B \subseteq k(x)^{+}/m_x$ (here $m_x \subseteq k(x)^{+}$ is the maximal ideal of $k(x)^{+}$), either $(x,B) \in (U_{\text{red}})_v$ has no center on $X_{\text{red}}$ or $(x,B)$ has a center on $U_{\text{red}}$.  
\end{enumerate}
Let $c_U \colon R(U) \to U_c$ be the mapping which assigns to $(x,A) \in R(U)$ the center of $(x,A)$ on $U_{c, \text{red}}$ (since $U_c$ is affinoid, $c_U$ is well defined). By Lemma \ref{lem:valutopsamasadsp}, the rest of the proof follows almost verbatim as the proof of \cite[Theorem 5.1.5]{hub96}. We provide a brief summary of the proof as follows:
\begin{enumerate}
\item We put $U_d := c_U(T(U)) \subseteq U_c$ and one shows that $U_d$ is open in $U_c$ and $U \subseteq U_d$. For $U,V \in \mathscr{F}$ with $U \subseteq V$, let $\psi_{V,U} \colon U_c \to V_c$ be the morphism of adic spaces which is induced by the restriction mapping $\mathcal{O}_X(V) \to \mathcal{O}_X(U)$. Then $\psi_{V,U}(U_d) \subseteq V_d$ and the induced morphism $\psi_{V,U} \colon U_d \to V_d$ is an open embedding of adic spaces. One needs Lemma \ref{lem:extmoralrepov}(3) for this.
\item We proceed to glue the family $U_d$ for $U \in \mathscr{F}$ in two steps.
 \begin{enumerate}
 \item 
 Suppose we are given a sub-family $\sF_0 \subset \sF$ such that if $U,V \in \sF_0$ then $U \cap V$ belongs to $\sF_0$.
  Given, $U,V \in \sF_0$,
 the set $Q_{U,V} := \psi_{U,U \cap V}((U \cap V)_d) \subseteq U_d$ is open in $U_d$ and the induced morphism $\psi_{U, U \cap V} \colon (U \cap V)_d \to Q_{U,V}$ is an isomorphism. For $U,V \in \sF_0$, we put $\lambda_{V,U} := \psi_{V, U \cap V} \circ \psi_{U,U \cap V}^{-1} \colon Q_{U,V} \xrightarrow{\sim} Q_{V,U}$. Then the 
 $\{U_d\}_{U \in \sF_0}$ glue along the $\lambda_{V,U}$ (i.e. the $\lambda_{V,U}$ satisfy the cocycle condition) to give an analytic adic space over $S$.
\item In the general situation, we proceed as follows. Let $U, V \in \sF$ and $\{W_1,\ldots,W_m\}$ be a finite family of 
affinoid open subsets of $U \cap V$ such that $U \cap V = \bigcup_i W_i$.
Let $\sF_0$ denote the family obtained by taking intersections of the elements of $\{W_1,\ldots,W_m\}$.
Observe that these intersections are affinoid spaces and hence by Remark \ref{rem:spmorlftrefde} $\sF_0 \subset \sF$. Hence,
by (2a), we get that the $W_{i,d}$ glue together to give $(U \cap V)_d$. 
An important point to note is that this construction is independent of the choice of 
the family $\sF_0$. 
 We deduce that
 the set $Q_{U,V} := \psi_{U,U \cap V}((U \cap V)_d) \subseteq U_d$ is open in $U_d$ and the induced map $\psi_{U, U \cap V} \colon (U \cap V)_d \to Q_{U,V}$ is an isomorphism. We put $\lambda_{V,U} := \psi_{V, U \cap V} \circ \psi_{U,U \cap V}^{-1} \colon Q_{U,V} \xrightarrow{\sim} Q_{V,U}$. Then the 
 $\{U_d\}_{U \in \sF}$ glue along the $\lambda_{V,U}$ (i.e. the $\lambda_{V,U}$ satisfy the cocycle condition) to give the analytic adic space $Y$ over $S$.
\end{enumerate}
\item We denote $\beta \colon Y \to X$ and $j \colon X \hookrightarrow Y$ the natural inclusion. Then every point of $Y$ is a specialization point of $j(X)$ and $\mathcal{O}_Y \to j_{*}\mathcal{O}_X$ is an isomorphism of sheaves of topological rings. Moreover $j$ is a  quasi-compact open embedding and $\beta$ is a partially proper specialization morphism. 
\end{enumerate}
\end{proof}

{\begin{rem} \label{description of compactification}
 \emph{Observe from the proof of Theorem \ref{thm:compspemor} that if $f \colon U \to S$ is a good specialization morphism with $U$ affinoid and $S$ affine then 
 $\mathcal{O}_{U_c}(U_c) = \mathcal{O}_U(U)$ and 
 there exists $e_1,\ldots,e_m \in  \mathcal{O}^+_U(U)$ for some $m \in \mathbb{N}$ such that 
 $U = \{x \in U_c | |e_i(x)| \leq 1, 1 \leq i \leq m \}$.}
\end{rem} }

\begin{cor} \label{cor:compmorcompmor}
In the notations of Theorem \ref{thm:compspemor} if $\alpha$ is quasi-compact, then $\beta$ is proper.
\end{cor}

\begin{proof}
We need to show $\beta$ is quasi-compact (by Theorem \ref{thm:compspemor} it is already partially proper). Let $U$ be a quasi-compact open subset of $S$. It suffices to show $\beta^{-1}(U)$ is quasi-compact. By Theorem \ref{thm:compspemor} every point of $\beta^{-1}(U)$ is a specialization of a point of $j(\alpha^{-1}(U))$. By assumption $\alpha^{-1}(U)$ is quasi-compact and therefore $j(\alpha^{-1}(U))$ is quasi-compact. The result now follows from Lemma \ref{lem:hub1313}. 
\end{proof}

{We introduce notation regarding compactifications of specialization morphisms. 
Let $\alpha : X \to S$ be a specialization morphism which is separated, taut and locally of finite type. 
Recall from Theorem \ref{thm:compspemor} that we have a universal compactification,  
for which we use the following notation.}
$$
\begin{tikzcd} [row sep = large, column sep = large] 
X \arrow[r, "j", hookrightarrow] \arrow[d, "\alpha"] &
\overline{X}^{/S} \arrow[dl, "\overline{\alpha}^{/S}"] \\
S. &  
\end{tikzcd}
$$ 

For the purposes of proving proper base change (cf. Theorem \ref{proper base change}), we
record that universal compactifications behave well with respect to base change.

{\begin{lem} \label{compactification commutes with fibre products}
 Let $f \colon U \to S$ be a good specialization morphism.
 We assume in addition that $\mathcal{O}_U(U)$ is Tate.
  Let $W$ be an affine scheme and $g \colon W \to S$ be a standard \'etale morphism. 
  We then have that 
  \[ \overline{U \times_S W}^{/W} \simeq (\overline{U}^{/S} \times_S W).\]
\end{lem} }
Note that 
by Lemma \ref{lem : affinoid to affine sep and taut}, the morphism 
$f$ satisfies the conditions of Theorem \ref{thm:compspemor}. Hence,
$\overline{U}^{/S}$ exists. 
By 
Lemma \ref{lem:base change for goodness},
${U \times_S W} \to W$ is a good specialization morphism. 
Hence, for the same reason, $\overline{U \times_S W}^{/W}$ exists. 

\begin{proof} 
{
Let $j$ denote the quasi-compact
 open embedding $U \to \overline{U}^{/S}$.
 It suffices to show that the diagram }
 $$
\begin{tikzcd} [row sep = large, column sep = large] 
U \times_S W \arrow[r, "j_W", hookrightarrow] \arrow[d, "f_W"] &
\overline{U}^{/S} \times_S W \arrow[dl, "\overline{f}^{/S}_{W}"] \\
W &  
\end{tikzcd}
$$ 
{where 
  the subscript $W$ is used to indicate the morphisms induced by base change to $W$,
 satisfies the hypotheses of Lemma \ref{lem:firredrefspemor}. 
 
 We have the following cartesian diagram:} 
 $$
\begin{tikzcd} [row sep = large, column sep = large] 
U \times_S W \arrow[r, "j_W", hookrightarrow] \arrow[d] & \overline{U}^{/S} \times_S W \arrow[d] \\
U \arrow[r, "j", hookrightarrow] &  \overline{U}^{/S}.
\end{tikzcd}
$$ 
{By \cite[Corollary 1.2.3(iii)(b)]{hub96} and the fact that open embeddings are stable under base change, 
the morphism $j_W$ is a quasi-compact open embedding. 
By Lemma \ref{properties of universal compactification}(3),
the specialization $\overline{U}^{/S} \to S$ is good.
Hence by Lemma \ref{lem:base change for goodness}, 
$\overline{U}^{/S} \times_S W$ is affinoid and the projection morphism
$\overline{U}^{/S} \times_S W \to W$ is good.
By Lemma \ref{lem : affinoid to affine sep and taut}, 
the projection 
$\overline{U}^{/S} \times_S W \to W$ is separated.
We deduce easily using Lemma \ref{lem:valucritsperef} that 
the specialization morphism $\overline{U}^{/S} \times_S W \to W$ is partially proper.

 By Remark \ref{description of compactification}, there exists 
 $e_1,\ldots,e_m \in \mathcal{O}_{U}^+(U)$
  such that 
 $$U = \{x \in \overline{U}^{/S} | |e_i(x)| \leq 1, 1 \leq i \leq m\}.$$ 
It follows that 
$U \times_S W = \{ x \in \overline{U}^{/S} \times_S W | |e_i(x)| \leq 1, 1 \leq i \leq m\}$
where we have abused notation and denoted by $e_i$ their images in $\mathcal{O}_{U \times_S W}(U \times_S W)$. 
We deduce that
$$\mathcal{O}_{\overline{U}^{/S} \times_S W} \simeq j_{W*}\mathcal{O}_{U \times_S W}.$$


It remains to verify that every point in $\overline{U}^{/S} \times_S W$ is the specialization of a point 
in the image of the morphism $j_W$. 
Let $x \in \overline{U}^{/S} \times_S W$.
Let $y$ denote the image of $x$ for the projection 
$\overline{U}^{/S} \times_S W \to \overline{U}^{/S}$.
Since $\overline{U}^{/S}$ is a universal compactification of the 
map $U \to S$, there exists a point $y' \in U$ which specializes to $y$. 
By \cite[Lemma 1.1.10(v)]{hub96}, there exists a point $x' \in \overline{U}^{/S} \times_S W$
such that $x'$ specializes to $x$ and maps to $y'$.
Note that by the universal property of fibre products, the preimage of 
$U$ for the projection
$\overline{U}^{/S} \times_S W \to \overline{U}^{/S}$ is 
$U \times_S W$. Hence $x' \in U \times_S W$.}
%
%
%
\end{proof}

\section{Smooth base change} \label{section : smooth base change}

In order to develop a well behaved theory of cohomology with compact support for specialization morphisms, we need some preparation. In this section we establish a variant of a \emph{smooth} base change result for specialization morphisms. This is a formal consequence of the commutativity of fiber products of specialization morphisms and morphisms between schemes, cf. Proposition \ref{prop:compmorrepovi}. For the remainder of the paper, we fix a torsion ring $A$.

\begin{lem} \label{lem:smoobaschanrefpovi}
Let
$$
\begin{tikzcd} [row sep = large, column sep = large] 
X \arrow[r, "j_1"] \arrow[d, "\alpha"] &
Y \arrow[d, "\beta"] \\
S \arrow[r, "j_2"] & T  
\end{tikzcd}
$$ 
be a cartesian diagram where $\alpha$ and $\beta$ are specialization morphisms, and $j_2$ (and hence also $j_1$ by 
Proposition \ref{prop:exifibprodspecmor}) an étale morphism.
 Then there is a 
natural equivalence from $\mathcal{D}^+(Y_{\text{ét}},A) \to \mathcal{D}^+(S_{\text{ét}},A)$
\[
j_{2}^{*}\circ R^{+}\beta_{*} \xrightarrow{\sim} R^{+}\alpha_{*}\circ j_{1}^{*}.
\]
\end{lem}

\begin{proof} 
It's enough to show that for an étale sheaf $\sF$ on $Y$ in $A$-modules, the base change morphism
\[
j_{2}^{*}\circ R^{n}\beta_{*}\sF \to R^{n}\alpha_{*}\circ j_{1}^{*}\sF
\]
is bijective. Let $\overline{s} \to S$ be a geometric point of $S$. We compute
\begin{align*}
(j_{2}^{*}\circ R^{n}\beta_{*}\sF)_{\overline{s}} &\overset{(i)}{=} \varinjlim_{(U,u)} R^{n}\beta_{*}\sF(U) \\
&\overset{(ii)}{=} \varinjlim_{(U,u)} H^{n}(U \times_T Y, \sF) \\
&\overset{(iii)}{=} \varinjlim_{(U,u)} H^{n}(U \times_S X, j_{1}^{*}\sF) \\
&\overset{(iv)}{=} (R^{n}\alpha_{*}\circ j_{1}^{*}\sF)_{\overline{s}}
\end{align*}
where the (direct) limit in (i) is over all étale neighbourhoods of $\overline{s}$ over $S$, (ii) and (iv) follow by definition of higher direct image, and (iii) follows from Proposition \ref{prop:compmorrepovi}.
\end{proof}

\section{Cohomology with compact support} \label{sec:ciohomocompsup}

If $\alpha \colon X \to S$ is a specialization morphism which admits a compactification in the sense of Definition \ref{def:compmorrefpoiofvie}, one can define the functor $R^{+}\alpha_!$. In this section we will do so for the class of morphisms that interests us. Armed with Theorem \ref{thm:compspemor} and Corollary \ref{cor:compmorcompmor}, we make the following definition. 

\begin{defi} \label{defi:lowershrfurepovi}
\emph{Let $\alpha \colon X \to S$ be a separated and finite type specialization morphism. Write $\alpha$ as the composite of the open immersion $j \colon X \hookrightarrow  \overline{X}^{/S}$ and the proper map $\beta \colon \overline{X}^{/S} \to S$ (here the triple $(\overline{X}^{/S}, \beta, j)$ is the universal compactification of $\alpha$). Then we define}
\[
R^{+}\alpha_! := R^{+}\beta_{*} \circ j_{!} \colon \mathcal{D}^+(X_{\text{ét}},A) \to \mathcal{D}^+(S_{\text{ét}},A).
\]
\end{defi} 

We now investigate the behaviour of $R^{+}\alpha_{!}$ with respect to composition of morphisms in the sense of Remark \ref{rem:compmoretad}. First we show that $R^{+}\alpha_!$ can be defined on any compactification.

\begin{lem} \label{lem:indlowshrirepovi}
Let $\alpha \colon X \to S$ be a separated and finite type specialization morphism. Let 
$$
\begin{tikzcd} [row sep = large, column sep = large] 
X \arrow[r, "j'", hookrightarrow] \arrow[d, "\alpha"] &
Y \arrow[dl, "\beta'"] \\
S &  
\end{tikzcd}
$$
be a compactification of $\alpha$. Then
\[
R^{+}\alpha_! = R^{+}\beta'_{*} \circ j'_{!}.
\]
\end{lem}

\begin{proof}
By the universal property there exists a unique morphism $i \colon \overline{X}^{/S} \to Y$ such that the diagram commutes
$$
\begin{tikzcd} [row sep = large, column sep = large] 
X \arrow[r, "j", hookrightarrow] \arrow[rr, "j'", hookrightarrow, bend left] \arrow[d, "\alpha"] &
\overline{X}^{/S} \arrow[dl, "\beta"] \arrow[r, "i"] &
Y \arrow[dll, "\beta'"] \\
S. &&
\end{tikzcd}
$$
We claim that $i$ is proper. By Remark \ref{rem:difrembueasen}, it is locally of $^{+}$weakly finite type. 
Let $T \subseteq S$ be an affine open subset.
By Corollary \ref{cor:compmorcompmor}, 
$\beta$ is proper and hence 
$\beta^{-1}(T)$ is quasi-compact. 
Since $\beta'$ is quasi-separated, we get that
$\beta'^{-1}(T)$ is quasi-separated. 
It follows that the 
morphism 
$i_{|\beta^{-1}(T)}$ is quasi-compact. 
We deduce from this that there exists a cover $\{W_t\}$ of
open subsets of $Y$ such that $i_{|i^{-1}(W_t)}$ is 
quasi-compact and hence $i$ is quasi-compact.  
%
Since $\beta' \circ i = \beta$ is proper (and $\beta'$ is in particular separated), it follows from applying Lemma \ref{lem:valucritsperef}, that $i$ satisfies the valuative criterion for partially properness. It remains to show that $i$ is quasi-separated.  Let $V \subseteq Y$ be an affinoid open which maps into an affine open of $S$. Let $U_1, U_2 \subseteq \overline{X}^{/S}$ be affinoid opens which map into $V$. Then $U_1 \cap U_2$ is a finite union of affinoid opens because $U_1, U_2$ map into a common affine open of $S$. Covering $Y$ by affinoid opens like $V$ gives that $i$ is quasi-separated. This proves the claim. 

We now compute 
\begin{align*}
R^{+}\beta'_{*} \circ j'_{!} &\overset{(i)}{=} R^{+}\beta'_{*} \circ (i \circ j)_{!} \\
&\overset{(ii)}{=} R^{+}\beta'_{*} \circ R^{+}i_{!} \circ j_{!} \\
&\overset{(iii)}{=} R^{+}\beta'_{*} \circ R^{+}i_{*} \circ j_{!} \\
&\overset{(iv)}{=} R^{+}\beta_{*} \circ j_{!} 
\end{align*}
where (i) follows from commutativity, (ii) follows from \cite[Theorem 5.4.3]{hub96}, (iii) because $i$ is proper, and (iv) follows from Proposition \ref{prop:compmorrepovi}. 
\end{proof}

\begin{lem} \label{lem:lowshrcomprefpovi}
Let $\alpha \colon X \to S$ be a separated and finite type specialization morphism. Let $f \colon Y \to X$ be a separated and $^{+}$weakly finite type
 morphism of analytic adic spaces. Then $\alpha \circ f$ is a separated and finite type specialization morphism. Moreover there is a natural equivalence 
\[
R^{+}(\alpha \circ f)_{!} \simeq R^{+}\alpha_{!} \circ R^{+}f_{!}
\]
of functors $\mathcal{D}^+(Y_{\text{ét}},A) \to \mathcal{D}^+(S_{\text{ét}},A)$.
\end{lem}

\begin{proof}
The fact that the composition $\alpha \circ f$ is a separated and finite type specialization morphism follows from Proposition \ref{prop:commorinvaprofrrep}. For the second part consider the following diagram
$$
\begin{tikzcd} [row sep = large, column sep = large] 
Y \arrow[r, "j_1", hookrightarrow] \arrow[rd, "f"] &
\overline{Y}^{/X} \arrow[d, "f_1"] \arrow[r, "j_2", hookrightarrow] &
Z \arrow[d, "f_2"] \\
& X \arrow[r, "j_3", hookrightarrow] \arrow[rd, "\alpha"] &
\overline{X}^{/S} \arrow[d, "\beta"] \\
&& S
\end{tikzcd}
$$
where the top left triangle is the universal compactification of $f$ via \cite[Theorem 5.1.5]{hub96}, the bottom right triangle is the universal compactification of $\alpha$ via Theorem \ref{thm:compspemor}. Let us explain the top right square. The morphism $j_3 \circ f_1 \colon \overline{Y}^{/X} \to \overline{X}^{/S}$ is separated and $^{+}$weakly of finite type. Thus again by \cite[Theorem 5.1.5]{hub96}, it admits a universal compactification which we denote by the triple $(Z, f_2, j_2)$. In particular all the horizontal arrows (in the diagram) are quasi-compact open embeddings and the vertical arrows are proper. Therefore the composition $j_2 \circ j_1$ is a quasi-compact open embedding of analytic adic spaces and by Proposition \ref{prop:commorinvaprofrrep}, the composition $\beta \circ f_2$ is a proper specialization morphism. Thus the outer triangle is a compactification of $\alpha \circ f$.
 
Then 
$$R^{+}\alpha_{!} \circ R^{+}f_{!} = R^{+}\beta_{*}\circ j_{3!}\circ R^{+}f_{1*} \circ j_{1!}$$
and 
\begin{align*}
R^{+}(\alpha \circ f)_{!} &\overset{(i)}{=} R^{+}(\beta \circ f_2)_{*} \circ (j_2 \circ j_1)_{!} \\
&\overset{(ii)}{=} R^{+}\beta_{*} \circ R^{+}f_{2*} \circ j_{2!} \circ j_{1!} \\
&\overset{(iii)}{=} R^{+}\beta_{*} \circ j_{3!} \circ R^{+}f_{1*} \circ j_{1!}
\end{align*}
where (i) follows from Lemma \ref{lem:indlowshrirepovi}, (ii) follows from Proposition \ref{prop:compmorrepovi} and \cite[Theorem 5.4.3]{hub96}, and (iii) follows from \cite[Lemma 5.4.2]{hub96}. This completes the proof.
\end{proof}

\section{Proper base change}\label{sec:propbasersuil}

       Our goal in this section is to prove a base change theorem for 
       proper specialization morphisms. 
       Owing to the fact that we only construct fibre products for specialization 
       morphisms when the base change map is \'etale and the fact that arbitrary projective 
       limits of adic spaces do not exist, we formulate a 
       version of proper base change in Theorem \ref{proper base change} that takes into account these challenges.
       In the case when the target of a proper specialization map is strictly local then we 
       have the familiar statement as in Corollary \ref{proper base change to strict henselian local rings}.
       
       We then employ Theorem \ref{proper base change} to 
       deduce the crucial identity in Corollary \ref{lem:lowshricomutwicomrepov}
       from which can be easily deduced the commutativity of the nearby cycles and lower shriek functors.

{
\begin{rem} \label{rem : pseudo-adic space}
\emph{In what follows, we make use of the notion of (pre)pseudo-adic space.
Recall from \cite[\S 1.10]{hub96} that a {prepseudo-adic space} $Z$ is given by a pair $(\underline{Z},|Z|)$ 
  where $\underline{Z}$ is an adic space and $|Z|$ is a subset of $\underline{Z}$.
  When $|Z|$ is convex and locally pro-constructible in $\underline{Z}$, it is said to be a {pseudo-adic space}. The conditions are satisfied for instance when $|Z|$ is closed in $\underline{Z}$. 
  Given a specialization morphism 
  $f \colon X \to W$ and a point $w \in W$, we write 
  $f^{-1}(w)$ for the prepseudo-adic space 
  with $\underline{f^{-1}(w)} = X$ and 
 $|f^{-1}(w)| = \{x \in X|f(x) = w\}$.}
 \end{rem}}
 
 {We also require the notion of \'etale neighbourhood of a geometric point in the context of schemes. The relevant definitions 
and proofs can be found in \cite[Tag 03PN]{stacks-project}. We calculate all cohomology groups in the relevant 
\'etale topos.}

\begin{sit} \label{the main situation}
\emph{Let $f \colon X \to S$ be a specialization morphism.
 {Let $s \in S$ be a closed point and $\overline{s} \to S$ be a geometric
  point over $s$. 
   Let $\sF$ be a torsion abelian sheaf on $X_{\text{\'et}}$.  
 For every
 $n \in \mathbb{N}$, we have a natural map
  \begin{equation}\label{main situation equation}
[R^nf_*(\sF)]_{\overline{s}} \to \varinjlim_{(W,\overline{w})} H^n(f_W^{-1}(w),\sF)
\end{equation}
 where the colimit on the right runs over all \'etale neighbourhoods $(W,\overline{w})$ 
 of $(S,\overline{s})$, $w \in W$ is the image of $\overline{w}$, 
 $f_W \colon X_W \to W$ is the morphism induced by base change and  
 $f_W^{-1}(w)$ is the pseudo-adic space
  as defined in Remark \ref{rem : pseudo-adic space}.  
Indeed, for every 
\'etale neighbourhood $(W,\overline{w})$ 
 of $(S,\overline{s})$, we have a natural morphism 
 $H^n(X_W,\sF) \to H^n(f_W^{-1}(w),\sF)$ and hence we get a natural map}
 \[[R^nf_*(\sF)]_{\overline{s}} = \varinjlim_{(W,\overline{w})} H^n(X_W,\sF) \to \varinjlim_{(W,\overline{w})} H^n(f_W^{-1}(w),\sF).\]}
\end{sit}

{
\begin{thm}\label{proper base change}
In Situation \ref{the main situation}, if in addition $f$ is proper
then the natural morphism $(\ref{main situation equation})$ is an isomorphism. 
%
\end{thm}}

 Our strategy to prove Theorem \ref{proper base change} is to reduce to 
 when $X$ is affinoid, $S$ is affine and the morphism $f$ is \emph{tight}. This situation can then be dealt with by 
 Proposition \ref{using prospecial} which makes crucial use of \cite[Theorem 3.2.1]{hub96}. 
 
We introduce the following definition which plays a crucial role in Proposition \ref{using prospecial}.

\begin{defi} \label{def : tight morphism}
 \emph{Let $Y$ be an analytic affinoid adic space and $T$ be an affine scheme.
 A specialization morphism $f \colon Y \to T$ induces a morphism
 \[a \colon \mathcal{O}_T(T) \to \mathcal{O}_{Y_{\text{red}}}(Y).\]
 Furthermore, we have the morphism
  \[b \colon \mathcal{O}_Y^+(Y) \to \mathcal{O}_{Y_{\text{red}}}(Y).\]
  We say that the
 specialization morphism $f$ is} tight \emph{if the following conditions are satisfied.}
 \begin{enumerate}
 \item \emph{$\mathrm{im}(a) \subseteq \mathrm{im}(b)$.}
 \item \emph{$\mathrm{im}(b)$ is integral over $\mathrm{im}(a)$.}
 \end{enumerate}
 \end{defi}
 
{\begin{lem} \label{lem : tight implies partially proper}
  Let $f \colon X \to S$ be a tight specialization morphism.
  We suppose that $\mathcal{O}_X(X)$ is Tate.
   Then $f$ is partially proper. 
 \end{lem} 
 \begin{proof} 
Since $f$ is tight, it is in particular good. 
By Lemma \ref{lem : affinoid to affine sep and taut},
$f$ satisfies the conditions of Theorem \ref{thm:compspemor}. Hence,
$\overline{X}^{/S}$ and $\overline{f}^{/S}$ exist.  
 We claim that $X = \overline{X}^{/S}$ and $f = \overline{f}^{/S}$. 
 Let
$a \colon \mathcal{O}_S(S) \to {\mathcal{O}_{X_{\text{red}}}}(X)$
be the morphism induced by $f$ and 
let
$b$ denote the map
 $\mathcal{O}_X^+(X) \to {\mathcal{O}_{X_{\text{red}}}}(X)$.
 We follow the construction of $\overline{X}^{/S}$ in the proof 
 of Theorem \ref{thm:compspemor}. 
 Let $X_c := \mathrm{Spa}(\mathcal{O}_X(X),I(X))$ where 
 $I(X)$ is the integral closure of
 $b^{-1}(\mathrm{im}(a))$ in $\mathcal{O}_X(X)$.
 We show $X_c = X$. 
 Indeed, 
 suppose $x \in \mathcal{O}_X^+(X)$. Since $f$ is tight,
 there exists a monic polynomial $g$ with coefficients 
 in $\mathrm{im}(a)$ such that 
 $g(b(x)) = 0$. 
 Since $f$ is tight, $\mathrm{im}(a) \subseteq \mathrm{im}(b)$. 
 It follows that we can choose a monic lift $g'$ of $g$ 
 with coefficients in $b^{-1}(\mathrm{im}(a))$.
 Note that $g'(x) \in \mathfrak{m}_{\mathcal{O}_X^+}(X)$. 
 Let $r := g'(x)$. 
 Hence, the monic polynomial $g'' := g' - r$ has coefficients in 
 $b^{-1}(\mathrm{im}(a))$ and is such that $g''(x) = 0$. 
 Since $I(X)$ is integrally closed, we get that $x \in I(X)$. 
 
 Let $\sF$ be 
 as defined in the beginning of the proof of Theorem 
 \ref{thm:compspemor}.
 Recall that we obtain $\overline{X}^{/S}$ by glueing together 
 the $U_d$ for $U \in \sF$ where $U_d$ is as defined 
 in the end of the proof of Theorem \ref{thm:compspemor}.
 Observe that $X \in \sF$ and $X_d = X$. 
 Furthermore, for every $U \in \sF$, $U_d \subseteq X_d = X$. 
 Since $\overline{X}^{/S}$ is obtained by glueing $U_d$ as 
 $U$ varies along $\sF$,
  it follows that 
  $\overline{X}^{/S} = X$. 
 \end{proof} }

\begin{lem} \label{lem : finite + tight = tight}
Let $f \colon X \to S$ be a tight specialization morphism. 
Let $p \colon Y \to X$ be a finite morphism of adic spaces.
Then the specialization $f \circ p$ is tight. 
\end{lem}
\begin{proof} 
Note that since $Y$ is finite over $X$ and $X$ is analytic affinoid, $Y$ is analytic affinoid as well.
Let $(A,A^+) := (\mathcal{O}_X(X),\mathcal{O}^+_X(X))$ 
and $(B,B^+) := (\mathcal{O}_Y(Y),\mathcal{O}^+_Y(Y))$. 
Let $R := \mathcal{O}_S(S)$. 
 We have the following commutative diagram where we add appropriate labels. 
 $$
\begin{tikzcd} [row sep = large, column sep = large] 
A^+ \arrow[d,"b"] \arrow[r] &   B^+  \arrow[d,"b'"]  \\ 
\mathcal{O}_{X_{\text{red}}}(X) \arrow[r]& \mathcal{O}_{Y_{\text{red}}}(Y)   \\
 R \arrow[u,"a"] \arrow[ur,"a'"] \\  
\end{tikzcd}
$$
A diagram chase shows that $\mathrm{im}(a') \subseteq \mathrm{im}(b')$. 
Since $Y$ is finite over $X$, $B^+$ is integral over $A^+$. 
It follows that
$\mathrm{im}(b')$ is
integral over the image of $A^+$ in $\mathcal{O}_{Y_{\text{red}}}(Y)$. 
Since the diagram is commutative and the specialization morphism $f$ is tight, 
we see that the image of $A^+$ in $\mathcal{O}_{Y_{\text{red}}}(Y)$
is integral over $\mathrm{im}(a')$. 
Thus, $\mathrm{im}(b')$ is integral over $\mathrm{im}(a')$. 
This concludes the proof. 
\end{proof}

{
\begin{lem} \label{lem:base change for tightness}
Let $\alpha \colon X \to S$ be a {tight} specialization morphism 
  such that $\mathcal{O}_X(X)$ is 
 a Tate ring. Let $T$ be an affine scheme and $g \colon T \to S$ be a standard \'etale morphism. 
 Let $Y := X \times_S T$ be the fibre product as constructed in Proposition \ref{prop:exifibprodspecmor} and 
 $\beta \colon Y \to T$ be the projection. {We then have that $Y$ is affinoid and
 $\beta$ is a tight specialization morphism. } 
\end{lem} 
}

\begin{proof} 
{The proof is the same as that of Lemma 3.8, with $\mathfrak{D}$ (and hence $\mathfrak{E}$) being the empty sets there}. 

\end{proof}

{
\begin{lem} \label{lem : fibre is pro-special}
{Let $f : X \to S$ be a specialization morphism such that 
 the following properties are satisfied. 
 \begin{enumerate}
 \item $X$ is affinoid. 
 \item  $S$ is affine.
 \item With respect to the natural morphisms 
     \begin{align*}
&a \colon \mathcal{O}_S(S) \to \mathcal{O}_{X_{\text{red}}}(X) \\
&b \colon \mathcal{O}_X^+(X) \to \mathcal{O}_{X_{\text{red}}}(X)
\end{align*}
we have $\mathrm{im}(a) \subseteq \mathrm{im}(b)$.
\end{enumerate}

 Let $s \in S$ be a Zariski closed point. 
  Then $f^{-1}(s)$ is a
   pro-special subset of $X$. }
\end{lem}}
\begin{proof}
{ 
 Let $(A,A^+) := (\mathcal{O}_X(X),\mathcal{O}^+_X(X))$
 and $R$ be such that $S = \mathrm{Spec}(R)$. 
 Let 
 $\mathfrak{p}$ be the maximal ideal corresponding to the point $s$.}
Let $x \in X$. 
We then have the following commutative diagram. 
$$
\begin{tikzcd} [row sep = large, column sep = large] 
& A^+ \arrow[d,"b"] \arrow[r] & k(x)^+ \arrow[d] \\
R \arrow[r,"a"] & \mathcal{O}_{X_{\text{red}}}(X) \arrow[r] & \mathcal{O}_{X_{\text{red}},x}.
\end{tikzcd}
$$
Recall that we have an injection
$A^+/\mathfrak{m}_{\mathcal{O}^+_X}(X) \hookrightarrow \mathcal{O}_{X_{\text{red}}}(X)$.
    By condition (3), we have that 
    the image of $a$ is contained in $A^+/\mathfrak{m}_{\mathcal{O}^+_X}(X)$.
    We can hence remake the diagram above as follows. 
    $$
\begin{tikzcd} [row sep = large, column sep = large] 
& A^+ \arrow[d, "d"] \arrow[r] & k(x)^+ \arrow[d] \\
R \arrow[r, "e"] &A^+/\mathfrak{m}_{\mathcal{O}^+_X}(X) \arrow[r,"e'_x"] & \mathcal{O}_{X_{\text{red}},x}.
\end{tikzcd}
$$

Recall that $\mathcal{O}_{X_{\text{red}},x} = k(x)^+/k(x)^{\circ \circ}$. 
Let $\mathfrak{p}_x := \{g \in k(x)| |g(x)| < 1\}$. Observe
that $\mathfrak{p}_x/k(x)^{\circ \circ}$ is the maximal ideal of $\mathcal{O}_{X_{\text{red}},x}$.
Let $\lambda_x$ denote the composition $R \to A^+/\mathfrak{m}_{\mathcal{O}^+_X}(X) \to \mathcal{O}_{X_{\text{red}},x}$. 
By \cite[Tag 01HY]{stacks-project}, we get that 
$f(x) = s$ if and only if 
\begin{align*}
\mathfrak{p} &\subseteq \lambda_x^{-1}(\mathfrak{p}_x/k(x)^{\circ \circ})
\end{align*}
Note that the containment above is because the point $s$ is closed. 
We deduce from this that 
$f(x) = s$ if and only if 
\begin{align*}
e(\mathfrak{p}) &\subseteq {e'_x}^{-1}(\mathfrak{p}_x/k(x)^{\circ \circ}).
\end{align*}
By construction, ${e'_x}^{-1}(\mathfrak{p}_x/k(x)^{\circ \circ})$ is an ideal
in $A^+/\mathfrak{m}_{\mathcal{O}_X^+}(X)$. It follows that 
$f(x) = s$ if and only if 
\begin{align*}
e(\mathfrak{p})\cdot A^+/\mathfrak{m}_{\mathcal{O}_X^+}(X) &\subseteq {e'_x}^{-1}(\mathfrak{p}_x/k(x)^{\circ \circ}).
\end{align*}
Equivalently, $x \in f^{-1}(s)$ if and only if 
for every $g \in e(\mathfrak{p})\cdot A^+/\mathfrak{m}_{\mathcal{O}_X^+}(X)$, 
$e'_x(g) \in \mathfrak{p}_x/k(x)^{\circ \circ}.$

 Let $g' \in A^+$.  
 We then have that $e'_x(d(g')) \in \mathfrak{p}_x/k(x)^{\circ \circ}$ if and only if 
 $|g'(x)| < 1$. This can be deduced from the commutative diagram above. 
         Hence we see that $x \in f^{-1}(s)$ if and only if 
     for every $g' \in d^{-1}(e(\mathfrak{p})\cdot A^+/\mathfrak{m}_{\mathcal{O}_X^+}(X))$, 
     $|g'(x)| < 1$. 
{Thus
\[
f^{-1}(s) = \left\{ x \in X \text{ }\lvert \text{ } \lvert g'(x) \lvert <1 \text{ }\forall g' \in d^{-1}(e(\mathfrak{p}) \cdot A^+/\mathfrak{m}_{\mathcal{O}^+_X}(X)) \right\}
\]
which confirms that $f^{-1}(s)$ is a pro-special subset of $X$.} 
\end{proof}

{\begin{sit} \label{minor situation} 
\emph{Let $f \colon X \to S$ be a partially proper specialization morphism with 
 $S$ affine.  
 Let $U$ be an affinoid open subset of $X$ such that 
 $f_{|U}$ is good and $\mathcal{O}_{X}(U)$ is Tate. 
By Lemma \ref{lem : affinoid to affine sep and taut}, the morphism 
  $f_{|U} \colon U \to S$ satisfies the
   conditions 
  of Theorem \ref{thm:compspemor}. Hence
  the universal compactification 
  $f'\colon \overline{U}^{/S} \to S$ exists
  and 
  $\overline{U}^{/S}$ is affinoid. 
Let $g \colon \overline{U}^{/S} \to X$ denote the unique map that satisfies the 
 universal property associated to $\overline{U}^{/S}$.}
\end{sit}  }

{\begin{lem} \label{lem : reduce from closed subset to affinoid} 
 In Situation \ref{minor situation}, the following hold. 
 \begin{enumerate} 
 \item The morphism $g$ induces a homeomorphism 
 from $|\overline{U}^{/S}|$ onto $\overline{|U|}$ where 
 $\overline{|U|}$ is the closure of $|U|$ in $|X|$. 
 \item Let $v \in \overline{U}^{/S}$. There exists affinoid
 open neighbourhoods $P \subseteq \overline{U}^{/S}$ 
 and $Q \subseteq X$ of $v$ and $g(v)$ respectively such that 
       $g(P) \subseteq Q$ and the induced morphism 
       $\mathcal{O}_{X}(Q) \to \mathcal{O}_{\overline{U}^{/S}}(P)$ has dense image.
 \end{enumerate}      
\end{lem}

\begin{proof}
 Since $\overline{U}^{/S}$ is affinoid, it is quasi-compact. 
 Furthermore, since $f'$ is partially proper, it is separated and by 
 Theorem \ref{thm:compspemor}, every point of 
 $\overline{U}^{/S}$ is a specialization of a point of $U$. 
Therefore, by Lemma \ref{lem:extmoralrepov}(3), $g$ is a homeomorphism 
from $|\overline{U}^{/S}|$ onto
its image in $X$. Observe that this image must be contained in 
 $\overline{|U|}$. 
 
 We show that 
 the image of $g$ coincides with
 $\overline{|U|}$. 
It suffices to verify that $\mathrm{im}(g)$ is closed. 
Observe firstly that $g$ is quasi-compact. 
This is a consequence of the fact that $\overline{U}^{/S}$ is quasi-compact and 
$X$ is quasi-separated. 
Hence $\mathrm{im}(g)$ is closed if 
$g$ is specializing. 
Let $y \in X$ be a point which generalizes to a point 
$g(x)$ for some $x \in U$. 
Let $A \subseteq k(x)^+$ be a valuation ring such that 
$y$ is the center of $(x,A)$. 
Note that $g$ lifts centers uniquely because 
both $f$ and $f'$ lift centers uniquely. 
It follows that there exists a unique center $y'$ 
of $(x,A)$ in $\overline{U}^{/S}$ and $g(y') = y$. 
This verifies that $g$ is specializing. 

       The proof of assertion (2) in the statement of the lemma
        is identical to the proof of the similar statement in
       (7) of \cite[Proposition 4.4.3]{hub96} where we use Lemma 
       \ref{properties of universal compactification}(1)-(2) in place of
        5(i) and 5(iii) of loc.cit., respectively.
\end{proof} }

{\begin{lem} \label{lem : reduce from closed subset to affinoid II}
 We work within the context of Situation \ref{minor situation}. 
  Let $\sF$ be a torsion abelian sheaf on $X_{\text{\'et}}$. 
   Let $\mathscr{J}$ denote the category of standard
 \'etale neighbourhoods of $(S,\overline{s})$
 and 
 given $(W,\overline{w}) \in \mathscr{J}$, 
 let $g_W \colon \overline{U}^{/S} \times_S W \to X_W$ 
 denote the induced morphism.
 The natural morphism 
  $$\varinjlim_{(W,\overline{w}) \in \mathscr{J}} H^q(({X_W},\overline{|U_W|}),\sF)
  \to \varinjlim_{(W,\overline{w}) \in \mathscr{J}} H^q(({X_W},\overline{|U_W|} \cap f_W^{-1}(w)),\sF)$$
  is an isomorphism if and only if the natural morphism
    $$\varinjlim_{(W,\overline{w}) \in \mathscr{J}} H^q(\overline{U}^{/S} \times_S W, \sF)
  \to \varinjlim_{(W,\overline{w})  \in \mathscr{J}} H^q({f'_W}^{-1}(w),\sF)$$
 is an isomorphism where $f'_W \colon \overline{U}^{/S} \times_S W \to W$ is the projection to $W$.  
 \end{lem}} 
 \begin{proof}
 {Let $j$ denote the open embedding 
     $U \hookrightarrow \overline{U}^{/S}$.
        By
      Lemma \ref{compactification commutes with fibre products},
      $(\overline{U}^{/S} \times_S W,f'_W,j_W)$ is a universal compactification 
      of ${f_W}_{|U_W} \colon U_W \to W$ where 
      $j_W$ is induced by base change.  
The map $g_W \colon \overline{U}^{/S} \times_S W \to X_W$ which is induced from 
$g \colon \overline{U}^{/S} \to X$ is the unique map such that the following diagram}
       $$
\begin{tikzcd} [row sep = large, column sep = large] 
U_W \arrow[r,"j_W"] \arrow[drr, "{f_W}_{|U_W}", bend right] & \overline{U}^{/S} \times_S W \arrow[dr,"f'_W"] \arrow[r, "g_W"] & X_W \arrow[d,"f_W"] \\
& &W 
\end{tikzcd}
$$
{commutes. 

  By 
Lemma \ref{lem : reduce from closed subset to affinoid}
 and \cite[Proposition 2.3.7]{hub96}, we 
    get 
    compatible isomorphisms
      $$H^q(({X_W},\overline{|U_W|}),\sF) \simeq  H^q(\overline{U}^{/S} \times_S W,\sF)$$ and 
          $$H^q(({X_W},\overline{|U_W|} \cap f_W^{-1}(w)),\sF) \simeq  H^q({f'_W}^{-1}(w),\sF).$$
This is sufficient to conclude the proof.}
 \end{proof}

{
\begin{prop} \label{using prospecial}
In Situation \ref{the main situation},
if in addition we assume 
\begin{enumerate} 
\item The morphism $f$ is tight and $\mathcal{O}_X(X)$ is a Tate ring. 
\item Let $Y$ be the pseudo-adic space $(X,|Y|)$ 
  where 
  $$|Y| := \{x \in {X} | |e_i(x)| < |e_0(x)| \mbox{ for } 1 \leq i \leq n\}$$
 and $e_0,\ldots,e_n \in \mathcal{O}_{{X}}({X})$ are
  such that 
  $\mathcal{O}_{{X}}({X}) = \sum_i e_i\mathcal{O}_{{X}}({X})$.
  Let
  $i \colon Y \hookrightarrow X$ be the natural closed embedding. 
  The sheaf 
  $\sF$ is of the form $i_*(G_Y)$  where 
$G_Y$ is the constant sheaf on $Y_{\text{\'et}}$ associated to an
abelian torsion group $G$. 
\end{enumerate} 
Then (\ref{main situation equation}) is an isomorphism. 
\end{prop} }
{Note that in what follows, when there is no ambiguity, we simplify notation and write 
$G$ for the constant sheaf $G_X$ on $X_{\text{\'et}}$ for $X$ a pseudo-adic space or scheme.} 
\begin{proof}
{Recall that 
 $$[R^nf_*(\sF)]_{\overline{s}} =  \varinjlim_{(W,\overline{w})} H^n(X_W,\sF)$$
 where the 
 colimit runs over \'etale neighbourhoods of $(S,\overline{s})$ and we must show that 
 the natural morphism 
 $$\varinjlim_{(W,\overline{w})} H^n(X_W,\sF)
  \to \varinjlim_{(W,\overline{w})} H^n(f_W^{-1}(w),\sF)$$
  is an isomorphism. 
 Note that 
 we can restrict to when the colimits above run over affine \'etale neighbourhoods of $(S,\overline{s})$. 
 By \cite[Tag 02GT]{stacks-project}, we can further restrict
 so that the colimits are over standard \'etale neighbourhoods of $(S,\overline{s})$.   
Let $\mathscr{J}$ denote the category of standard
 \'etale neighbourhoods of $(S,\overline{s})$
 and let $(W,\overline{w}) \in \mathscr{J}$.
 Observe that 
 by Lemma \ref{lem:base change for tightness}, 
 $X_W$ is affinoid. 
 For the remainder of the proof, 
all limits under consideration will run over pairs $(W,\overline{w}) \in \mathscr{J}$. 

 \begin{lem} \label{lem : structure of Y_W}
 \begin{enumerate}
 \item We have a well defined pseudo-adic space 
 $Y_W := Y \times_X X_W$ such that 
 the morphism $i_W \colon Y_W \to X_W$ induced 
 by base change is a closed embedding.
 If 
 for every $0 \leq i \leq n$, $e'_i$ denotes
  the image of $e_i$ for the 
   morphism $\mathcal{O}_X(X) \to \mathcal{O}_{X_W}(X_W)$ 
  then $|Y_W| = \{x \in X_W | |e'_i(x)| < |e'_0(x)| \mbox{ for } 1 \leq i \leq n\}.$ 
\item We have that
$i_{W_*}(G) = \sF_{|X_W}$. 
\item If $\epsilon_W$ denotes the closed embedding 
$(X_W,(f_W^{-1}(w) \cap |Y_W|)) \hookrightarrow f_W^{-1}(w)$ of pseudo-adic spaces then 
$$\sF_{|f_W^{-1}(w)} = \epsilon_{W*}(G).$$
\end{enumerate}
 \end{lem} 
\begin{proof} 
The morphism $i \colon Y \to X$ is locally of finite type. 
By \cite[Lemma 1.10.6(a)]{hub96}, $Y_W$ exists and from the proof of loc.cit.,
it must be of the form described in part (1) of the lemma.
 Similarly, $(X_W,(f_W^{-1}(w) \cap |Y_W|)) = f_W^{-1}(w) \times_X Y$. 
 By \cite[Lemma 1.10.17(i)]{hub96}, 
 $i$ is proper. 
 Parts (2) and (3) follow from \cite[Proposition 4.1.2(b)]{hub96}. 
 \end{proof}  
 
   We first reduce to when $e_0 = 1$.
  Let $$U := \{x \in {X} | |e_i(x)| \leq |e_0(x)| \mbox{ for } i = 1,\ldots,n\}.$$ 
  Then $U$ is an affinoid open subset of $X$.   
  Observe that it suffices to show that the natural morphism
  $$\varinjlim_{(W,\overline{w}) \in \mathscr{J}} H^n((X_W,\overline{|U_W|}),\sF)
  \to \varinjlim_{(W,\overline{w}) \in \mathscr{J}} H^n((X_W, \overline{|U_W|} \cap f_W^{-1}(w)),\sF)$$
  is an isomorphism where 
  $U_W := U \times_S W$. 
  Indeed, 
  by Lemma \ref{lem : structure of Y_W}(1,2), $i_W \colon Y_W \hookrightarrow X_W$ is a closed embedding and $\sF_{|X_W} = i_{W*}(G)$. Hence, 
  \begin{align} \label{equn : F restricted to Y}
  H^n(X_W,\sF) &= H^n(Y_W,G)  \\
                         &= H^n((X_W,\overline{|U_W|}),\sF). \nonumber
  \end{align}
 Similarly, by Lemma \ref{lem : structure of Y_W}(3), we deduce that
  \begin{align} \label{equn : F restricted to Y II}
H^n(f_W^{-1}(w),\sF) &= H^n((X_W,|Y_W| \cap f_W^{-1}(w)),G) \\ 
                                   &= H^n((X_W, \overline{|U_W|} \cap f_W^{-1}(w)),\sF). \nonumber
\end{align}
 
 By Lemma \ref{lem : tight implies partially proper}, $f$ is partially proper and
     by Remark \ref{rem:spmorlftrefde}, $f_{|U}$ is good. 
 We are hence in the context of Situation \ref{minor situation}. 
 By Lemma \ref{lem : reduce from closed subset to affinoid II},
 it suffices to verify the proposition 
 after replacing $X$ with $\overline{U}^{/S}$. 
 Since $\mathcal{O}_X(U) = \mathcal{O}_{\overline{U}^{/S}}(\overline{U}^{/S})$,
  $e_0$ is invertible on $\overline{U}^{/S}$ 
 and after replacing $e_i$ with $e_i/e_0$ 
 for every $1 \leq i \leq n$, we can suppose $e_0 = 1$ in the 
 definition of $|Y|$.  }

{Let $(A,A^+) := (\mathcal{O}_X(X),\mathcal{O}^+_X(X))$.
  We begin by simplifying the left hand side of (\ref{main situation equation}).  
  By Lemma \ref{lem:base change for tightness},
 $X_W$ is affinoid and
 let $(A_W,A_W^+) := (\mathcal{O}_{X_W}(X_W),\mathcal{O}^+_{X_W}(X_W))$. 
 Associated to 
 the special subset $Y_W \subseteq X_W$, we define the  
 algebra $A_W(Y_W)$ as follows. 
 Let $A^+_W(Y_W)$ denote the henselization of 
 $A^+_W[e_1,\ldots,e_n]$ with respect to the ideal 
 generated by $\{e_1,\ldots,e_n\} \cup A_W^{\circ \circ}$
 where we abuse notation and use $e_i$ to denote the image of $e_i$ for the 
morphism $A \to A_W$. 
 We define $A_W(Y_W) := A^+_W(Y_W) \otimes_{A^+_W[e_1,\ldots,e_n]} A_W$.   
By 
(\ref{equn : F restricted to Y})
 and \cite[Theorem 3.2.1]{hub96}, 
we have a natural isomorphism
\begin{equation} \label{natural isomorphism adic and schemes}
H^n(X_W,\sF) \simeq H^n(\mathrm{Spec}(A_W(Y_W)),G).
\end{equation} 

 Hence we have that
\begin{align} 
\varinjlim_{(W,\overline{w})} H^n(X_W,\sF) &\overset{(i)}{\simeq} \varinjlim_{(W,\overline{w})} H^n(\mathrm{Spec}(A_W(Y_W)),G)  \nonumber  \\ 
&  \overset{(ii)}{=} H^n(\varprojlim_{(W,\overline{w})} \mathrm{Spec}(A_W(Y_W)),G) \nonumber \\
 &\overset{(iii)}{=} H^n(\mathrm{Spec}(\varinjlim_{(W,\overline{w})}A_W(Y_W)),G) \label{LHS}
 \end{align} 
 where the colimit and limit runs over all standard \'etale neighbourhoods of $(S,\overline{s})$.
Observe that (i) is a consequence of the naturality of the isomorphism (\ref{natural isomorphism adic and schemes}). 
 Note that (ii) follows from \cite[Tag 09YQ]{stacks-project} and (iii)
is due to \cite[Tag 01YW]{stacks-project}.

  We now focus on the right hand side of (\ref{main situation equation}).
 Observe that by Lemma \ref{lem:base change for tightness}, 
 the morphism $f_W$ is tight. 
 Then by Lemma \ref{lem : fibre is pro-special}, $f_W^{-1}(w)$ is a pro-special subset 
  of $X_W$. 
Furthermore, if $R$ and $R_W$ are such that $S = \mathrm{Spec}(R)$ 
 and $W = \mathrm{Spec}(R_W)$, and 
 $w \in W$ corresponds to a maximal ideal $\mathfrak{p}_W$ in $R_W$
  then
 we have the following diagram}
 $$
\begin{tikzcd} [row sep = large, column sep = large] 
&
A_W^+ \arrow[d, "d_W"] \\
R_W \arrow[r, "e_W"] & A^+_W/\mathfrak{m}_{\mathcal{O}^+_{X_W}}(X_W).
\end{tikzcd}
$$
{We abuse notation and write $\mathfrak{p}_W$ for $e_W(\mathfrak{p}_W) \subset A_W^{+}/\mathfrak{m}_{\mathcal{O}_{X_W}^+}(X_W)$. 
As in the proof of Lemma \ref{lem : fibre is pro-special}, 
\[
f_W^{-1}(w) = \left\{ x \in X_W \text{ }\lvert \text{ } \lvert a(x) \lvert <1 \text{ }\forall a \in d_W^{-1}(\mathfrak{p}_W \cdot A_W^{+}/\mathfrak{m}_{\mathcal{O}_{X_W}^+}(X_W)) \right\}.
\]
Let $I_W$ denote the ideal generated by 
$d_W^{-1}(\mathfrak{p}_W \cdot A_W^{+}/\mathfrak{m}_{\mathcal{O}_{X_W}^+}(X_W)) \bigcup \{e_1,\ldots,e_n\}$
in $A^+_W[e_1,\ldots,e_n]$
and let ${A^+_W}^{h(I_W)}$ denote the henselization of $A^+_W[e_1,\ldots,e_n]$ along 
$I_W$.

Define 
$$A_W^{h(I_W)} := {A^+_W}^{h(I_W)} \otimes_{A^+_W[e_1,\ldots,e_n]} A_W.$$
Note that $$A_W^{h(I_W)} = {A^+_W}^{h(I_W)} \otimes_{A^+[e_1,\ldots,e_n]} A.$$ 
The equality above is a consequence of the fact that $A$ is
 Tate and hence $A = A^+[1/\pi]$ for some
pseudo-uniformizer $\pi \in A$.
By  (\ref{equn : F restricted to Y II})
and 
\cite[Theorem 3.2.1]{hub96}, we have a natural isomorphism
\[H^n(f_W^{-1}(w),\sF) \simeq H^n(\mathrm{Spec}(A_W^{h(I_W)}),G).\] 
By a similar calculation as made for (\ref{LHS}), 
\begin{align}
\nonumber \varinjlim_{(W,\overline{w})} H^n(f_W^{-1}(w),\sF) &\simeq \varinjlim_{(W,\overline{w})} H^n(\mathrm{Spec}(A_W^{h(I_W)}),G) \\
\nonumber  &=  H^n(\varprojlim_{(W,\overline{w})} \mathrm{Spec}(A_W^{h(I_W)}),G) \\ \label{RHS}
  &=  H^n(\mathrm{Spec}(\varinjlim_{(W,\overline{w})} A_W^{h(I_W)}),G). 
\end{align}  
 We claim that 
 \[\varinjlim_{(W,\overline{w})} A_W^{h(I_W)} = \varinjlim_{(W,\overline{w})} A_W(Y_W).\]
  
 Let $J_W \subset A^+_W(Y_W)$ denote the ideal generated 
 by the image of 
$I_W$
for the map 
 $A^+_W[e_1,\ldots,e_n] \to A^+_W(Y_W)$. 
A diagram chase implies that if $(V,\overline{v})$ is a standard \'etale neighbourhood of $(S,\overline{s})$ such that 
the map $(V,\overline{v}) \to (S,\overline{s})$ factors through 
$(W,\overline{w}) \to (S,\overline{s})$ 
then
$$J_W \cdot A^+_V(Y_V) \subseteq J_V$$ 
where $J_W$ on the left is the image of the ideal $J_W \subset A^+_W(Y_W)$ for the morphism $A^+_W(Y_W) \to A^+_V(Y_V)$.  

We set $B :=  \varinjlim_{(W,\overline{w})} A^+_W(Y_W)$.
Let 
$J \subset B$
 denote the ideal $\varinjlim_{(W,\overline{w})} J_W$.
Let
$J'_W \subset A^+_W(Y_W)$ be the ideal generated by the image of
$\mathfrak{m}_{\mathcal{O}^+_{X_W}}(X_W) \cup \{e_1,\ldots,e_n\}$ for the morphism 
$A^+_W[e_1,\ldots,e_n] \to A^+_W(Y_W)$
and set $J' := \varinjlim_{(W,\overline{w})} J'_W \subset B$.
Finally, let 
$J''_W \subset A^+_W(Y_W)$ be the ideal generated by the image of
$A_W^{\circ \circ} \cup \{e_1,\ldots,e_n\}$ for the morphism 
$A^+_W[e_1,\ldots,e_n] \to A^+_W(Y_W)$ 
and set $J'' := \varinjlim_{(W,\overline{w})} J''_W \subset B$. }

{ \begin{lem} \label{lem : BJ is henselian}
\begin{enumerate} 
\item The pair $(B,J')$ is henselian. 
\item The pair $(B/J',J/J')$ is henselian.
\item The pair $(B,J)$ is henselian. 
\end{enumerate}
\end{lem} }
\begin{proof} 
{Let us verify part (1) of the lemma. 
By construction, $A^+_W(Y_W)$ is henselian along $J''_W$. 
Using \cite[Tag 0FWT]{stacks-project}, we deduce that
  $B$ is henselian along $J''$. We check that 
  the pair $(B/J'',J'/J'')$ is henselian. 
  Since filtered colimits are exact, 
  $$(B/J'',J'/J'') = (\varinjlim_{(W,\overline{w})} A^+_W(Y_W)/J''_W, \varinjlim_{(W,\overline{w})} J'_W/J''_W).$$
Since $A^+_W(Y_W)$ is the henselization of $A^+_W[e_1,\ldots,e_n]$ along the ideal 
$\langle A_W^{\circ \circ} \cup \{e_1,\ldots,e_n\} \rangle$, 
we get that 
\begin{equation} \label{equation : important simplification}
A^+_W(Y_W)/J''_W = A^+_W/A^{\circ \circ}_W.
\end{equation}
Furthermore, $J'_W/J''_W = \mathfrak{m}_{\mathcal{O}^+_{X_W}}(X_W)/A^{\circ \circ}_W$. 
Hence, 
$$(B/J'',J'/J'') = (\varinjlim_{(W,\overline{w})} A^+_W/A^{\circ \circ}_W, \varinjlim_{(W,\overline{w})} \mathfrak{m}_{\mathcal{O}^+_{X_W}}(X_W)/A^{\circ \circ}_W).$$
By Corollary \ref{cor:henspairlif} and \cite[Tag 0DYD and Tag 0FWT]{stacks-project},
we deduce that $(B/J'',J'/J'')$ is a henselian pair.
Since $(B,J'')$ is a henselian pair, applying \cite[Tag 0DYD]{stacks-project} 
once again implies that $(B,J')$ is henselian.

We now verify part (2) of the lemma. 
Since filtered colimits are exact, 
 \begin{equation} \label{trivial equality} 
 B/J' = \varinjlim_{(W,\overline{w})} A^+_W(Y_W)/J'_W.
 \end{equation} 
Similarly, 
 \[\varinjlim_{(W,\overline{w})} A^+_W/\varinjlim_{(W,\overline{w})}  \mathfrak{m}_{\mathcal{O}^+_{X_W}}(X_W) = \varinjlim_{(W,\overline{w})} A^+_W/\mathfrak{m}_{\mathcal{O}^+_{X_W}}(X_W).\]
We have the following commutative diagram.}

$$
\begin{tikzcd} [row sep = large, column sep = large] 
& \varinjlim_{(W,\overline{w})} A^+_W \arrow[d, "\varinjlim d_W"] \arrow[r] & B \arrow[d] \\
\varinjlim_{(W,\overline{w})} R_W \arrow[r, "\varinjlim e_W"] & \varinjlim_{(W,\overline{w})} A^+_W/\mathfrak{m}_{\mathcal{O}^+_{X_W}}(X_W) \arrow[r,equal] &
\varinjlim_{(W,\overline{w})}A^+_W(Y_W)/J'_W.
\end{tikzcd}
$$
{The equality in the above diagram is because
by construction,
$A^+_W(Y_W)/J'_W = A^+_W/\mathfrak{m}_{\mathcal{O}^+_{X_W}}(X_W).$

Observe that $\varinjlim_{(W,\overline{w})} R_W$ is a 
strictly henselian local ring whose maximal ideal is
$\mathfrak{p}_0 := \varinjlim_{(W,\overline{w})} \mathfrak{p}_W$. 
By Lemma \ref{lem:base change for tightness},
the morphism
$$e_W \colon R_W \to A^+_W/\mathfrak{m}_{\mathcal{O}^+_{X_W}}(X_W)$$ is integral
 and hence 
by the equality in the diagram, the induced map
$$R_W \to A^+_W(Y_W)/J'_W$$ 
 is integral as well.
By \cite[Tag 09XK]{stacks-project},
$\varinjlim_{(W,\overline{w})} A^+_W(Y_W)/J'_W$ is henselian along the 
ideal $\mathfrak{p}_0 \cdot \varinjlim_{(W,\overline{w})} A^+_W(Y_W)/J'_W = \varinjlim_{(W,\overline{w})} \mathfrak{p}_W \cdot A^+_W(Y_W)/J'_W $. 
Using that filtered colimits are exact, we deduce that                     
$$J/J' = \varinjlim_{(W,\overline{w})} J_W/J'_W.$$
A diagram chase shows that 
$J_W/J'_W = \mathfrak{p}_W \cdot A_W^+(Y_W)/J'_W.$
The equality (\ref{trivial equality}) now
implies that $B/J'$ is henselian along the ideal 
$J/J'$. 

 Part (3) of the lemma is a direct consequence of
Parts (1)-(2) and \cite[Tag 0DYD]{stacks-project}.}
\end{proof} 

Note that 
 \[\varinjlim_{(W,\overline{w})} A_W^{h(I_W)} = (\varinjlim_{(W,\overline{w})} {A_W^+}^{h(I_W)}) \otimes_{A ^+[e_1,\ldots,e_n]} A.\]
 and similarly, 
 \[\varinjlim_{(W,\overline{w})} A_W(Y_W) = (\varinjlim_{(W,\overline{w})} {A_W^+}(Y_W)) \otimes_{A ^+[e_1,\ldots,e_n]} A.\]
 Hence, to show that 
 $\varinjlim_{(W,\overline{w})} A_W^{h(I_W)} \simeq \varinjlim_{(W,\overline{w})} A_W(Y_W)$, it 
 suffices to verify that 
 $\varinjlim_{(W,\overline{w})} {A^+_W}^{h(I_W)} \simeq \varinjlim_{(W,\overline{w})} A^+_W(Y_W)$.

{
To complete the proof, 
we apply \cite[Remark 3.19]{CMM21}
with $S = \varinjlim_{(W,\overline{w})} A_W^+[e_1,\ldots,e_n]$, 
$I = \varinjlim_{(W,\overline{w})} I_W$ 
and $\tilde{S} = B$. Note that 
by \cite[Tag 0A04]{stacks-project},
$B$ is the henselization of 
$\varinjlim_{(W,\overline{w})} A^+_W[e_1,\ldots,e_n]$
along the ideal 
$\varinjlim_{(W,\overline{w})} \langle A_W^{\circ \circ} \cup \{e_1,\ldots,e_n\} \rangle$.
It follows that $B$ is a filtered colimit of \'etale $S$-algebras. 
Since $(B,J)$ is henselian by Lemma \ref{lem : BJ is henselian}(3), we can apply 
\cite[Remark 3.19]{CMM21}
to get that 
$(B,J)$ is the henselization of 
$(\varinjlim_{(W,\overline{w})} {A_W^+}[e_1,\ldots,e_n], \varinjlim_{(W,\overline{w})} I_W)$.
By  \cite[Tag 0A04]{stacks-project},
$\varinjlim_{(W,\overline{w})} {A^+_W}^{h(I_W)}$ is the henselization of 
$\varinjlim_{(W,\overline{w})} A^+_W[e_1,\ldots,e_n]$ along 
$\varinjlim_{(W,\overline{w})} I_W$. 
It follows that the natural 
morphism 
$B \to \varinjlim_{(W,\overline{w})} {A^+_W}^{h(I_W)}$ is an
isomorphism.  
%
Equations (\ref{LHS}), (\ref{RHS}) and the naturality of the isomorphism in 
\cite[Theorem 3.2.1]{hub96} allow us to conclude the proof. }


\end{proof}

{\begin{lem} \label{lem : reduce to finite cover of closed subsets}
 Let $\mathscr{W}$ be a cofiltered category. Let $\mathcal{A}$ be the 
 category of pseudo-adic spaces. Let 
 $p \colon \mathscr{W} \to \mathcal{A}$
 and 
 $q \colon \mathscr{W} \to \mathcal{A}$ 
 be functors such 
 for any $W \in \mathscr{W}$,
 $\underline{q(W)} = \underline{p(W)}$ and 
 $|q(W)| \subseteq |p(W)|$.
Suppose $\mathscr{W}$ has a final object $0$ and let $X_0 := p(0)$. 
Let $\sF_0$ be a sheaf of abelian groups on $X_{0,\text{\'et}}$
and for 
every $W \in \mathscr{W}$, let 
$\sF_W$ denote the pullback of $\sF_0$ to $p(W)$. 
Assume that we have a finite index set $\mathscr{I}$ such that 
$|X_0| = \bigcup_{j \in \mathscr{I}} U_j$ where 
for every $j \in \mathscr{I}$, 
 $U_j$ is a
subspace of $|X_0|$.  
For every $j \in \mathscr{I}$, let
$U_{j,W}$ be the preimage of $U_j$ for the morphism $|p(W)| \to |X_0|$.
Then 
the natural map 
 \[\varinjlim_{W \in \mathscr{W}} H^n(p(W),\sF_W) \to \varinjlim_{W \in \mathscr{W}} H^n(q(W),\sF_W)\]
  is an isomorphism if
  for every non-empty subset $J \subseteq \mathscr{I}$, 
  the natural morphism 
\begin{equation} \label{hypothesis of the lemma}  
\varinjlim_{W \in \mathscr{W}} H^n(\bigcap_{j \in J} \overline{U_{j,W}} ,\sF_W) \to \varinjlim_{W \in \mathscr{W}} H^n(\bigcap_{j \in J} \overline{U_{j,W}} \cap q(W), \sF_W).
\end{equation} 
is an isomorphism where 
by $\overline{U_{j,W}}$, we mean the closure of $U_{j,W}$ in $|p(W)|$. 
\end{lem} 
\begin{proof}
 
  By \cite[Corollary 2.6.10]{hub96}, we have a spectral sequence 
   \[\prod_{\alpha \in \mathscr{I}^{r+1}} H^s(\bigcap_{0 \leq i \leq r} \overline{U_{\alpha_i,W}},\sF_W) \implies H^{r+s}(p(W),\sF_W)\]
    where $\alpha = (\alpha_i)_{0 \leq i \leq r}$ and for every $i$, $\alpha_i \in \mathscr{I}$. 
    Note that $\mathscr{I}^{r+1}$ is a finite set.
    
     Since $\varinjlim_{W \in \mathscr{W}}$ is a filtered colimit,
    it is exact by \cite[Tag 04B0]{stacks-project}. 
    Note that finite products are coproducts in the category of abelian groups. 
    It follows that we have a spectral sequence 
    \[ \prod_{\alpha \in \mathscr{I}^{r+1}} \varinjlim_{W \in \mathscr{W}} 
    H^s( \bigcap_{0 \leq i \leq r} \overline{U_{\alpha_i,W}},\sF_W) 
    \implies \varinjlim_{W \in \mathscr{W}}  H^{r+s}(p(W),\sF_W).\]
   Similar arguments then show that we have a spectral sequence 
    \[\prod_{\alpha \in \mathscr{I}^{r+1}} \varinjlim_{W \in \mathscr{W}}
      H^s(\bigcap_{0 \leq i \leq r} \overline{U_{\alpha_i,W}} \cap q(W),\sF_W) \implies \varinjlim_{W \in \mathscr{W}}
    H^{r+s}(q(W),\sF_W).\]
    Note that the natural morphisms 
     \[\prod_{\alpha \in \mathscr{I}^{r+1}} \varinjlim_{W \in \mathscr{W}}
      H^s(\bigcap_{0 \leq i \leq r} \overline{U_{\alpha_i,W}},\sF_W)
   \to
   \prod_{\alpha \in \mathscr{I}^{r+1}} \varinjlim_{W \in \mathscr{W}} 
    H^s(\bigcap_{0 \leq i \leq r} \overline{U_{\alpha_i,W}} \cap q(W),\sF_W)\]
   are compatible with the differentials and hence they induce the natural map
   \[ \varinjlim_{W \in \mathscr{W}}  H^{r+s}(p(W),\sF_W) \to  \varinjlim_{W \in \mathscr{W}} H^{r+s}(q(W),\sF).\]
 The hypothesis (\ref{hypothesis of the lemma}) shows that this is
  sufficient to conclude the proof. 
\end{proof}}

\begin{proof} { (of Theorem \ref{proper base change})
{We can reduce to when 
  $S$ is an affine scheme. 
  Indeed, this is a consequence of 
  Lemma \ref{lem:smoobaschanrefpovi} and 
  Remark \ref{rem : certain base change for proper}.
 Let $\mathscr{J}$ denote the category of standard
 \'etale neighbourhoods of $(S,\overline{s})$
 and let $(W,\overline{w}) \in \mathscr{J}$.
  Note that the category $\mathscr{J}$ is cofiltered. 
 Observe that it suffices
 to prove (\ref{main situation equation}) when the colimit on the right runs 
 over $\mathscr{J}$.

 We proceed as in \cite[Proposition 4.4.3]{hub96}. 
 We first reduce to the case when $X$ is affinoid and $\mathcal{O}_X(X)$ is a Tate ring.
  Since $f$ is a quasi-compact morphism and 
  $S$ is quasi-compact, we get that $X$ is quasi-compact. 
  Let $\{U_i\}_{i \in I}$ be a finite covering of $X$ by affinoid open subsets
  such that for every $i \in I$, the morphism $U_i \to S$ is good. 
  By Remark \ref{rem:spmorlftrefde}, we
  can assume in addition that $\mathcal{O}_{U_i}(U_i)$ is a Tate ring.
 By Lemma \ref{lem:base change for goodness},
$\{U_{i,W} := U_i \times_S W\}_{i \in I}$ is a finite covering of $X_W$ by affinoid open subsets.
  
By Lemma \ref{lem : reduce to finite cover of closed subsets}, 
we deduce that 
to verify (\ref{main situation equation}),
 it suffices to prove that 
  for every non-empty subset $J \subseteq I$ 
  and every $q \in \mathbb{N}$, the natural morphism
  $$\varinjlim_{(W,\overline{w}) \in \mathscr{J}} H^q(({X_W},\bigcap_{j \in J} \overline{|U_{j,W}|}),\sF) \xrightarrow{} \varinjlim_{(W,\overline{w}) \in \mathscr{J}}  H^q(({X_W},\bigcap_{j \in J}\overline{|U_{j,W}|} \cap f_W^{-1}(w)),\sF)$$
  is an isomorphism.} 
Let $J \subseteq I$ be non-empty. 
  Let 
  $\{V_{i}\}_{i \in I_J}$ be a finite affinoid cover of 
  $\bigcap_{j \in J} U_j$.
 Hence, 
  $\{\overline{|V_{i,W}|}\}_{i \in I_J}$ is a finite cover 
  of 
  $\bigcap_{j \in J} \overline{|U_{j,W}|}$     
  because 
  $\overline{\bigcap_{j \in J} |U_{j,W}|} = \bigcap_{j \in J} \overline{|U_{j,W}|}$.
   This is a consequence of 
    \cite[Lemma 1.1.10(i)]{hub96}.    
  Then, as before,
  it suffices to show that for any $J' \subseteq I_J$,
  the natural morphism
   $$\varinjlim_{(W,\overline{w}) \in \mathscr{J}} H^q(({X_W},\bigcap_{j' \in J'} \overline{|V_{j',W}|}),\sF) \xrightarrow{} \varinjlim_{(W,\overline{w}) \in \mathscr{J}} H^q((X_W,\bigcap_{j' \in J'} \overline{|V_{j',W}|} \cap f_W^{-1}(w)),\sF)$$
   is an isomorphism.

      We choose an index $j_0 \in J$. 
   Let $J' \subseteq I_J$. 
  Observe that 
  for every $j' \in J'$, $V_{j'} \subseteq U_{j_0}$. 
  It follows that 
  $\bigcap_{j' \in J'} V_{j'}$ is an affinoid space.
  Furthermore, by Remark \ref{rem:spmorlftrefde}, 
  the composition $\bigcap_{j' \in J'} V_{j'} \to U_{j_0} \to S$ is good. 
  In summary, it suffices to prove that 
  the natural map
  $$\varinjlim_{(W,\overline{w}) \in \mathscr{J}} H^q( ({X_W},\overline{|U_W|}),\sF) \to \varinjlim_{(W,\overline{w}) \in \mathscr{J}} H^q(({X_W},\overline{|U_W|} \cap f_W^{-1}(s)),\sF)$$
  is an isomorphism
  where $U$ is an affinoid open subset of $X$ and $U \to S$ is good. 
  In addition, we have that $\mathcal{O}_U(U)$ is a Tate ring.  

 Observe that $f_{|U}$ is good. We are hence in the context of Situation 
 \ref{minor situation}. 
 By Lemma \ref{lem : reduce from closed subset to affinoid II},
 we can suppose $X$ is affinoid and 
     is the universal compactification of 
    $U \to S$ with $\mathcal{O}_U(U)$ Tate. By Lemma \ref{properties of universal compactification}(3),
    $f \colon X \to S$ is tight and 
    $\mathcal{O}_X(X) = \mathcal{O}_U(U)$ is Tate.

     Suppose $\sF \simeq \varinjlim_i \sF_i$ is a filtered colimit of torsion abelian sheaves.
     By Lemma \ref{push forward commutes with filtered colimits}
     and \cite[Expos\'e VI, Theorem 5.1]{SGA4},
     we deduce that 
     \[(R^nf_*\sF)_{\overline{s}} \simeq \varinjlim_i ((R^nf_*\sF_i)_{\overline{s}}).\]
     By Lemma \ref{lem:base change for goodness},
     $X_W$ is affinoid. 
     Since $w$ is a closed point, $f_W^{-1}(w)$ is a closed subspace of $X_W$. 
 Observe that
      $f_W^{-1}(w)$ is a closed subspace of a spectral space.
     It is hence quasi-separated and quasi-compact. 
      By \cite[Lemma 2.3.13(i)]{hub96},
      \[H^n(f_W^{-1}(w),\sF) \simeq \varinjlim_i H^n(f_W^{-1}(w),\sF_i).\]
      
      Just as in (8) of the proof of \cite[Proposition 4.4.3]{hub96},
      we can reduce to when $\sF = p_*G$
     where $p \colon Y \to X$  is a finite morphism of pseudo-adic spaces and
     $G$ is a constant sheaf
      of $\mathbb{Z}/m\mathbb{Z}$-modules of
     finite type for some $m \in \mathbb{N}$.
     
 { \begin{lem} 
  We can reduce to the case where
  there exists $e_0,\ldots,e_n \in \mathcal{O}_{\underline{Y}}(\underline{Y})$ 
  such that 
  $\mathcal{O}_{\underline{Y}}(\underline{Y}) = \sum_i e_i\mathcal{O}_{\underline{Y}}(\underline{Y})$ and 
  $|Y| = \{y \in \underline{Y} | |e_i(y)| < |e_0(y)| \mbox{ for } i = 1,\ldots,n\}$.   
 \end{lem}    
 \begin{proof} 
  Since the morphism $p$ is finite, $|Y|$ is closed in $\underline{Y}$. 
   Hence $|Y| = \bigcap_{i \in \mathscr{I}} C_i$, where 
   $C_i \subset \underline{Y}$ is a closed constructible set
   where 
   we use the notion of constructible set as introduced in 
   \cite[Chapitre 0, D\'efinition 2.3.10]{EGAI}. 
   We claim that we can reduce to when $|Y|$ is closed and constructible.
   Note that  
   Indeed, for every
   $(W,\overline{w}) \in \mathscr{J}$, let 
   $Y_W := Y \times_X X_W$ where we take the fibre product in the 
   category of pseudo-adic spaces. 
   We then have that $|Y_W| = \bigcap_{i \in \mathscr{I}} C_{i,W}$ where 
   for every $i \in \mathscr{I}$, 
   $C_{i,W}$ is the preimage of $C_i$ 
   in $\underline{Y_W}$ for the projection 
   $\underline{Y_W} \to \underline{Y}$. 
   Since $p$ is finite, we deduce that to show (\ref{main situation equation}), it 
   suffices to show that for every $q \in \mathbb{N}$,
    the natural morphism 
   $$\varinjlim_{(W,\overline{w}) \in \mathscr{J}} H^q((\underline{Y_W},\bigcap_{i \in \mathscr{I}} C_{i,W}),G) \xrightarrow{} \varinjlim_{(W,\overline{w}) \in \mathscr{J}}  H^q((\underline{Y_W}, \bigcap_{i \in \mathscr{I}} C_{i,W} \cap f_W^{-1}(w)),G)$$
   is an isomorphism. 
   Since colimits commute with colimits, we deduce from 
   \cite[Corollary 2.4.6]{hub96} that it suffices 
   to verify that for every finite subset $J \subseteq \mathscr{I}$ and 
   $q \in \mathbb{N}$, the natural morphism 
   $$\varinjlim_{(W,\overline{w}) \in \mathscr{J}} H^q((\underline{Y_W}, \bigcap_{j \in J} C_{j,W}),G) \xrightarrow{} \varinjlim_{(W,\overline{w}) \in \mathscr{J}}  
   H^q((\underline{Y_W}, \bigcap_{j \in J} C_{j,W} \cap f_W^{-1}(w)),G)$$ 
   is an isomorphism.
   We may hence suppose that $|Y|$ is closed and constructible. 
 By \cite[Lemma 3.1.10(ii)]{hub96} and Lemma 
 \ref{lem : reduce to finite cover of closed subsets}, 
 we can conclude a proof of the lemma. 
 \end{proof}}  
%
%
     Note that since $X$ is affinoid, $\underline{Y}$ is affinoid as well. 
     Let $Y' := \underline{Y}$. 
     Let $i$ denote the closed embedding $Y \hookrightarrow Y'$ and 
     set $\sG := i_*(G)$. Let $p'$ denote the finite morphism 
     $Y' \to X$. 
     
     We claim that it suffices to prove the theorem for the specialization 
     morphism $f \circ p' \colon Y' \to S$ and the sheaf $\sG$.
     Indeed, observe firstly
     that by \cite[Lemma 1.10.17(i)]{hub96}
      the morphisms $i$ and $p'$ are 
     proper.  Hence,
     \begin{equation} \label{another one bites the dust}
     \sF_{|X_W} \xrightarrow{\sim} p'_{W*}(\sG_{|Y'_W})
     \end{equation} 
     where $p'_W \colon Y'_W \to X_W$ is the morphism 
     induced by base change. 
     This implies an 
     isomorphism 
     \begin{align} \label{tedious}
     H^n(X_W,\sF_{|X_W})  \xrightarrow{\sim} H^n(Y'_W,\sG_{|Y'_W}).
     \end{align}
     On the other hand, we have a sequence of ismorphisms, 
     \begin{align} \label{dust again}
 \nonumber   \sF_{|f_W^{-1}(w)} &\overset{(i)} \to (p'_{W*}(\sG_{|Y'_W}))_{|f_W^{-1}(w)} \\
                                   &\overset{(ii)} \to p'_{0W*} \sG_{|(f_W \circ p'_W)^{-1}(w)}
     \end{align} 
     where $p'_{0W}$ denotes the finite morphism 
     $(f_W \circ p'_W)^{-1}(w) \to f_W^{-1}(w)$. 
     The isomorphism (i) is due to equation (\ref{another one bites the dust}).
     Note that $p'_W$ is finite and hence proper. 
     Hence, \cite[Theorem 4.1.2(b)]{hub96} implies the isomorphism (ii). 
     Equation (\ref{dust again}) gives an isomorphism 
     \begin{align} \label{tedious 2}
     H^n(f_W^{-1}(w),\sF) \xrightarrow{\sim} H^n((f_W \circ p'_W)^{-1}(w),\sG).
      \end{align}
One deduces the claim using equations (\ref{tedious}) and (\ref{tedious 2}).
 
     
      By Lemma \ref{lem : finite + tight = tight}, since $Y'$ is finite over $X$ 
      and $X$ is a universal compactification, 
      the composition $Y' \to X \to S$ is tight. 
      We now apply Proposition \ref{using prospecial} to conclude a proof of the theorem.}
\end{proof}

{\begin{cor}\label{proper base change to strict henselian local rings}
 Let $f \colon X \to S$ be a proper specialization morphism where $S$ is strictly local i.e.
 $S = \mathrm{Spec}(R)$ where $R$ is a strictly henselian local ring. 
 Let $s \in S$ denote the closed point and $X_s := f^{-1}(s)$ be the fibre over $s$. 
 Let $\sF$ be a torsion abelian sheaf.  
 We then have that,
 for every
 $n \in \mathbb{N}$,
 \begin{align*} 
  [R^nf_*(\sF)]_{{s}} \simeq H^n(X_{{s}},\sF_{|X_{{s}}}). 
 \end{align*}  
\end{cor} 
\begin{proof} 
This is a direct consequence of Theorem \ref{proper base change}.
\end{proof} }

{
\begin{cor} \label{cor : proper base change for complexes}
Let $f \colon X \to S$ be a proper specialization morphism.
 Let $s \in S$ be a closed point and $\overline{s} \to S$ be a geometric
  point over $s$. 
   Let $\sF$ be a complex in $\mathcal{D}^+(X_{\text{\'et}},\mathbb{Z})$ 
   whose cohomology sheaves are torsion. 
 For every
 $n \in \mathbb{Z}$, we have a natural isomorphism
  \begin{align*}
[R^nf_*(\sF)]_{\overline{s}} \to \varinjlim_{(W,\overline{w})} R^n\Gamma(f_W^{-1}(w),\sF)
\end{align*}
 where the colimit on the right runs over all \'etale neighbourhoods $(W,\overline{w})$ 
 of $(S,\overline{s})$, $w \in W$ is the image of $\overline{w}$,
 $X_W := X \times_S W$ 
  and
 $f_W \colon X_W \to W$ is the morphism induced by base change.
  \end{cor} 
\begin{proof} 
 Our proof is similar to that of \cite[Tag 0DDE]{stacks-project}.
 Let $(W,\overline{w})$ be an \'etale neighbourhood of $(S,\overline{s})$. 
Observe that we have spectral sequences 
\[R^pf_*H^q(\sF) \implies R^{p+q}f_*(\sF)\]
and 
\[H^p(f_W^{-1}(w),H^q(\sF)) \implies R^{p+q}\Gamma(f_W^{-1}(w),\sF).\]
 Since $\varinjlim_{(W,\overline{w})}$ is a filtered colimit,
    it is exact by \cite[Tag 04B0]{stacks-project}. 
Moreover,
the functor of taking the stalk at $\overline{s}$ 
is exact as well. 
It follows that we have spectral sequences 
\[[R^pf_*H^q(\sF)]_{\overline{s}} \implies [R^{p+q}f_*(\sF)]_{\overline{s}}\]
and 
\[\varinjlim_{(W,\overline{w})} H^p(f_W^{-1}(w),H^q(\sF)) \implies \varinjlim_{(W,\overline{w})} R^{p+q}\Gamma(f_W^{-1}(w),\sF).\]
We now apply Theorem \ref{proper base change} to conclude the proof. 
\end{proof} }

\section{An application}\label{sec:anapplication}

Recall from \S \ref{section : smooth base change} that we fixed a
torsion ring $A$. As an application of our work we prove the following theorem concerning type (S) formal schemes (cf. \cite[\S 1.9]{hub96}). We remind the reader that given a type (S) formal scheme $\mathfrak{X}$, there is the associated reduced scheme $\mathfrak{X}_s$ and the analytic adic generic fiber $\mathfrak{X}_{\eta}$ (cf. the paragraph immediately preceding Definition \ref{defi:spemorrepov}). Moreover there is the specialization morphism $\lambda_{\mathfrak{X}} \colon \mathfrak{X}_{\eta} \to \mathfrak{X}_s$ and by Lemma \ref{lem:comsquspemoareneacu}, there is no ambiguity whether one considers the pushforward along \eqref{eq:thepsisneacu} or  \eqref{eq:nearbycuspe}.

\begin{thm} \label{vanishing cycles commutes with lower shriek ref}
 Let $\mathfrak{f} \colon \mathfrak{X} \to \mathfrak{Y}$ be a finite type and separated morphism of type (S) formal schemes. Suppose in addition that $\mathfrak{Y}$ is quasi-compact and quasi-separated. Then there is a canonical equivalence 
 \begin{align*} 
    R^{+}\mathfrak{f}_{s!} \circ R^{+}\lambda_{\mathfrak{X}*} \xrightarrow{\sim} R^{+}\lambda_{\mathfrak{Y}*} \circ R^{+}\mathfrak{f}_{\eta!}
 \end{align*}
 of functors $D^{+}(\mathfrak{X}_{\eta}, A) \to D^+(\mathfrak{Y}_s, A)$.
\end{thm} 

The idea for proving Theorem \ref{vanishing cycles commutes with lower shriek ref} is simple. One takes the compactification of the morphism $\mathfrak{f}_{\eta}$ in the sense of \cite[Definition 5.1.8]{hub96} and reduces Theorem \ref{vanishing cycles commutes with lower shriek ref} to the case of an open immersion, where the result is known by Corollary 3.5.11(ii) in loc.cit. However the compactification of $\mathfrak{f}_{\eta}$ may no longer be of finite type (in particular may not possess a formal model). Thus the strategy is use the enlarged class of specialization morphisms we have constructed and their compactifications (cf. Theorem \ref{thm:compspemor}). First we need some preparation. 

\begin{sit} \label{situation : application of proper base change}  
\emph{Let}
$$
\begin{tikzcd} [row sep = large, column sep = large] 
X \arrow[r, "j_1", hookrightarrow] \arrow[d, "\alpha"] &
Y \arrow[d, "\beta"] \\
S \arrow[r, "j_2", hookrightarrow] & T  
\end{tikzcd}
$$ 
\emph{be a commutative diagram where $\alpha$ and $\beta$ are proper specialization morphisms, and $j_1$ and $j_2$ are open embeddings.}
\end{sit}

\begin{lem} \label{lem : simple equation}
We work within the context of Situation \ref{situation : application of proper base change}.
There is a natural equivalence
\[j_{2!}j^*_{2}R^+\beta_*j_{1!} \to R^+\beta_*j_{1!}.\]
\end{lem} 
\begin{proof} 
 By adjointness, we have a natural morphism 
 \[\Phi \colon j_{2!}j^*_{2}R^+\beta_*j_{1!} \to R^+\beta_*j_{1!}.\]
 Observe that it suffices to show that 
 if 
 $\sF \in \mathcal{D}^+(X_{\text{\'et}},A)$, $t \in T \smallsetminus S$
 and $n \in \mathbb{Z}$ 
  then 
 $[R^n\beta_*j_{1!}(\sF)]_{\overline{t}} = 0$.
 By 
 Corollary \ref{cor : proper base change for complexes}, for every
 $n \in \mathbb{Z}$, we have an isomorphism
  \begin{equation}\label{main theorem equation}
[R^n\beta_*j_{1!}(\sF)]_{\overline{t}} \xrightarrow{\sim} \varinjlim_{(W,\overline{w})} R^n\Gamma(\beta_W^{-1}(w),j_{1!}(\sF))
\end{equation}
 where the colimit on the right runs over all \'etale neighbourhoods $(W,\overline{w})$ 
 of $(T,\overline{t})$.  
 Let $(W,\overline{w})$ be an \'etale neighbourhood of $(T,\overline{t})$ and let 
 $X_W$ be the preimage of $X$ for the morphism 
 $Y_W \to Y$. If $j_{1W}$ denotes the open embedding 
 $X_W \hookrightarrow Y_W$ then 
 by \cite[Theorem 5.2.2(iv)]{hub96}, 
 we have an isomorphism
 $$(j_{1!}(\sF))_{|Y_W} \xrightarrow{\sim} j_{1W!}(\sF_{|X_W}).$$  
 Since $\beta_W^{-1}(w) \cap X_W = \emptyset$, we see that 
 $R^n\Gamma(\beta_W^{-1}(w),j_{1!}(\sF)) = 0$. This concludes the proof.   
\end{proof}

\begin{cor} \label{lem:baschanspecmoropenr}
We work within the context of 
Situation \ref{situation : application of proper base change}. 
\begin{enumerate} 
\item Let $I$ be an injective $A$-module on $X_{\text{\'et}}$. Then, for every $n \geq 1$,
$$ R^n\beta_*j_{1!}(I) = 0.$$
\item There is a natural equivalence 
\[
R^{+}\beta_{*} \circ j_{1!} \xrightarrow{\sim} j_{2!} \circ R^{+}\alpha_{*}.
\]
\end{enumerate}
\end{cor}

\begin{proof}
{We verify part (1) as follows. Let $Z$ be the fibre product $Y \times_{T} S$, which exists by Proposition \ref{prop:exifibprodspecmor}. Consider the commutative diagram coming from the universal property of fibre products}
$$
\begin{tikzcd} [row sep = large, column sep = large] 
X \arrow[rr, "j_1", hookrightarrow, bend left] \arrow[r, "j_3"] \arrow[d, "\alpha"] 
& Z \arrow[r, "j_4"] \arrow[ld, "\gamma"] 
& Y \arrow[d, "\beta"] \\
S \arrow[rr, "j_2", hookrightarrow] && T.  
\end{tikzcd}
$$
{Note that since $j_1$ and $j_4$ are étale, so is $j_3$ by \cite[Proposition 1.6.7(iii)]{hub96}.
By Remark 
\ref{rem : certain base change for proper}, since
$\beta$ is proper, $\gamma$ is proper. 
Since $\alpha$ is proper as well, it follows 
that $j_3$ is both quasi-compact and separated (by similar arguments appearing in the proof of 
Lemma \ref{lem:indlowshrirepovi}).

Then
\begin{align} \label{equn : salt lake city}
\nonumber j_{2}^{*} \circ R^{+}\beta_{*} \circ j_{1!} &\overset{(i)}{=} R^{+}\gamma_{*} \circ j_{4}^{*} \circ j_{1!} \\
\nonumber &\overset{(ii)}{=} R^{+}\gamma_{*} \circ  j_{3!} \\
&\overset{(iii)}{=} R^{+}\alpha_{*}
\end{align}
where (i) follows from Lemma \ref{lem:smoobaschanrefpovi}, (ii) follows from $j_1 = j_4 \circ j_3$ and (iii)
 follows from Lemma \ref{lem:lowshrcomprefpovi}
 after noting that $j_3$ is quasi-compact because 
 $\alpha$ is quasi-compact and $\gamma$ is quasi-separated.
 Hence, we get that for every $n \geq 1$,
$$j_2^*R^n\beta_*j_{1!}(I) = R^{n }\alpha_{*}(I) = 0.$$ 
It remains to show that $R^{n}\beta_{*} j_{1!}(I)$ vanishes outside of $S$. 
This is an immediate consequence of Lemma \ref{lem : simple equation}.
Hence, $R^n\beta_*j_{1!}(I) = 0$ everywhere.

We now verify part (2). 
  Recall that we have a natural morphism 
  $j_{!} \to j_{*}$ which implies a natural map 
  $\beta_*j_{1!} \to \beta_*j_{1*}$. 
  Since $\beta \circ j_1 = j_2 \circ \alpha$, 
   $\beta_*j_{1*} = j_{2*}\alpha_*$ and hence we have a natural 
   morphism 
   \[\beta_*j_{1!} \to j_{2*}\alpha_*.\]
   Applying $j_{2!}j^*_{2}$, we get a morphism
   \[j_{2!}j^*_{2}\beta_*j_{1!} \to j_{2!}j^*_{2} j_{2*}\alpha_*.\] 
   By Lemma \ref{lem : simple equation}, 
   we get a natural morphism
   \[\epsilon \colon \beta_*j_{1!}  \to j_{2!}\alpha_*.\]
   We claim that 
   $\epsilon$ is an equivalence. 
   Indeed, observe that 
   Equation (\ref{equn : salt lake city}) shows that 
   this is true over $S$ while
   Lemma \ref{lem : simple equation} shows that this true outside of $S$.  
   Part (1) of the lemma implies that 
   $R^+(\beta_*j_{1!}) = R^+\beta_*j_{1!}$.
   Since $\epsilon$ is a natural equivalence, we have that
   the induced morphism 
   \[R^+\beta_*j_{1!} \to  j_{2!}R^+\alpha_*\]
   is an isomorphism.   
   }

\end{proof}

\begin{cor} \label{lem:lowshricomutwicomrepov}
Let $\alpha \colon X \to S$ be a separated and finite type specialization morphism. Let $g \colon S \to T$ {be a separated and
finite type morphism of schemes}. Then $g \circ \alpha$ is a separated and finite type specialization morphism. Moreover, if 
$T$ is quasi-compact and quasi-separated,
 there is a natural equivalence 
\[
R^{+}(g \circ \alpha)_{!} \simeq R^{+}g_{!} \circ R^{+}\alpha_{!}
\]
of functors $\mathcal{D}^+(X_{\text{ét}},A) \to \mathcal{D}^+(T_{\text{ét}},A)$.
\end{cor}

\begin{proof}
The strategy is very similar to the proof of Lemma \ref{lem:lowshrcomprefpovi}. The fact that the composition $g \circ \alpha$ is a separated and finite type specialization morphism follows from Proposition \ref{prop:commorinvaprofrrep}. For the second part consider the following diagram
$$
\begin{tikzcd} [row sep = large, column sep = large] 
X \arrow[r, "j_1", hookrightarrow] \arrow[rd, "\alpha"] &
\overline{X}^{/S} \arrow[d, "\beta_1"] \arrow[r, "j_2", hookrightarrow] &
Z \arrow[d, "\beta_2"] \\
& S \arrow[r, "j_3", hookrightarrow] \arrow[rd, "g"] &
U \arrow[d, "h"] \\
&& T
\end{tikzcd}
$$
where the top left triangle is the universal compactification of $\alpha$ via Theorem \ref{thm:compspemor}, the bottom right triangle is a compactification of 
$g$ via \cite[Theorem 4.1]{Co07}. Let us 
explain the top right square. The 
specialization morphism $j_3 \circ \beta_1 \colon \overline{X}^{/S} \to U$ is separated and of finite type by Proposition \ref{prop:commorinvaprofrrep}. Thus again by Theorem \ref{thm:compspemor}, it admits a universal compactification which we denote by the triple $(Z, \beta_2, j_2)$. In particular all the horizontal arrows (in the diagram) are quasi-compact open embeddings and the vertical arrows are proper. Therefore the composition $j_2 \circ j_1$ is a quasi-compact open embedding of analytic adic spaces and by Proposition \ref{prop:commorinvaprofrrep}, the composition $h \circ \beta_2$ is a proper specialization morphism. Thus the outer triangle is a compactification of $g \circ \alpha$.

Then 
$$R^{+}g_{!} \circ R^{+}\alpha_{!} = R^{+}h_{*}\circ j_{3!}\circ R^{+}\beta_{1*} \circ j_{1!}$$
and 
\begin{align*}
R^{+}(g \circ \alpha)_{!} &\overset{(i)}{=} R^{+}(h \circ \beta_2)_{*} \circ (j_2 \circ j_1)_{!} \\
&\overset{(ii)}{=} R^{+}h_{*} \circ R^{+}\beta_{2*} \circ j_{2!} \circ j_{1!} \\
&\overset{(iii)}{=} R^{+}h_{*} \circ j_{3!} \circ R^{+}\beta_{1*} \circ j_{1!}
\end{align*}
where (i) follows from Lemma \ref{lem:indlowshrirepovi}, (ii) follows from Proposition \ref{prop:compmorrepovi} and \cite[Theorem 5.4.3]{hub96}, and (iii) follows from 
Corollary \ref{lem:baschanspecmoropenr}. This completes the proof.
\end{proof}

We now arrive at the promised application of the theory developed thus far.

\begin{proof} (of Theorem \ref{vanishing cycles commutes with lower shriek ref})
We have a commutative diagram
$$
\begin{tikzcd} [row sep = large, column sep = large] 
\mathfrak{X}_{\eta} \arrow[r, "\lambda_{\mathfrak{X}}"] \arrow[d, "\mathfrak{f}_{\eta}"] &
\mathfrak{X}_s \arrow[d, "\mathfrak{f}_s"] \\
\mathfrak{Y}_{\eta} \arrow[r, "\lambda_{\mathfrak{Y}}"] &
\mathfrak{Y}_s  
\end{tikzcd}
$$
where we remind the reader that $\lambda_{\mathfrak{X}}$ (resp. $\lambda_{\mathfrak{Y}}$) are the specialization morphisms, in the sense of Definition \ref{defi:spemorrepov} induced by the morphisms of locally ringed spaces $(\mathfrak{X}_{\eta}, \mathcal{O}_{\mathfrak{X}_{\eta}}^+ ) \to (\mathfrak{X}, \mathcal{O}_{\mathfrak{X}})$ (resp. $(\mathfrak{Y}_{\eta}, \mathcal{O}_{\mathfrak{Y}_{\eta}}^+ ) \to (\mathfrak{Y}, \mathcal{O}_{\mathfrak{Y}})$), cf. \S \ref{apend:numner1}. The composition $\mathfrak{f}_s \circ \lambda_{\mathfrak{X}} = \lambda_{\mathfrak{Y}} \circ \mathfrak{f}_{\eta}$ is a separated and finite type specialization morphism (because both $\lambda_{\mathfrak{X}}$ and $ \lambda_{\mathfrak{Y}}$ are proper by Proposition \ref{prop:specmorarproprepov}, and both $\mathfrak{f}_{s}$ and $\mathfrak{f}_{\eta}$ are separated and of finite type by assumption on $\mathfrak{f}$). Thus by Theorem \ref{thm:compspemor} it admits a (universal) compactification. By Definition \ref{defi:lowershrfurepovi}/Lemma \ref{lem:indlowshrirepovi}, we have a well defined $R^{+}(\mathfrak{f}_s \circ \lambda_{\mathfrak{X}})_{!} = R^{+}(\lambda_{\mathfrak{Y}} \circ \mathfrak{f}_{\eta})_{!}$ right derived lower shriek functor. We compute
\begin{align*}
R^{+}(\mathfrak{f}_s \circ \lambda_{\mathfrak{X}})_{!} &\overset{(i)}{=} R^{+}\mathfrak{f}_{s!} \circ R^{+}\lambda_{\mathfrak{X}!} \\
&\overset{(ii)}{=} R^{+}\mathfrak{f}_{s!} \circ R^{+}\lambda_{\mathfrak{X}*} 
\end{align*}
where (i) follows from Corollary \ref{lem:lowshricomutwicomrepov}, (ii) follows from the fact that $\lambda_{\mathfrak{X}}$ is proper (cf. Proposition \ref{prop:specmorarproprepov}).

Similarly one obtains $R^{+}(\lambda_{\mathfrak{Y}} \circ \mathfrak{f}_{\eta})_{!} = R^{+}\lambda_{\mathfrak{Y}*} \circ R^{+}\mathfrak{f}_{\eta!}$ where one uses Lemma \ref{lem:lowshrcomprefpovi} in place of Corollary \ref{lem:lowshricomutwicomrepov}. The result now follows.
\end{proof} 

We end this section with the following remark.
\begin{rem}
\emph{Let $\alpha \colon X \to S$ be a separated, taut and locally of finite type specialization morphism. Then as in Definition \ref{defi:lowershrfurepovi}, one can define a functor $R^{+}\alpha_{!}$ with the expected properties. Moreover if $X$ satisfies some finiteness conditions (e.g. $X$ is locally of finite type over an algebraically closed valued field $k$), then the functor
\[
R^{+}\alpha_{!} \colon \mathcal{D}^+(X_{\text{ét}},A) \to \mathcal{D}^+(S_{\text{ét}},A)
\]
admits a right adjoint functor
\[
R^{+}\alpha^{!} \colon \mathcal{D}^+(S_{\text{ét}},A) \to \mathcal{D}^+(X_{\text{ét}},A).
\]
This places \cite[Corollary 4.3(1)]{ildarjohn}  into a larger framework.}
\end{rem}

  \bibliographystyle{plain}
\bibliography{library}
   
\noindent Ildar Gaisin\\
 Graduate School of Mathematical Sciences, The University of Tokyo,\\
 3-8-1 Komaba, Meguro, Tokyo, 153-0041, Japan. \\
\textit{email : ildar@ms.u-tokyo.ac.jp} \\
  
\noindent  John Welliaveetil \\
 Kavli Institute for the Physics and Mathematics of the Universe,
 The University of Tokyo,\\
 5-1-5 Kashiwanoha
Kashiwa, 277-8583, Japan\\
 \textit{email : welliaveetil@gmail.com}

   \end{document}